\documentclass[11pt]{amsart}
\usepackage{latexsym, amsmath, amssymb,amsthm,amsopn,amsfonts,mathrsfs}
\usepackage{stmaryrd}
\usepackage{version}
\usepackage{epsfig,graphics,color,graphicx,graphpap,dsfont,eufrak}
\usepackage{amssymb}

\usepackage[makeroom]{cancel}
\allowdisplaybreaks

\ifx\pdfoutput\undefined
  \DeclareGraphicsExtensions{.eps}
\else
  \ifx\pdfoutput\relax
    \DeclareGraphicsExtensions{.eps}
  \else
    \ifnum\pdfoutput>0
      \DeclareGraphicsExtensions{.pdf}
    \else
      \DeclareGraphicsExtensions{.eps}
    \fi
  \fi
\fi

\newcommand{\supp}{\operatorname{supp}}

\newcommand{\LB}{\left[}
\newcommand{\RB}{\right]}

\newcommand{\vertiii}[1]{{\left\vert\kern-0.25ex\left\vert\kern-0.25ex\left\vert #1 
    \right\vert\kern-0.25ex\right\vert\kern-0.25ex\right\vert}}

\newcommand{\be}{\begin{equation}}
\newcommand{\ee}{\end{equation}}

\newcommand{\bse}{\begin{subequations}}
\newcommand{\ese}{\end{subequations}}

\newcommand{\jump}[1]{\left\llbracket{#1}\right\rrbracket}

\newcommand{\realpart}[1]{\operatorname{Re}{#1}}

 \newcommand{\sgn}{\operatorname{sgn}}

\theoremstyle{plain}

\newtheorem{thm}{Theorem}[section]
\newtheorem{prop}{Proposition}[section]

\newtheorem{lemma}[prop]{Lemma}

\theoremstyle{definition}

\newtheorem{defn}[prop]{Definition}
\newtheorem{problem}[prop]{Problem}
\newtheorem{remark}{Remark}[section]

\numberwithin{equation}{section}

\def\squarebox#1{\hbox to #1{\hfill\vbox to #1{\vfill}}}

\usepackage{amsxtra}

\DeclareFontFamily{OT1}{pzc}{}
\DeclareFontShape{OT1}{pzc}{m}{it}{<-> s * [1.10] pzcmi7t}{}
\DeclareMathAlphabet{\mathpzc}{OT1}{pzc}{m}{it}

\def\ba{\begin{array}}
\def\ea{\end{array}}
\def\bea{\begin{eqnarray}}
\def\eea{\end{eqnarray}}
\def\beas{\begin{eqnarray*}}
\def\eeas{\end{eqnarray*}}
\def\bi{\begin{itemize}}
\def\ei{\end{itemize}}

\def\({\textnormal{(}}
\def\){\textnormal{)}}

\def\b1{{\bf 1}}

\usepackage{hyperref}

\title[Continuous dependence on density]{Continuous dependence on the density for stratified steady water waves}

\author[R.M. Chen]
{Robin Ming Chen}
\address{Robin Ming Chen\newline
Department of Mathematics\\
University of Pittsburgh\\
Pittsburgh, PA 15260} \email{mingchen@pitt.edu}
\author[S. Walsh]
{Samuel Walsh}
\address{Samuel Walsh \newline
Department of Mathematics \\ 
University of Missouri\\
Columbia, MO 65211}
\email{walshsa@missouri.edu}

\thanks{The work of R.M. Chen was partially supported by the NSF grant DMS-0908663.}

\begin{document}

\begin{abstract}
There are two distinct regimes commonly used to model traveling waves in stratified water:  {continuous} stratification, where the density is smooth throughout the fluid, and layer-wise continuous stratification, where the fluid consists of multiple immiscible strata.   The former is the more physically accurate description, but the latter is frequently more amenable to analysis and computation.   By the conservation of mass, the density is constant along the streamlines of the flow; the stratification can therefore be specified by prescribing the value of the density on each streamline.  We call this the streamline density function.  

Our main result states that, for every smoothly stratified periodic traveling wave in a certain small-amplitude regime, there is an $L^\infty$ neighborhood of its streamline density function such that, for any piecewise smooth streamline density function in that neighborhood, there is a corresponding traveling wave solution.  Moreover, the mapping from streamline density function to wave is Lipschitz continuous in a certain function space framework.  As this neighborhood includes piecewise smooth densities with arbitrarily many jump discontinues, this theorem provides a rigorous justification for the ubiquitous practice of approximating a smoothly stratified wave by a layered one.  We also discuss some applications of this result to the study of the qualitative features of such waves.   \end{abstract}
\maketitle

\section{Introduction} \label{intro section}

We are interested in studying two-dimensional traveling periodic water waves with heterogeneous density.  These are waves of permanent configuration that evolve simply by translating with a constant velocity.    Shifting to a moving reference eliminates time dependence from the system.  The wave can then be said to inhabit a steady fluid region $\Omega \subset \mathbb{R}^2$.   Throughout this work, we assume that $\Omega$ lies above a flat impermeable ocean bed, and below the graph of an \emph{a priori} unknown surface profile $\eta$:
$$ \Omega = \{ (x,y) \in \mathbb{R}^2 : -d < y < \eta(x) \}.$$
Here the axes are fixed so that the wave propagates in the positive $x$-direction with speed $c > 0$, and  the ocean depth is $d > 0$.   The flow is described mathematically by a velocity field $(u,v) : \Omega \to \mathbb{R}^2$, a pressure $P: \Omega \to \mathbb{R}$, and a density $\varrho: \Omega \to \mathbb{R}_+$.    Periodicity of the wave means that $u$, $v$, $P$, $\varrho$, and $\eta$ are $2L$-periodic in $x$.   For $(u,v, \varrho, P, \eta)$ to represent a water wave, they must satisfy the free boundary steady Euler equations (see \S\ref{intro formulation section}).  

Density stratification is an important feature of waves in the ocean, with many dynamical implications.   It  arises from salinity, temperature gradients due to heating from the sun, or the presence of pollutants, for example.  Ocean waves typically have large regions of nearly constant density separated by thin transition layers, the \emph{pycnoclines}, where the density may vary sharply.  For this reason, it is a very common practice to imagine these waves as consisting of two or more immiscible layers.  The density in each layer is assumed to be smooth --- often just constant --- and a jump discontinuity is permitted over the interfaces.    Doing so effectively collapses the pycnoclines to material lines.

With that in mind, we identify two distinct regimes.  A wave is said to be \emph{continuously stratified} provided that $\varrho$ is continuous throughout the entire fluid domain $\Omega$.  On the other hand, we say that $\varrho$ is \emph{layer-wise smooth} if $\Omega$ can be partitioned into finitely many immiscible fluid regions 
$$ \overline{\Omega} = \bigcup_{i=1}^N \overline{\Omega_i},$$
where each $\Omega_i \subset \Omega$ is an open set with smooth boundary, and the restriction $\varrho|_{\Omega_i}$ is smooth (the precise regularity of both $\partial \Omega_i$ and $\varrho$ will be specified shortly).    

The continuously stratified case is arguably more physically accurate, but the layered model can be an extremely convenient idealization in certain situations.  This is especially true when $\varrho$ is layer-wise constant, as it allows one to assume that the velocity field is irrotational in each fluid region (this is generally impossible with heterogeneous density).  Irrotational waves are considerably simpler to study, both analytically and computationally.   Indeed, our current understanding of the qualitative properties of steady waves with vorticity is comparatively quite primitive.  

The central objective of this work is to quantify the degree to which a continuously stratified water wave can be approximated by a merely layer-wise smooth wave.  We show that, in a certain small-amplitude regime, the wave depends \emph{continuously} on the stratification.  That is, if one fixes a continuously stratified wave of this type, there exists nearby many-layered traveling waves that converge to the smooth wave as the number of layers is taken to infinity.  In fact, these layer-wise smooth waves are parameterized by the density in a Lipschitz continuous fashion. 

This serves as a  rigorous justification for the layered model, albeit in a specific physical regime.  Furthermore, it provides a promising new avenue for studying a variety of qualitative features of continuously stratified waves.  One specific application, which we pursue in an accompanying paper, is the problem of surface reconstruction from pressure data on the ocean bed; see the discussion in \S\ref{intro statement of results section}. 

\subsection{Eulerian formulation of the problem} \label{intro formulation section}
Now that we have established the overarching goal of the paper, let us formulate things more carefully.  Suppose that we have a layer-wise smooth density (a continuously stratified density we view as the special case where there is a single fluid layer).  We say that $(u,v, \varrho, P, \eta)$ represents a steady water wave provided it satisfies the steady Euler equations that we now detail.  For reasons that will become clear, we work in the weak setting where everything should be interpreted in the distributional sense.  

First, in each layer, we require that the velocity field be divergence free 
\bse  \label{weakeuler} \be u_x + v_y =  0, \qquad \textrm{in } \Omega_i. \label{weakvolume}  \ee
In fluid mechanics, this is referred to as \emph{incompressibility}; it is typical of flows in the ocean.  We also assume that the density of each fluid particle is invariant under the flow, and that momentum is conserved.  The weak formulation of these statements amounts to the following:
\begin{align}
((u-c) \varrho )_x + (v \varrho)_y & =  0, \qquad \textrm{in } \Omega_i, \label{weakmass} \\
-c(\varrho u)_x + (\varrho u^2)_x + (\varrho uv)_y  & =  -P_x , \qquad \textrm{in } \Omega_i, \label{weakmomentumx} \\
-c(\varrho v)_x + (\varrho uv)_x + (\varrho v^2)_y & =  -P_y - g \varrho, \qquad \textrm{in } \Omega_i. \label{weakmomentumy} \end{align}  \ese

We assume that the density is strictly positive,
\be \varrho > 0, \qquad \textrm{in } \overline{\Omega},\label{varrhopositive}\ee
and that the fluid is \emph{stably stratified}:
\be y \mapsto \varrho(\cdot,y) \textrm{ is non-increasing.} \label{stablestratificationeuler} \ee
This simply says that, as one expects, the density increases with depth.

The kinematic and dynamic boundary conditions are
\bse \label{weakeulerboundary} \begin{align} 
v & =  (u-c) \eta_x, \qquad \textrm{on } \{y = \eta(x)\}, \label{weakkinematicsurface} \\
v & =  0, \qquad \textrm{on } \{y = -d\}, \label{weakkinematicbed} \\
P & =  \displaystyle P_{\textrm{atm}}, \qquad \textrm{on } \{y = \eta(x)\}. \label{weakdynamic} \end{align} 
Notice that \eqref{weakkinematicsurface}  simply states that the air--sea interface is a material line:  at each point  {$(x,\eta(x))$}, the normal velocity of the interface matches the normal velocity of the water.    $P_{\textrm{atm}}$ is the atmospheric pressure, which we take to be a given constant.  Then the dynamic condition \eqref{weakdynamic} enforces the continuity of the pressure across the air--sea interface.  

Analogous conditions are imposed on the interfaces between interior layers.  For simplicity we assume that each layer has graph geometry.  Thus, 
\be \Omega_i := \{ (x, y) \in \Omega : \eta_{i-1}(x) < y < \eta_i(x) \}, \qquad i = 1, \ldots, N, \label{def Omega_i} \ee
 for some functions $\eta_0, \ldots, \eta_N$, with $\eta_0 := -d$, $\eta_N := \eta$.  Implicit here is the convention that $\Omega_i$ lies beneath $\Omega_{i+1}$, for $i = 1, \ldots N-1$.  In particular, $\Omega_1$ is the layer directly above the ocean bed, while $\Omega_N$ lies right below the air--sea interface.   As in \eqref{weakkinematicsurface}, we require that 
\be v  =  (u-c) \partial_x \eta_i, \qquad \textrm{on } \{y = \eta_i(x)\}. \label{weakkinematicinterface} \ee \ese
This is equivalent to the immiscibility of the layers.  Similarly, we mandate that
\be P \textrm{ is continuous in } \overline{\Omega}. \label{pressure continuous} \ee

Finally, we make the important assumption that there is no horizontal stagnation in the flow:
\be u-c < 0 \qquad \textrm{in } \overline{\Omega} \label{nostagnation}. \ee
As can be seen above, points where $u = c$ lead to degeneracy in the governing equations.  One consequence of the absence of stagnation points is that the streamlines for the flow cannot be closed.  We exploit this fact later when we employ the Dubreil--Jacotin transformation in \S\ref{height eq formulation section}.

In total, we arrive at the following descriptions for the Euler problem.  Fix $\alpha \in (0,1)$ and put $r := 2/(1-\alpha)$.  

\begin{problem}[Steady weak Euler problem] \label{weak Euler prob}  Find velocity field $(u,v)$, density $\varrho$, pressure $P$, and interfaces $\eta$, $\eta_1, \ldots, \eta_{N-1}$ with the following regularity
\begin{align*}
u,\, v, \, \varrho  &\in L^r_{\textrm{per}}(\Omega) \cap W^{1,r}_{\textrm{per}}({\Omega_1}) \cap \cdots \cap W^{1,r}_{\textrm{per}}({\Omega_N}), \qquad
%\varrho & \in L^r_{\textrm{per}}(\Omega) \cap W^{1,r}_{\textrm{per}}({\Omega_1}) \cap \cdots \cap W^{1,r}_{\textrm{per}}({\Omega_N})  \\
P  \in W^{1,r}_{\textrm{per}}({\Omega}), \qquad 
\eta, \,\eta_i  \in C_{\textrm{per}}^{1,\alpha}(\mathbb{R}), \end{align*}
 that satisfy \eqref{weakeuler}, \eqref{varrhopositive}--\eqref{stablestratificationeuler}, \eqref{pressure continuous},  \eqref{weakeulerboundary}, and have no horizontal stagnation \eqref{nostagnation}.
\end{problem}
Here the subscript ``per'' indicates $2L$-periodicity in the $x$-direction.  Note also that, by  Morrey's inequality,  $P \in W_{\textrm{per}}^{1,r}(\Omega) \subset C_{\textrm{per}}^{0,\alpha}(\overline{\Omega})$; hence the continuity of the pressure across the internal interfaces \eqref{pressure continuous} is encoded in the choice of function spaces.  

 For classical solutions of Problem \ref{weak Euler prob}, conservation of mass and incompressibility \eqref{weakvolume}--\eqref{weakmass} ensures that we may define a function $\psi = \psi(x,y)$ by 
\be \psi_x = -\sqrt{\varrho}v,\qquad \psi_y = \sqrt{\varrho} (u-c), \qquad \textrm{in } \bigcup_i \Omega_i. \label{defpsi} \ee
The same holds true in the weak setting, but this fact is not immediately obvious; we prove it in Lemma \ref{chainrulelemma}.  $\psi$ is called the pseudo (relative) stream function, though we will simply refer to it as the \emph{stream function}.  The factor of $\sqrt{\rho}$ is an innovation due to Yih \cite{yih1965dynamics}, its utility will become apparent later in \S\ref{stream function formulation section}.

From the definition \eqref{defpsi} and \eqref{nostagnation}, we see that the no stagnation condition takes the form  
\be \psi_y < 0, \qquad \textrm{in } \bigcup_i{\Omega_i}. \label{nostagnationpsi} \ee
The level sets of $\psi$, called the \emph{streamlines}, capture a great deal of information about the flow.  In particular, observe that  \eqref{weakkinematicsurface}, \eqref{weakkinematicbed}, and \eqref{weakkinematicinterface} state precisely that the free surface, internal interfaces, and ocean bed are each streamlines.  Since \eqref{defpsi} only determines $\psi$ up to a constant in each $\Omega_i$, we may take $\psi$ to be continuous in $\overline{\Omega}$, and, without loss of generality, set  $\psi = 0$ on the air--sea interface.  Then $\psi = -p_0$ on the bed $\{y = -d\}$, where $p_0$ is the \emph{(relative) pseudo-volumetric mass flux}: 
\be p_0 := \int_{-d}^{\eta(x)} \sqrt{\varrho(x,y)}\left[u(x,y) -c \right] \, dy. \label{defp0} \ee
It is straightforward to show that $p_0$ is a (strictly negative) constant, i.e., it does not depend on $x$ (cf., e.g.,  \cite{walsh2009stratified}).  Physically, $p_0$ describes the rate of fluid moving through any vertical line in the fluid domain and with respect to the transformed vector field $\sqrt{\varrho}(u-c,v)$.

The conservation of mass \eqref{weakmass} implies that $\nabla \varrho$ is orthogonal to the velocity field in each layer, and hence we may let $\rho:[p_0,0] \to \mathbb{R}^{+}$ be given such that 
\be \varrho(x,y) = \rho(-\psi(x,y)) \label{defrho} \ee
throughout the fluid. The choice to use $-\psi$ as the argument is motivated by the change of variables introduced in \S\ref{height eq formulation section}. We shall refer to $\rho$ as the \emph{streamline density function}, though one may alternatively view it as the Lagrangian density.  Conversely, $\varrho$ will be called the Eulerian density.  From the definition and \eqref{varrhopositive}, we see that
\be \rho > 0, \qquad \textrm{in } [p_0,0].\label{rhopositive} \ee
Moreover, taking into account the values of $\psi$ on the boundary of $\Omega$ and \eqref{nostagnationpsi}, the stable stratification condition \eqref{stablestratificationeuler} is equivalent to 
\be p \mapsto \rho(p) \textrm{ is non-increasing on $[p_0, 0]$.} \label{stablestratification} \ee

Conservation of energy can be expressed via Bernoulli's theorem, which states that the quantity
\be E := P + \frac{\varrho}{2}\left( (u-c)^2 + v^2\right) + g\varrho y, \label{defE} \ee
is constant along streamlines.  This is well-known for classical solutions of the Euler equations, and remains true in the weak setting (we confirm this in the process of proving Lemma \ref{equivalence lemma 1}).  This permits us to define a function $\beta:[0,|p_0|] \to \mathbb{R}$ such that 
\be \frac{dE}{d\psi}(x,y) = -\beta(\psi(x,y)), \qquad \textrm{in } \bigcup_i \Omega_i. \label{defbeta} \ee
Following the terminology of \cite{walsh2009stratified,walsh2014local,walsh2014global}, we call $\beta$ the \emph{Bernoulli function} corresponding to the flow; it describes roughly how the Bernoulli constant varies with respect to the streamlines.    	

In this work, we will consider waves with a Bernoulli function of a specific form: we say that the wave is \emph{periodic localized near the crest} provided that
$$ \beta(\psi) = \rho^\prime(-\psi) \left[ \frac{1}{2} c^2 + g\int_{p_0}^{-\psi} \frac{1}{c \sqrt{\rho(s)}} \, ds - gd \right].$$
While it may not be immediately apparent, this choice has a physical significance:  one can show that any \emph{solitary} stratified wave limiting to uniform flow upstream and downstream must have a Bernoulli function of the type above (see \S\ref{localized near crest section}).  These were the first class of stratified waves for which an exact solution theory was obtained \cite{terkrikorov1963theorie}, and it is the only class for which the existence of large-amplitude solitary waves is currently known.  Although this paper considers the periodic case, the waves we construct will decay  exponentially away from the crest, with a rate that is independent of the period for $L$ sufficiently large.  Loosely speaking, they are periodic approximations of solitary waves.  For a more thorough discussion, see Remark \ref{localized periodic remark}.

\subsection{Informal statement of results} \label{intro statement of results section}

We now give a summary of our results, interspersed with some explanatory comments.  For the time being, several of the hypotheses are left unquantified.  The  complete statement is in \S\ref{proof of main theorems section}.

Fix a H\"older exponent $\alpha \in (0,1)$, and put $r := 2/(1-\alpha)$.  Choose a pseudo volumetric mass flux $p_0 < 0$, period $2L$, and ocean depth $d > 0$.  Let $\rho_* \in C^{1,\alpha}([-p_0,0])$ be a stably stratified streamline density function, and suppose that $(u_*,v_*, \varrho_*, P_*, \eta_*)$ is a solution of Problem \ref{weak Euler prob}.  Assume further that  (i) it is periodic localized near the crest with period sufficiently large, (ii) its wave speed $c_*$ is supercritical, (iii) it is a wave of strict elevation, and (iv) it is sufficiently small-amplitude.  Then each of the statements (A1), (A2), (B), and (C) below hold true. 

\begin{itemize}
\item[(A1)]  \emph{Existence of nearby many-layered solutions.} There is a neighborhood $\mathcal{U}$ of $\rho_*$ in $L^\infty([-p_0,0])$ such that, for any $\rho \in \mathcal{U}$ that is non-increasing and piecewise smooth, there exists a solution $(u,v, \varrho, P, \eta)$ to the steady Euler equations with streamline density function $\rho$, period $L$, and wave speed $c$.  Moreover,  $u$ and $\eta$ are even in $x$, while $v$ is odd in $x$.  \end{itemize}  

  The key point here is that $\mathcal{U}$ contains streamline densities functions with \emph{arbitrarily many jump discontinuities}.   These many-layered solutions in fact converge to the continuously stratified wave as $\rho \to \rho_*$ in $L^\infty$ in the following sense: 
\begin{itemize}
\item[(A2)]  \emph{Convergence of the height function and wave speed.} For each $(q,p) \in \mathcal{R} := \mathbb{R} \times [p_0, 0]$, let $h_*(q,p)$ denote the height above the bed $\{ y = -d\}$ of the point with $x$-coordinate $q$ that lies on the streamline $\{ \psi_* = -p\}$ for the wave with velocity field $(u_*, v_*)$; let $h$ designate the corresponding height for the wave furnished by statement (A1).  Then 
$$  h_* = h + \mathcal{O}(\| \rho - \rho_*\|_{L^\infty}) \qquad \textrm{in } W_{\mathrm{per}}^{1,r}(\mathcal{R}) \subset C_{\textrm{per}}^{0,\alpha}(\overline{\mathcal{R}}).$$
Likewise, the wave speed $c$ satisfies 
$$ c_* = c + \mathcal{O}(\| \rho - \rho_* \|_{L^\infty}).$$ \end{itemize}
The reason we do not write  (A1) and (A2) directly in terms of $(u,v)$ and $(u_*,v_*)$ is simple:  the velocity fields are defined on different domains --- $\Omega$ and $\Omega_*$, respectively --- and so comparing them in a single function space is unwieldy.   The formulation of the problem in terms of the height function is described in \S\ref{height eq formulation section}, and the equivalence of this to the original Euler formulation is proved in Lemma \ref{equivalence lemma 1}.   We also mention that there is an exact expression for $c-c_*$, see \eqref{wave speed estimate}. 

Statements (A1) and (A2) are a form of continuity result.  Let $\mathscr{D}$ denote the set of bounded, layer-wise smooth, stable streamline density functions; $\mathscr{D}$ can be viewed as a convex subset of $L^\infty([p_0,0])$.  Then (A1) proves that there exists a mapping $\rho \in \mathscr{D} \cap \mathcal{U} \mapsto h \in W_{\textrm{per}}^{1,r}(\mathcal{R})$, and (A2) follows from the fact that this mapping is Lipschitz continuous. 
 
Away from the internal interfaces, the solutions enjoy improved regularity: 
\begin{itemize}
\item[(B)]  \emph{Improved regularity.} Let $I \subset\subset [p_0, 0] \setminus \{ p_1, \ldots, p_{N-1}\}$ be a connected set for which $\rho \in C^{1,\alpha}(\overline{I})$.  Then
$$ \| h - h_*\|_{C_{\mathrm{per}}^{1,\alpha}(\mathbb{R} \times \overline{I})}  \leq C_1 \left(   \| \rho - \rho_*\|_{L^\infty([p_0,0])} +  \| \rho - \rho_*\|_{C^{1,\alpha}(\overline{I})} \right) ,$$
where $C_1 > 0$ depends on the length of $I$, $\rho_*$, and $h_*$.  
%In particular, as $\eta_* = h_*(\cdot, 0)$, this implies $\| \eta - \eta_*\|_{C^{1,\alpha}} \lesssim \| \rho - \rho_*\|_{W^{1,\infty}}$, where the $W^{1,\infty}$ norm is evaluated in the layer directly below the air--sea interface.   Moreover, if $\tilde\Omega \subset \Omega \cap \Omega_*$ is an open region compactly contained in set on which $\varrho$ is $C^{1,\alpha}$, then 
%$$ \| u - u_* \|_{C_{\textrm{per}}^{0,\alpha}(\tilde \Omega)}, ~ \| v - v_* \|_{C_{\textrm{per}}^{0,\alpha}(\tilde \Omega)} \lesssim \| \varrho - \varrho_* \|_{C_{\textrm{per}}^{1,\alpha}(\tilde \Omega)}.$$
\end{itemize}
In general, $C_1$ will increase as the length of $I$ decreases.  One consequence of (B) is that, if $\rho_*$ is  constant in some region, then the approximation by a layer-wise constant density stratification converges in a higher regularity norm there.  

Lastly, we prove a result on the convergence of the pressure.  This is specifically aimed at the surface reconstruction problem.  
\begin{itemize}
\item[(C)] \emph{Convergence of the pressure.} Let a connected set $I \subset\subset [p_0, 0]\setminus \{ p_1, \ldots, p_{N-1}\}$ be given with $p_0 \in I$, and assume that $\rho \in C^{1,\alpha}(\overline{I})$.   Denote by  $P_{\textrm{b}}$  the trace of the pressure on the ocean bed for the traveling wave with density $\rho$, and let $P_{{\textrm{b}}*}$ be the trace of $P_*$ on the bed.    Then 
$$ \| P_{\textrm{b}} - P_{\textrm{b}*} \|_{C_{\textrm{per}}^{0, \alpha}(\mathbb{R})} \leq C_2 \left(   \| \rho - \rho_*\|_{L^\infty([p_0,0])} +  \| \rho - \rho_*\|_{C^{1,\alpha}(\overline{I})} \right) ,$$
where $C_2 > 0$ depends on the length of $I$, $\rho_*$, and $h_*$
\end{itemize}
Actually, we prove something much stronger than this:  convergence occurs in a region, not merely on the bed.   

\begin{remark}  A large family of waves meeting these hypotheses are known to exist due to the work of Turner \cite{turner1981internal,turner1984variational}, Amick \cite{amick1984semilinear}, and Amick--Turner \cite{amick1986global}; see also Theorem \ref{Turner main theorem}.  For the precise definition of periodic waves localized near the crest, and strict waves of elevation,  see Definition \ref{definition localized periodic}.  Roughly speaking, a wave of elevation is one where each streamline lies above the corresponding streamline in a hydrostatic flow.  The concept of supercritical speed is discussed in \S\ref{ter-krikorov formulation section}, and is strongly connected to the idea of conjugate flows (cf., e.g., \cite{benjamin1971unified}). We also mention that, in the constant density case, M. Wheeler \cite{wheeler2014froude} has recently proved that hypothesis (iii) implies hypothesis (ii), which suggests that our assumptions can be pared down further.  
\end{remark}

The closest analogue to these results in the mathematical literature is due to G. James \cite{james2001internal}, who considered the reverse limit.  That is, he showed that in an $L^2$ neighborhood of a steady {two-layer} solitary wave in a channel with piecewise constant streamline density, there is a manifold of continuously stratified waves.  While this shares some common features with the present work, they are quite distinct.  For either case, one of the main challenges is finding a formulation that encapsulates both layer-wise continuous and continuous stratification.  Because we must contend with arbitrarily many layers, however, the similarity more-or-less ends there.   The choice of tools is also quite different.  James employs some sophisticated techniques from spatial dynamics, essentially a center manifold reduction approach in the spirit of Kirchg\"assner \cite{kirchgassner1982wavesolutions}.  On the other hand, our method, at its heart, amounts to a novel application of the implicit function theorem supplemented by a penalization scheme and nonlinear elliptic PDE theory.

Applied scientists frequently elect to use simply a two-layer model.  Needless to say, this introduces some baseline error into the modeling.  In certain physical applications --- particularly when the pycnoclines are extremely thin --- this error is higher order and two-layer schemes are well-aligned with experimental data  \cite{camassa2011optimal,grue1999properties}.   For flows with relatively fat pycnoclines, however, the two-layered model is less successful, and so a number of alternative procedures have been proposed (cf., e.g., \cite{fructus2004grue,grue2000breaking,grue2002solitary}).  We will not give a full account of the applied literature on this topic.  Suffice it to say that our results imply that the many-layered approximation will converge, and hence the baseline error can be made arbitrarily small. 

Finally, let us consider some potential applications and extensions.  One of our primary reasons for initiating this program was the desire to further the qualitative theory of steady stratified waves.  As one example, we mention the problem of recovering the air--water interface of a traveling wave knowing only its wave speed, its upstream and downstream form, and its pressure on the ocean floor.  In the irrotational and homogeneous density setting, this has recently been studied by several authors.  Constantin \cite{constantin2012pressure}, and Clamond and Constantin \cite{constantin2013pressure} derived an explicit formula relating the trace of the pressure on the bed to $\eta$.   Concurrently,  Oliveras, Vasan, Deconinck, and Henderson \cite{oliveras2012recovering} obtained an implicit relation via an alternative formulation of the problem.    It turns out that each of these works can be readily adapted to the case of layer-wise irrotational and constant density waves.  Hence, we are able to reconstruct a continuously stratified wave using an approximation scheme.  This process will be detailed in an accompanying paper \cite{ChenWalsh2014pressure}.

It would be highly desirable to be able to treat directly the case of solitary waves.  This may indeed be possible, but it would require a nontrivial generalization of our approach.  Specifically, we rely on the Fredholm properties of several elliptic differential operators, which in general fail on unbounded domains.  On the other hand, the recent work of M. Wheeler \cite{wheeler2013solitary} on large-amplitude rotational (but constant density) solitary waves provides some ideas for resolving these issues.  This is something we hope to address in a forthcoming paper. 

Another natural improvement would be to broaden the class of allowable Bernoulli functions.  Over the past several years, a fairly robust existence theory for large-amplitude periodic steady stratified waves has been developed by the group of J. Escher, D. Henry, A.--V. Matioc, and B.--V. Matioc (cf. \cite{escher2011stratified,henry2013local,henry2013global}), and one of the authors (cf.  \cite{walsh2009stratified,walsh2014local,walsh2014global}).  In particular, these works allow for either a general $\beta$, or at least Bernoulli functions in a substantially less restrictive class.   It seems clear from the analysis in \S\ref{continuity section} that the continuity result would  hold for these solutions, provided that they were waves of elevation.  This, however, does not follow from the approach pursued by the above authors.  Some new ingredient may be necessary.

\subsection{Structure of the paper}  
We begin, in \S\ref{formulation section}, by introducing several more amenable formulations of the Euler system.  In particular, we employ the Dubreil-Jacotin transformation to fix the domain.  This {leads} us to the height function $h$ encountered statements (A2) and (B) above.  A further rescaling --- one that is especially well-suited to analyzing periodic waves localized near the crest --- furnishes a new unknown $w = w(\xi, \zeta)$, and a rescaled streamline density function $\mathring{\rho} = \mathring{\rho}(\zeta)$.  

The result of these efforts is a quasilinear divergence form PDE satisfied by $w$ in a periodic strip (see Problem \ref{ter-krikorov prob}).  In the absence of stagnation, the system is elliptic and $\mathring{\rho}$ {appears} as a coefficient.  However, in this formulation, stagnation is prevented precisely when $w_\zeta > -1$.  This is a serious difficultly:  the solutions $w$ we consider can only be expected to be of class $W_{\textrm{per}}^{1,r}$ on the whole strip.  They will naturally enjoy improved regularity inside each fluid layer, but because we are allowing for arbitrarily many layers, we cannot exploit this additional smoothness.

Our approach is to instead introduce a {penalized problem} in the spirit of Turner \cite{turner1981internal}.  That is, we add a cutoff function so that, when $\| \nabla w \|_{L^\infty}$ is larger than a certain threshold, the principal part of the PDE is replaced by $\nabla \cdot (\mathring{\rho} \nabla w)$, and when $\| \nabla w \|_{L^\infty}$ is sufficiently small, the problem agrees with the physical one.  

This enables us, in \S\ref{continuity penalized section}, to construct a smooth curve of solutions to the penalized problem parameterized by $\mathring{\rho}$ lying in an $L^\infty$ neighborhood $\mathring{\mathcal{V}}$ of $\mathring{\rho}_*$.    However, because $\mathring{\rho}$ is merely in $L^\infty$, elliptic regularity theory does not directly imply that these are  physical solutions:  in general, one does not have control of $w$ in $W^{1,\infty}$, and hence even for $\| \mathring{\rho} - \mathring{\rho}_*\|_{L^\infty} \ll 1$,  $\|\nabla w\|_{L^\infty}$ may lie above the cutoff threshold.  This is not merely a technical point: $\mathring{\mathcal{V}}$ includes densities with infinitely many jump discontinuities; the physicality of such flows is dubious at best.  

To complete the argument, we derive \emph{a priori} estimates for smooth rescaled streamline density functions lying inside $\mathring{\mathcal{V}}$.  Using a limiting procedure, we show that layer-wise smooth rescaled streamline densities in $\mathring{\mathcal{V}}$ likewise give rise to solutions of the physical problem, provided that they are sufficiently small-amplitude.  This is carried out in \S\ref{a priori section} and \S\ref{continuity proof section}, culminating in Theorem \ref{main continuity theorem}.  

Finally, in \S\ref{proof of main theorems section}, we translate this back into the language of statements (A)--(C) above.  This gives our main result, Theorem \ref{main theorem}.

\section{Reformulations} \label{formulation section}
In this section, we introduce a number of equivalent formulations of the steady stratified water wave system.  Each of these will be particularly suited to one of the problems that we consider in the remainder of the paper.  The fact that they are equivalent is far from obvious, particularly in the weak regularity setting.  We therefore include a proof, but relegate it to Appendix \ref{formulation equivalence appendix} as it is not our primary concern.  

\subsection{Stream function formulation} \label{stream function formulation section}
Recall from the introduction that the (pseudo relative) stream function $\psi$ is defined by 
$$ \nabla^\perp \psi = \sqrt{\varrho} (u-c,v) \qquad \textrm{in } \bigcup_i \Omega_i.$$
A relatively simple computation confirms that it solves Yih's equation 
$$ \Delta \psi - gy \rho^\prime(-\psi) + \beta(\psi)  = 0, \qquad \textrm{in } \bigcup_i \Omega_i. $$
Indeed, this relatively elegant expression was the motivation for defining $\psi$ as we did in \eqref{defpsi}; slightly less pleasant versions of Yih's equation were found earlier by Dubreil-Jacotin \cite{dubreil1937theoremes} and Long \cite{long1953some}.  We mention that, for weak solutions, the derivation of Yih's equation is not quite so simple (cf. Lemma \ref{equivalence lemma 1}).

By the kinematic boundary {conditions}, each of the free surfaces is a streamline.  As we have already discussed, the energy density $E$ \eqref{defE} is constant on streamlines.  Evaluating it on the air--sea interface yields
\be |\nabla \psi|^2 + 2g \varrho(x, \eta(x)) \left(\eta(x)+d\right) = Q, \qquad \textrm{on } y = \eta(x) \label{p2:capgravdefQ} \ee
where 
 \be Q := 2(E|_\eta-P_{\textrm{atm}} + g\varrho|_\eta d). \label{defQ} \ee
One can repeat this procedure at the interface between any two layers; by the continuity assumption on the pressure, this yields the identity
$$ \jump{|\nabla \psi |^2}_i + 2g \jump{\varrho}_i (y+d) = Q_i, \qquad \textrm{on $\{ y = \eta_i(x)\}$,} $$
where $\jump{\cdot}_i$ denotes the jump over the interface $\partial \Omega_i \cap \partial \Omega_{i+1}$ of a quantity defined on $\overline{\Omega_i \cup \Omega_{i+1}}$ from $\Omega_{i+1}$ to $\Omega_i$.  Here $Q_i$ is a constant representing the jump in the energy density across the $i$-th interface:
\be Q_i := 2( \jump{E}_i + g\jump{\varrho}_i d). \label{defQi} \ee 

Collecting these together, we arrive at the following reformulation of the Euler problem in terms of the stream function.

\begin{problem}[Weak stream function problem] \label{general weak stream function prob} Let a Bernoulli function $\beta$ be given with
$$ \beta \in L^r([p_0, 0]) \cap W^{1,r}([p_0, p_1]) \cap \cdots \cap W^{1,r}([p_{N-1}, 0]). $$
We say that a stream function $\psi$, interfaces $\eta$, $\eta_1$, \ldots, $\eta_{N-1}$, and constants $(Q, Q_1, \ldots, Q_{N-1})$ solve the weak stream function problem provided that the following statements hold: they exhibit the regularity 
$$ \psi \in W_{\mathrm{per}}^{1,r}(\Omega) \cap W_{\mathrm{per}}^{2,r}(\Omega_1) \cap \cdots \cap W_{\mathrm{per}}^{2,r}(\Omega_N), \qquad 
{{\eta,}} \, \eta_i  \in C^{1,\alpha}_{\mathrm{per}}(\mathbb{R}),$$
where $\Omega_i$ is defined as in \eqref{def Omega_i};   $\psi$ solves Yih's equation (in the distributional sense) 
\bse \label{streamfunctionprob}  \begin{align} \Delta \psi - gy \rho^\prime(-\psi) + \beta(\psi) & = 0, \qquad \textrm{in } \bigcup_i \Omega_i,  \label{weakyih} \end{align}
along with the boundary conditions 
\begin{align}
|\nabla \psi|^2 + 2g \rho (y+d) &= Q, \qquad \textrm{on } \{y = \eta(x)\}, \label{weakbernoulli} \\
\psi & = 0, \qquad \textrm{on } \{y = \eta(x)\}, \label{psisurfacecond} \\
\psi & = -p_i, \qquad \textrm{on } \{y = \eta_i(x)\}, \label{psitracecond} \\
\psi &= -p_0, \qquad \textrm{on } \{y = -d\}; \label{psibedcond} \end{align} 
 the corresponding pressure is continuous,
\be \jump{ |\nabla \psi|^2}_i + 2g \jump{{\varrho}}_i (y+d) = Q_i, \qquad \textrm{on $\{ y = \eta_i(x)\}$; } \label{psicontpress} \ee \ese
and there is no horizontal stagnation \eqref{nostagnationpsi}.
\end{problem} 
 
\subsection{Height equation formulation}  \label{height eq formulation section} A natural way to fix the boundary in the absence of stagnation is to use the streamlines as a vertical coordinate.  One strategy in this direction is to employ the Dubreil-Jacotin transformation 
$$ (x,y) \mapsto (x, -\psi(x,y)) =: (q,p),$$
which has the effect of mapping a single horizontal period of the fluid domain to the rectangle 
$$ \mathcal{R} := \{ (q,p) \in (-L,L) \times (p_0,0)\}.$$
Similarly, each fluid layer $\Omega_i$ is mapped to a strip
$$ \mathcal{R}_i := \{ (q,p) \in (-L,L) \times (p_{i-1}, p_i) \},$$
where $p_N := 0$, and $\partial \Omega_i \cap \partial \Omega_{i+1} = \{ \psi = -p_i \}$, for $i = 1, \ldots N-1$.    

Let $h = h(q,p)$ be the height above the bed of the point with $x = q$ and lying on the streamline $\{\psi = -p\}$, 
\be h(q,p) := y + d. \label{defheight} \ee
Assuming for the time being that $h$ and $\rho$ are smooth,  the stream-function problem  \eqref{streamfunctionprob} can be reformulated as follows: Find  $(h,Q)$ such that $h$ is $2L$-periodic in $q$, 
\be h_p > 0, \qquad \textrm{in $\overline{\mathcal{R}}$},\label{hppositive} \ee
 and the height equation is satisfied,
\be \left \{ \begin{array}{lll}
(1+h_q^2)h_{pp} + h_{qq}h_p^2 - 2h_q h_p h_{pq}  
\\ \qquad -g(h-d)\rho_p h_p^3  = -h_p^3 \beta(-p), & \textrm{in } \mathcal{R}, \\
& & \\
1+h_q^2 + h_p^2(2g\rho h- Q)  = 0, & \textrm{on } \{p = 0\}, \\
h = 0, & \textrm{on }  \{p = p_0\}. \end{array} \right. \label{heighteq} \ee
See \cite{walsh2009stratified} for the details. 

Now let us consider the situation where $\rho$ is layer-wise smooth.  We can recast \eqref{heighteq} in a weaker form by exploiting the divergence structure of the interior equation:
\bse \label{weakheightprob}
\begin{align}
\left( - \frac{1+h_q^2}{2h_p^2} + B - g \rho (h-d)  \right)_p + \left( \frac{h_q}{h_p} \right)_q  + g\rho h_p  & = 0,\qquad \textrm{in } \bigcup_i \mathcal{R}_i, \label{weakheighteq} \\
-\frac{1+h_q^2}{2h_p^2}  - g\rho h + \frac{Q}{2} & = 0, \qquad \textrm{on } \{ p = 0 \},\label{weakventtseleq} \\
 h & = 0, \qquad \textrm{on } \{ p = p_0\}. \label{weakbedcond} 
\end{align} 
Here 
\be B(p) := \int_0^p \beta(-s) \, ds, \qquad p \in [p_0, 0]. \label{defB}\ee 
The continuity of the pressure \eqref{psicontpress} becomes a transmission boundary condition posed on each interfacial streamline:
\be -\jump{\frac{1+h_q^2}{2h_p^2}}_i - g \jump{\rho}_i h + \frac{Q_i}{2} = 0 \qquad \textrm{on } \{ p = p_i \}. \label{weaktransmissioncond} \ee \ese
Alternatively, we combine \eqref{weakheighteq} and \eqref{weaktransmissioncond} to obtain a single PDE satisfied in the distributional sense on the entire strip $\mathcal{R}$:
\be \left( - \frac{1+h_q^2}{2h_p^2} + B + \sum_i \frac{Q_i}{2} \mathds{1}_{\mathcal{R}_i} - g \rho (h-d)  \right)_p + \left( \frac{h_q}{h_p} \right)_q  + g\rho h_p   = 0 \qquad \textrm{in } \mathcal{R}, \label{weakheighteq2}  \ee
where $\mathds{1}_{\mathcal{R}_i}$ is the indicator function for $\mathcal{R}_i$.  Of course these additional terms are meant to account for the fact that the energy $E$ will  jump across the interfaces.  

\begin{problem}[Weak height equation problem] \label{general weak height eq prob} 
Let streamline density function $\rho$, and Bernoulli function $\beta$  be given with the regularity
\begin{align*}
 \rho, ~\beta &\in L^r([p_0, 0]) \cap W^{1,r}([p_0, p_1]) \cap \cdots \cap W^{1,r}([p_{N-1}, 0]). %\\
 %\beta & \in L^r([p_0, 0]) \cap W^{1,r}([p_0, p_1]) \cap \cdots \cap W^{1,r}([p_{N-1}, 0]).
  \end{align*}
 We say $(h, Q, Q_1, \ldots, Q_{N-1})$ solves the weak height equation problem provided that the following statements hold:  $h$ exhibits the regularity   
$$ h  \in W_{\mathrm{per}}^{1,r}(\mathcal{R}) \cap W_{\mathrm{per}}^{2,r}(\mathcal{R}_1) \cap \cdots \cap W_{\mathrm{per}}^{2,r}(\mathcal{R}_N); $$
there is no stagnation \eqref{hppositive}; and the quasilinear elliptic system \eqref{weakheightprob} is satisfied.  Equivalently, we may replace \eqref{weakheighteq} and \eqref{weaktransmissioncond} with the requirement that \eqref{weakheighteq2} holds in the sense of distributions. \end{problem}

\subsection{Periodic waves localized near the crest} \label{localized near crest section}
  Up to this point, we have made no restriction on the form of the Bernoulli function $\beta$; we now specialize to the setting of our applications.  As motivation,  suppose for the moment that the fluid domain is of infinite extent in the horizontal direction.  Assume also that 
$$ (u-c, v) \to (-c, 0), ~ \eta \to 0 \qquad \textrm{as } x \to \pm \infty,$$  
meaning that the flow is irrotational and laminar at upstream and downstream infinity.  In particular, from the definition of $E$ in \eqref{defE}, this implies that 
$$ E = P_{\textrm{atm}} + \frac{1}{2} \varrho c^2 \qquad \textrm{on } \{ y = \eta(x) \},$$
whence, by \eqref{defQ},
$$ Q = \rho(0) c^2 + 2g \rho(0) d.$$

Consider the form that the corresponding Bernoulli function must take.   Observe that the limiting pressure will be hydrostatic
$$ P(x,y) \to \mathring{P}(y) := P_{\textrm{atm}} + g \int_y^0 \mathring{\varrho}(s) \, ds, \qquad \textrm{as } x \to \pm \infty,$$
where $\mathring{\varrho}$ is the limiting value of the Eulerian density:
$$ \varrho \to \mathring{\varrho} = \mathring{\varrho}(y), \qquad \textrm{as } x \to \pm \infty.$$

Let $h(q,p) \to \mathring{h}(p)$ as $q \to \pm \infty$.  That is, $\mathring{h}(p)$ is the asymptotic height above the bed of the streamline $\{ \psi = -p \}$.  Recalling the definition of $h$ in \eqref{defheight}, this means that $\mathring{y} := \mathring{h} - d$ is the limiting value of the $y$-coordinate of points on that streamline.  From the change of variables identity  
$$ h_p = \frac{1}{\sqrt{\rho}(c-u)},$$
we see that $\rho$ and $c$ determine $\mathring{h}$ according to
\be \mathring{h}(p) = \int_{p_0}^p \frac{1}{c \sqrt{\rho(s)}} \, ds. \label{identity for limiting h} \ee
Moreover, since $\eta$ limits to $0$ upstream and downstream, we see that 
\be \label{c determined by rho} d = \int_{p_0}^0 \frac{1}{c\sqrt{\rho(s)}} \, ds. \ee
In this work, we will keep $d$ and $p_0$ fixed, and so \eqref{c determined by rho} will determine the wave speed $c$.  

This allows us to compute the value of $E$ on an internal streamline (away from the discontinuities of $\rho$), by evaluating it at upstream or downstream infinity:
$$ E|_{\{ \psi = -p\}}  = \mathring{P}(\mathring{y}) + \frac{1}{2} \rho c^2 + g \rho \mathring{y}.$$
Differentiating this with respect to $p$ yields
\begin{align*} \beta(-p) &=  \mathring{P}^\prime(\mathring{y})  \mathring{y}_p+ \frac{1}{2} \rho^\prime c^2 + g \rho^\prime \mathring{y}+ g \rho \mathring{y}_p \\
& = \rho^\prime \left[ \frac{1}{2} c^2 + g\mathring{y} \right].  \end{align*}
Here we have used the fact that $\mathring{P}$ is hydrostatic to infer the second line from the first.  By a similar argument, we see that $\mathring{h}$ is related to the constants $Q_1$, \ldots, {$Q_{N-1}$} according to  
$$ Q_i = \jump{\rho}_i c^2 + 2g \jump{\rho}_i \mathring{h} \qquad \textrm{on } \{ p = p_i \}.$$

The above considerations {show} that a solitary wave that limits to a uniform irrotational flow upstream and downstream will necessarily have a Bernoulli function of a specific form.   Moreover,  the constants $Q$, $Q_1$, \ldots, {$Q_{N-1}$} are determined by the limiting heights of the corresponding streamlines.     We therefore make the following definition:
\begin{defn} \label{definition localized periodic} A periodic traveling wave is said to be \emph{localized near the crest} provided that the Bernoulli function $\beta$ for the flow is of the form 
\be \beta(-p) = \rho^\prime \left[ \frac{1}{2} c^2 + g(\mathring{h}-d) \right], \label{localized beta form} \ee
where $\mathring{h} : [p_0, 0] \to \mathbb{R}_+$ is defined by \eqref{identity for limiting h}, the wave speed $c$ is determined via \eqref{c determined by rho}, and the constants $Q, Q_i$ are given by
\be 
Q  = \rho(0) c^2 + 2g \rho(0) d, \qquad Q_i  = \jump{\rho}_i \left[ c^2 + 2g \mathring{h}(p_i)\right].  \label{Q and Q_i for localized wave} \ee
If, in addition, the height $h$ for the flow satisfies  
\be h-\mathring{h} \geq 0 \qquad \textrm{in } \mathcal{R}, \label{wave of elevation condition} \ee
then the wave is said to be a \emph{wave of elevation}, and a  \emph{strict wave of elevation} provided that
$$ h -\mathring{h} > 0 \qquad \textrm{in } \overline{\mathcal{R}} \setminus \{ p = p_0 \}.$$
\end{defn}

We now state precisely the formulations of the previous two subsections in the setting of periodic waves localized near the crest.  

\begin{problem}[Stream function problem for waves localized near the crest] \label{localized stream function prob} Let streamline density function $\rho$ be given with the regularity
\begin{align*}
 \rho &\in L^\infty([p_0, 0]) \cap W^{1,r}([p_0, p_1]) \cap \cdots \cap W^{1,r}([p_{N-1}, 0]),
% \mathring{h} & \in W_{\mathrm{per}}^{1,r}(\mathcal{R}) \cap W_{\mathrm{per}}^{2,r}(\mathcal{R}_1) \cap \cdots \cap W_{\mathrm{per}}^{2,r}(\mathcal{R}_N).  
\end{align*}
define limiting height $\mathring{h}$ by \eqref{identity for limiting h}, and let the wave speed $c$ be given by \eqref{c determined by rho}.   Find $(\psi, \eta_1, \ldots, \eta_{N-1}, \eta)$ with the regularity   
$$ \psi \in W_{\mathrm{per}}^{1,r}(\Omega) \cap W_{\mathrm{per}}^{2,r}(\Omega_1) \cap \cdots \cap W_{\mathrm{per}}^{2,r}(\Omega_N), \qquad 
\eta, \eta_i  \in C^{1+\alpha}_{\mathrm{per}}(\mathbb{R}).$$
  We require that $\psi$ is constant on $\partial \Omega_i$, there is no horizontal stagnation \eqref{nostagnationpsi}, and $\psi$ solves the elliptic PDE \eqref{streamfunctionprob}, where $\beta$, $Q$, $Q_1, \ldots, {Q_{N-1}}$ are defined according to \eqref{localized beta form}, and \eqref{Q and Q_i for localized wave}, respectively.
\end{problem}

\begin{problem}[Height equation problem for waves localized near the crest] \label{localized height equation prob}  Let streamline density function $\rho$  be given with the regularity
\begin{align*}
 \rho &\in L^\infty([p_0, 0]) \cap W^{1,r}([p_0, p_1]) \cap \cdots \cap W^{1,r}([p_{N-1}, 0]),
\end{align*}
 define the limiting height $\mathring{h}$ by \eqref{identity for limiting h}, and wave speed $c$ by \eqref{c determined by rho}.  Find  
$$ h  \in W_{\mathrm{per}}^{1,r}(\mathcal{R}) \cap W_{\mathrm{per}}^{2,r}(\mathcal{R}_1) \cap \cdots \cap W_{\mathrm{per}}^{2,r}(\mathcal{R}_N), $$
with no stagnation \eqref{hppositive}, and satisfying \eqref{weakheightprob}, where $\beta$, $Q$, $Q_1, \ldots, {Q_{N-1}}$ are defined according to \eqref{localized beta form}, and \eqref{Q and Q_i for localized wave}, respectively.
\end{problem}
\begin{remark}  The interior equation \eqref{weakheighteq} and  transmission boundary condition \eqref{weaktransmissioncond} can be captured by a single equation posed in $\mathcal{R}$ and satisfied in the sense of distributions:
\be \left( - \frac{1+h_q^2}{2h_p^2} - \left[ \frac{1}{2} c^2 + g(h-\mathring{h}) \right] \rho \right)_p + \left( \frac{h_q}{h_p} \right)_q  + g\rho (h-\mathring{h})_p   = 0,\qquad \textrm{in } \mathcal{R}. \label{localized weak height 2} \ee
\end{remark}

In Lemma \ref{equivalence lemma 1}, we prove that the stream function and height equation formulations are equivalent in the more general settings of Problem \ref{general weak stream function prob} and Problem \ref{general weak height eq prob}.  The equivalence of  Problem \ref{localized stream function prob} and Problem \ref{localized height equation prob} is then an immediate corollary.  Moreover, in Lemma \ref{equivalence lemma 1}, it is shown that the existence of solutions to the stream function problem (for any given $\beta$), implies the existence of a solution to Problem \ref{weak Euler prob}.    Thus, in particular, the existence of a solution to either Problem \ref{localized stream function prob} or Problem \ref{localized height equation prob} implies the existence of a solution to Problem \ref{weak Euler prob}.  The converse will not be true unless further restrictions are made to enforce the localization. 
%Taking a $\beta$ with the ansatz of \eqref{localized beta form}, the interior equation \eqref{heighteq} becomes
%\be (1+h_q^2)h_{pp} + h_{qq}h_p^2 - 2h_q h_p h_{pq} = h_p^3 \rho_p \left[ \frac{1}{2}c^2+ g (h-\mathring{h}) \right]. \label{localized height} \ee
%Likewise, the weak formulation reduces to the statement that
%\be \left( - \frac{1+h_q^2}{2h_p^2} - \left[ \frac{1}{2} c^2 + g(h-\mathring{h}) \right] \rho \right)_p + \left( \frac{h_q}{h_p} \right)_q  + g\rho (h-\mathring{h})_p   = 0,\qquad \textrm{in } \mathcal{R} \label{localized weak height 2} \ee
% in the distributional sense.   Note that the values of $\mathring{h}$ at the interfaces dictated by \eqref{Q_i for localized wave} ensures that the transmission boundary conditions \eqref{weaktransmissioncond} are satisfied. 

\begin{remark} \label{localized periodic remark}  (i)  Periodic waves localized near the crest are particularly well-suited to approximation by many-layered constant density irrotational flows.  To see this, note that $\beta$ will vanish in any region where $\rho$ is constant.  Recalling Yih's equation \eqref{streamfunctionprob}, this implies that the flow is irrotational in any region of constant density.  On the other hand, in continuously stratified regions, it is easy to see that the vorticity will \emph{not} vanish.  

(ii) From \eqref{localized weak height 2}, it is immediately apparent that waves of elevation are especially important, since for these solutions the term $\rho_p (h-\mathring{h})$  is non-positive.  This implies that the height function is a supersolution of a certain quasilinear elliptic operator, a fact we strongly exploit in \S\ref{a priori section}.  In particular, it is used to deduce monotonicity properties that lead to the key \emph{a priori} estimates.   
\end{remark}

\subsection{Ter-Krikorov formulation for localized waves}  \label{ter-krikorov formulation section}

With Remark \ref{localized periodic remark} in mind, it is useful to consider a further reformulation of the problem specifically aimed at waves of elevation localized near the crest.  This idea, to the best of our knowledge, originates with Ter-Krikorov \cite{terkrikorov1963theorie}.  It is also used in the works of Turner \cite{turner1981internal,turner1984variational}, Amick \cite{amick1984semilinear}, and Amick--Turner \cite{amick1986global} that we draw on in \S\ref{continuity section}.  

We begin by performing a change variables 
$$ (q, p) \mapsto \frac{1}{d}(q, \mathring{y}(p)) =: (\xi, \zeta).$$
Unravelling definitions, we see that $\zeta$ is a streamline coordinate which has been rescaled and non-dimensionalized so that the $y$-coordinates of points sitting on the streamline with label $\zeta$ limit to $\zeta d$.  It can be related directly to $p$ via the formula  
$$ \zeta = \frac{1}{cd} \int_0^p \frac{1}{\sqrt{\rho(s)}} \, ds.$$
Note that the positivity of $\rho$ guarantees that $\zeta(p)$ has an inverse, call it $p(\zeta)$.   The rectangle $\mathcal{R}$ is mapped by the transformation $(q,p) \mapsto (\xi, \zeta)$ to the strip
$$ \mathcal{S} := \{ (\xi, \zeta) : \xi \in (-\frac{L}{d}, \frac{L}{d}), ~ \zeta \in (-1, 0) \}.$$
Analogously, the layers $\mathcal{R}_i$ are mapped to strips
$$ \mathcal{S}_i := \{ (\xi, \zeta) : \xi \in (-\frac{L}{d}, \frac{L}{d}), ~ \zeta \in (\zeta_{i-1}, \zeta_i) \},$$
where $\zeta_i$ is the image of $p_i$.  

Finally, we introduce a new unknown 
\be w(\xi, \zeta) := \dfrac{y(\xi, p(\zeta))}{d} - \zeta. \label{def w} \ee
Recalling the definition of $\zeta$, it is clear that $w$ is a dimensionless quantity measuring the deviation of the height of a point on a streamline from its asymptotic height.  In fact, it is nothing but a rescaled version of $h - \mathring{h}$.  As we have mentioned, $w$ is an extremely natural choice of unknown for waves of elevation. Indeed, $w$ is a wave of elevation precisely when 
\be w \geq 0, \qquad \textrm{in } \mathcal{S}, \label{ter-krikorov wave of elevation} \ee
and a strict wave of elevation provided that 
\be w > 0, \qquad \textrm{in } \overline{\mathcal{S}} \setminus \{ \zeta = -1 \}. \label{def strict wave of elevation} \ee

Let $\mathring{\rho}$ denote the rescaled streamline density function, 
\be \mathring{\rho}(\zeta) := \rho(p(\zeta)). \label{def scaled rho} \ee
  This is slightly inconsistent with our notation in \S\ref{localized near crest section}, but we justify it on the grounds that the rescaling of $\rho$ is being done with a view towards some form of limiting behavior upstream and downstream.  More importantly, this choice allow us to avoid introducing another symbol or variety of accent mark.

From the definitions above, it is elementary to show that  
\begin{gather*} \partial_\xi = \frac{1}{d}\partial_q, \qquad \partial_\zeta = \frac{c\sqrt{\rho}}{d} \partial_p \\
\partial_{\xi}^2 = \frac{1}{d^2} \partial_q^2, \qquad \partial_\xi \partial_\zeta = \frac{c \sqrt{\rho}}{d^2} \partial_p \partial_q, \qquad \partial_\zeta^2 = \frac{c^2}{d^2} \rho \partial_p^2 + \frac{1}{2} \frac{c^2}{d^2} \rho_p \partial_p, \end{gather*}
and hence the height equation \eqref{localized weak height 2} translates to the following divergence form quasilinear system for $w$: 
\bse \label{ter-krikorov eq}  \begin{align} 
 \left( \mathring{\rho} \frac{w_\xi}{1+ w_\zeta} \right)_\xi +  \left( \mathring{\rho} \frac{w_\zeta}{1+w_\zeta} - \mathring{\rho} \frac{w_\xi^2 + w_\zeta^2}{2(1+w_\zeta)^2} \right)_\zeta - \lambda \partial_\zeta ( \mathring{\rho} w) + \lambda \mathring{\rho} w &= 0, \qquad \textrm{in $\mathcal{S}$} \label{ter-krikorov interior eq} \\
w & = 0, \qquad \textrm{on } \{ \zeta = -1\} \label{ter-krikorov bed cond} \\
\mathring{\rho}\left( \frac{w_\zeta}{1+w_\zeta} - \frac{w_\xi^2 + w_\zeta^2}{2(1+w_\zeta)^2} \right) - \lambda \mathring{\rho} w & = 0, \qquad \textrm{on } \{\zeta = 0 \}. \label{ter-krikorov conormal cond} 
\end{align} 
Here
\be \lambda := \frac{gd}{c^2} \label{def lambda}, \ee
which is the Richardson number for the flow.  Note that the no stagnation condition, stated in terms of $w$, is simply
\be w_\zeta > -1. \label{no stagnation w} \ee\ese

\begin{problem}[Ter-Krikorov problem] \label{ter-krikorov prob}  For a given scaled streamline density function
\be  \mathring{\rho} \in L^r((-1, 0)) \cap W^{1,r}((-1, \zeta_1)) \cap \cdots \cap W^{1,r}((\zeta_{N-1}, 0)), \label{regularity of mathringrho} \ee
  find $(w, \lambda)$ with
$$ w  \in W_{\mathrm{per}}^{1,r}(\mathcal{S}) \cap W_{\mathrm{per}}^{2,r}(\mathcal{S}_1) \cap \cdots \cap W_{\mathrm{per}}^{2,r}(\mathcal{S}_N), $$
 satisfying \eqref{ter-krikorov eq}.
\end{problem}
\begin{remark}  We are abusing notation here, since $w$ will be $2L/d$-periodic, rather than $2L$-periodic.  In the sequel, whenever we refer to a space of periodic functions in the $(\xi, \zeta)$-coordinates, this is how it should be interpreted.
\end{remark}
We prove in Lemma \ref{equivalence ter-krikorov} that Problem \ref{ter-krikorov prob} is equivalent to Problem \ref{localized height equation prob}.   By the remarks in the previous section, this implies that solutions of Ter-Krikorov formulation lead to solutions of the other problem formulations as well.  

Problem \ref{ter-krikorov prob} has been studied {by} many authors.  We paraphrase here the existence theorem most relevant to the focus of the present work.  Before that, we must introduce one additional concept:  for a stable rescaled streamline density function $\mathring{\rho}$, the corresponding \emph{critical wave speed} is given by  
\be c_{\textrm{crit}} = c_{\textrm{crit}}(\mathring{\rho}) := \sqrt{gd} \left[ \inf_{\substack{v \in H_{\textrm{per}}^1(\mathcal{S}),\\ v \not\equiv 0}} \frac{ \int_{\mathcal{S}} \mathring{\rho} |\nabla v|^2 \, d\xi \, d\zeta}{ -\int_{\mathcal{S}} \mathring{\rho}^\prime v^2 \, d\xi \, d\zeta + \int_{\{ \zeta = 0\}} \mathring{\rho} v^2 \, d\xi}  \right]^{-1/2}. \label{def crit c} \ee
Physically, this corresponds to the speed at which infinitesimal long waves propagate (cf., e.g., \cite{benjamin1971unified}).  The relevance becomes clearer when it is expressed in terms of $\lambda_{\textrm{crit}} := gd/c_{\textrm{crit}}^2$, the critical Richardson number,  which will satisfy 
\be \lambda_{\textrm{crit}} = \lambda_{\textrm{crit}}(\mathring{\rho}) = \inf_{\substack{v \in H_{\textrm{per}}^1(\mathcal{S}),\\ v \not\equiv 0}} \frac{ \int_{\mathcal{S}} \mathring{\rho} |\nabla v|^2 \, d\xi \, d\zeta}{ -\int_{\mathcal{S}} \mathring{\rho}^\prime v^2 \, d\xi \, d\zeta + \int_{\{ \zeta = 0\}} \mathring{\rho} v^2 \, d\xi}. \label{def lambda crit} \ee
The right-hand side above is easily recognizable as the Rayleigh quotient corresponding to the linearization of \eqref{ter-krikorov eq} about $w \equiv 0$.  We say that $c$ is \emph{supercritical} provided that $c > c_{\textrm{crit}}$, or, equivalently, $\lambda < \lambda_{\textrm{crit}}$.  

\begin{thm}[Turner, \cite{turner1984variational}]  \label{Turner main theorem} There exists a constant $R_{\max} > 0$, and a minimal period $L_{\min}$, depending on $R_{\max}$, such that the following holds.  Fix a rescaled streamline density function $\mathring{\rho}$ as in \eqref{regularity of mathringrho}.  If the period $L \in [L_{\min}, \infty]$, then for each $0 < R \leq R_{\max}$, there exists a solution $w$ to Problem \ref{ter-krikorov prob} for some choice of $\lambda$.  (Here $L = \infty$ corresponds to a solitary wave).   This solution will satisfy 
$$ \int_{\mathcal{S}} \mathring{\rho} \frac{|\nabla w|^2}{1+\partial_\zeta w} \, d\xi \, d\zeta = R^2,$$
will be a strict wave of elevation, and even in the $\xi$-variable.   Moreover, the wave speed will be supercritical, with the explicit bound:
\be \lambda \leq \lambda_{\mathrm{crit}} (1- CR^{4/3}) < \lambda_{\mathrm{crit}}, \label{supercritical bound} \ee
for some constant $C > 0$.  Lastly, $\xi \mapsto w(\xi, \cdot)$ is monotonically decreasing from the crest at $\{ \xi = 0 \}$ to the trough $\{ \xi = L/d\}$.  In fact, $w$ and $|\nabla w|$ are localized exponentially near the crest, with a rate of decay depending on $\mathring{\rho}$ and $R$, but independent of $L$.  
\end{thm}

%Turner constructs the solitary wave as a limit of periodic waves taking the period to $\infty$, which is possible because of this exponential localization property.   In contrast, the periodic solutions obtained in \cite{walsh2009stratified,walsh2014local,walsh2014global} are \emph{not} localized, nor are they waves of elevation in the sense above.  The key difference separating these results is the super-criticality of the wave speed.  Roughly speaking, the method of  \cite{walsh2009stratified,walsh2014local,walsh2014global}, it is shown that there exists a curve of non-laminar solutions bifurcating from a family of laminar flows.  This bifurcation occurs at a point where the linearized problem has a generalized eigenvalue, corresponding to a solution of the form $w(\xi, \zeta) = v(\zeta) 

%  Note that the critical Richardson number $\lambda_{\textrm{crit}}$ is a minimal eigenvalue for the linearized problem.  Performing a separation of variables, we may consider solutions with the ansatz $w(\xi, \zeta) = v(\zeta) \cos{(k \pi d \xi/L)}$.  It is easy to confirm that the Rayleigh quotient in \eqref{def lambda crit} minimized when $k = 0$, in other words, at the $0$-th mode.  On the other hand, in \cite{walsh2009stratified,walsh2014local,walsh2014global} we impose conditions that imply the minimum eigenvalue occurs at the $1$-st mode, and construct      
%\end{remark}

\section{Continuous dependence on the density in the Ter-Krikorov formulation} \label{continuity section}
\subsection{Overview}  In this section, we prove that small-amplitude periodic traveling waves of elevation localized near the crest depend continuously on the streamline density function.  Stated in terms of the Ter-Krikorov formulation, the main result is the following. 

\begin{thm} \label{main continuity theorem}  
Let $\mathring{\rho}_* \in C^{1,\alpha}([-1,0])$ be a stable rescaled streamline density function with $\mathring{\rho}_*(0) = 1$.  There exists $S_{\mathrm{max}} > 0$ such that, for any non-laminar solution $(w_*, \mathring{\rho}_*, \lambda_*)$ of Problem \ref{ter-krikorov prob} that is a strict wave of elevation \eqref{def strict wave of elevation}, and satisfies $\| \nabla w_* \|_{L^\infty} < S_{\mathrm{max}}$, the following is true.  There is a constant $\mathring{\rho}_{\max} > 0$, and a neighborhood $\mathring{\mathcal{U}}$ of $\mathring{\rho}_*$ in $L^\infty([-1,0])$ such that, for any stable rescaled streamlined density function
\be  \mathring{\rho} \in \mathring{\mathcal{U}} \cap  W^{1,\infty}([-1, \zeta_1]) \cap \cdots \cap W^{1,\infty}([\zeta_{N-1}, 0]) \label{continuity reg of mathringrho} \ee
with $\mathring{\rho}(0) = 1$, and $\mathring{\rho}(-1) \leq \mathring{\rho}_{\max}$,  there exists 
$$w \in W_{\mathrm{per}}^{1,r}(\mathcal{S}) \cap W_{\mathrm{per}}^{2,r}(\mathcal{S}_1) \cap \cdots \cap W_{\mathrm{per}}^{2,r}(\mathcal{S}_N)$$
with $(w, \mathring{\rho}, \lambda)$ solving Problem \ref{ter-krikorov prob}. Moreover,
$$\| w - w_* \|_{C^{0,\alpha}(\overline{\mathcal{S}})} \leq C\| \mathring{\rho} - \mathring{\rho}_* \|_{L^\infty},$$
for a constant $C > 0$ independent of $\mathring{\rho}$.  
  \end{thm}
\begin{remark}  In light of Theorem \ref{Turner main theorem}, we know that solutions $(w_*, \mathring{\rho}, \lambda_*)$ meeting the hypotheses exist.  Moreover, for Turner's solutions, one may replace the bound on $\| \nabla w_* \|_{L^\infty}$ with one on the energy:  There exists $R_0 > 0$ such that, for any non-laminar solution $(w_*, \mathring{\rho}_*, \lambda_*)$ of Problem \ref{ter-krikorov prob} furnished by Theorem \ref{Turner main theorem} with $R < R_0$, the conclusion of the above theorem holds.  This follows from the \emph{a priori} estimates of \S\ref{a priori section}, and the analogous ones in \cite{turner1984variational}).  We have taken $\mathring{\rho}_*(0) = \mathring{\rho}(0) = 1$ in order to simplify slightly some of the arguments, it is not essential.  
\end{remark}

Notice that the neighborhood of $\mathring{\rho}_*$ includes densities with an \emph{arbitrary} number of jump discontinuities.     To appreciate the implications of this, it is most convenient to re-express the Ter-Krikorov equation \eqref{ter-krikorov eq} in a more compact form:
\be  \begin{split}
\nabla \cdot \left( \mathring{\rho} \mathbf{F}(\nabla w) \right) - \lambda \partial_\zeta ( \mathring{\rho} w) + \lambda \mathring{\rho} \partial_\zeta w &= 0  \qquad \textrm{in } \mathcal{S} \\
w & = 0 \qquad \textrm{on } \{ \zeta = -1\} \\
\mathring{\rho} F_2(\nabla w) - \lambda \mathring{\rho} w &= 0 \qquad  \textrm{on } \{ \zeta = 0 \}.  \end{split}  \label{ter-krikorov eq2} \ee
Here the divergence and gradient are with respect to the $(\xi, \zeta)$ variable, and 
$$\mathbf{F} = (F_1, F_2) = (\partial_1 f, \partial_2 f),$$
 with $f: \mathbb{R}^2 \to \mathbb{R}$ defined by
\be f(p_1 , p_2) := \frac{p_1^2 + p_2^2}{2(1+p_2)}. \label{def f} \ee
Note that we are using $(p_1, p_2)$ as dummy variables; they have no connection to $p_0,\ldots, p_{N-1}$ introduced in \S\ref{stream function formulation section}.

Written this way, it is clear that \eqref{ter-krikorov eq2} is a quasilinear elliptic problem, with ellipticity constant related to the lower bound of $1+ w_\zeta$, and that the boundary condition on the air--sea interface $\{ \eta = 0 \}$ is of co-normal derivative type.  Suppose that we have a solution $(w_*, \mathring{\rho}_*, \lambda_*)$.  To establish the continuous dependence, we employ an implicit function theorem, attempting to find a curve of nearby solutions $\{(w, \mathring{\rho})\}$ parameterized by $\mathring{\rho}$.  However, because we cannot know \emph{a priori} the location of the layers, we may only assume that $\| \mathring{\rho} - \mathring{\rho}_* \|_{L^\infty} \ll 1$.   Heuristically, elliptic regularity would then provide control of $w$ in the space $W_{\textrm{per}}^{1,r}(\mathcal{S})$.  This is not enough: we must have that $w$ lies in a $W_{\textrm{per}}^{1,\infty}(\mathcal{S})$ neighborhood of $w_*$ in order to guarantee $1+w_\zeta > 0$.

To circumvent this issue, we replace the physical problem \eqref{ter-krikorov eq2} with a penalized problem that is elliptic for any $w \in W^{1,r}$ and agrees with \eqref{ter-krikorov eq2} for $w$ with energy below a certain bound.  In \S\ref{continuity penalized section}, we carry out the implicit function theorem scheme to get continuous dependence on the density for the penalized problem.  We then derive \emph{a priori} estimates in \S\ref{a priori section} for solutions of the penalized problem in terms of their energy and $\| \mathring{\rho}\|_{L^\infty}$.  Finally, in \S\ref{continuity proof section}, these estimates enable us to return to the physical problem, proving Theorem \ref{main continuity theorem}.

\subsection{Continuity for a penalized problem} \label{continuity penalized section}
Following Turner \cite{turner1981internal,turner1984variational}, consider the following penalization scheme.  Let $\Phi \in C_c^{\infty}(\mathbb{R})$ be a cutoff function such that 
$$
0 \leq \Phi \leq 1, \qquad \supp{\Phi} \subset (-3/2,3/2), \qquad \Phi \equiv 1 \textrm{ on } [-1, 1],
$$
and for each $s > 0$, define $\Phi_s := \Phi(\cdot/s)$.  We replace $f$ defined in \eqref{def f} with 
\be  a(p_1, p_2;s) := \Phi_s(p_1^2 + p_2^2) f(p_1, p_2) + \left[1-\Phi_s(p_1^2 + p_2^2)\right] \frac{p_1^2 + p_2^2}{2}. \label{def penalized a} \ee
To keep notation manageable, the dependence of $a$ on $s$ will be suppressed. Note that for $w \in W_{\textrm{per}}^{1,r}(\mathcal{S})$,  
$$\Phi_s(|\nabla w|^2) \frac{1}{1+\partial_\zeta w}  ,~ \Phi_s^\prime(|\nabla w|^2) \frac{1}{1+\partial_\zeta w}  \in L^q(\mathcal{S}), \qquad \textrm{for all $q \in [1, \infty]$}.$$
We are therefore justified in looking for weak solutions $w \in W_{\textrm{per}}^{1,r}(\mathcal{S})$ of the penalized problem
\be 
\left\{ \begin{array}{ll} 
\nabla \cdot \left( \mathring{\rho} (\nabla a)(\nabla w) \right) - \lambda \partial_\zeta ( \mathring{\rho} w) + \lambda \mathring{\rho} \partial_\zeta w = 0 & \textrm{in } \mathcal{S} \\
w = 0 & \textrm{on } \{ \zeta = -1\} \\
\mathring{\rho} (\partial_2 a)(\nabla w) - \lambda \mathring{\rho} w = 0 & \textrm{on } \{ \zeta = 0 \}. \label{Turner penalized problem} \end{array} \right. \ee
where $\mathring{\rho} \in L^\infty([-1,0]) \subset L^r([-1,0])$.  

Observe that, if $\| \nabla w \|_{L^\infty}^2 < 2s$, then $a(\nabla w) = f(\nabla w)$, and hence a solution of \eqref{Turner penalized problem} solves \eqref{ter-krikorov eq2}.  However, for $\| \nabla w\|_{L^\infty}$ large, $a(\nabla w) = |\nabla w|^2/2$.  One can therefore prove the following lemma characterizing the ellipticity of the penalized problem.  
\begin{lemma}[Turner, \cite{turner1981internal}] \label{ellipticity lemma} There exists a constant $s_0 \in (0, 1/\sqrt{2})$ such that, for all $s \in (0, s_0)$, there are constants $\sigma_1, \ldots, \sigma_5, \nu > 0$ such that the following hold.
\begin{itemize}
\item[(i)] $\frac{1}{2} \sigma_1(p_1^2 + p_2^2) \leq a(p_1, p_2) \leq \frac{1}{2} \sigma_2 (p_1^2 + p_2^2)$ .
\item[(ii)] $\sigma_3(p_1^2 + p_2^2) \leq (\nabla a)(p_1, p_2) \cdot (p_1, p_2) \leq \sigma_4 (p_1^2 + p_2^2)$.
\item[(iii)] $|\nabla a(p_1, p_2)|^2 \leq \sigma_5 (\nabla a)(p_1, p_2) \cdot (p_1, p_2)$.
\item[(iv)] $\sum_{ij} a_{ij}(p_1, p_2) \xi_i \xi_j \geq \nu( \xi_1^2 + \xi_2^2)$, for all $(p_1, p_2), (\xi_1, \xi_2) \in \mathbb{R}^2$.
\item[(v)] $\nabla \partial_2 a = (0,1) + \mathcal{O}(s)$
\item[(vi)] $(\partial_i \partial_j \partial_k a)(p_1, p_2) = (\partial_i \partial_j \partial_k a)(0, 0) + \mathcal{O}(s)$.  Moreover, at $(0,0)$, $a_{111} = 0$, $a_{112} = 2$, $a_{122} = 0$, $a_{222} = 6$, and the rest are determined by symmetry.  
\item[(vii)] $\sigma_1, \ldots, \sigma_5, \nu = 1+\mathcal{O}(s).$
\end{itemize}
Here we are using the shorthand $a_{i} = \partial_i a$, where $\partial_1 := \partial_\xi$, $\partial_2 := \partial_\xi$, and similarly for $a_{ij}$, $a_{ijk}$.  
\end{lemma}

The penalized problem \eqref{Turner penalized problem} can be stated abstractly as 
\be \mathcal{G}(w, \mathring{\rho}, \lambda) = 0,\label{abstract penalized problem} \ee
where $\mathcal{G} : X_1 \times X_2  \times \mathbb{R} \to Y$, 
\begin{align*} X_1 & :=  \{ w \in W_{\textrm{per}}^{1,r}({\mathcal{S}}) : w = 0 \textrm{ on } \{ \zeta = -1\}  \} \\
X_2 &:=  L^\infty([-1,0]) \\
 Y &:= \{ u \in (W_{\textrm{per}}^{1,r'}(\mathcal{S}))^* : u = \mathcal{A} + \partial_\xi \mathcal{B}_1  + \partial_\zeta \mathcal{B}_2, \textrm{ for } \mathcal{A} , \mathcal{B}_1, \mathcal{B}_2 \in L^r({\mathcal{S}}) \},  \end{align*}
and, for each test function $\Psi \in W^{1,r^\prime}_{\textrm{per}}(\mathcal{S})$, 
\be \begin{split}
\langle \mathcal{G}(w, \mathring{\rho}, \lambda), \, \Psi \rangle & := -\int_{\mathcal{S}} [ \mathring{\rho} \nabla \Psi \cdot (\nabla a)(\nabla w) - \lambda \mathring{\rho} w \partial_\zeta \Psi] \, d\xi \, d\zeta - \int_{\mathcal{S}} \lambda \mathring{\rho} (\partial_\zeta w) \Psi \, d\xi \, d\zeta.
\end{split}  \label{Fw formulas} \ee
Here $r'$ is the H\"older conjugate exponent of $r$, and $\langle \cdot, \cdot \rangle$ denotes the pairing of $(W_{\textrm{per}}^{1,r'}(\mathcal{S}))^*$ with $W_{\textrm{per}}^{1,r'}(\mathcal{S})$.  We topologize $Y$ by endowing it with the norm
$$ \| u \|_{Y} := \inf\{ \| \mathcal{A} \|_{L_{\textrm{per}}^r} + \|  \mathcal{B}_1 \|_{L_{\textrm{per}}^r} + \|  \mathcal{B}_2 \|_{L_{\textrm{per}}^r} : u = \mathcal{A} + \partial_\xi \mathcal{B}_1 + \partial_\zeta \mathcal{B}_2,~\mathcal{A}, \mathcal{B}_i \in L_{\textrm{per}}^r(\mathcal{S}) \}.$$
It is elementary to see that it is a Banach space.  Moreover, the following simple technical lemma holds.

\begin{lemma} \label{compact embedding lemma} 
The space $X_1$ is compactly embedded in $Y$ in the sense that the identification mapping $I : X_1 \to Y$, defined by  
$$ \langle I(v), \Psi \rangle := \int_\mathcal{S}v \Psi \, d\xi \, d\zeta, \qquad \textrm{for all }  v \in X_1, \, \Psi \in W_{\mathrm{per}}^{1,r'}(\mathcal{S}),$$
is compact.
\end{lemma}
\begin{proof}   Since $r > 2$, we may choose $\tilde r < r$ satisfying $2\tilde{r}/(2-\tilde{r}) > r$.  Then, by the Rellich-Kondrachov theorem, $W_{\textrm{per}}^{1,{r}}(\mathcal{S}) \subset W_{\textrm{per}}^{1,\tilde{r}}(\mathcal{S}) \subset \subset L_{\textrm{per}}^r(\mathcal{S}).$
The first inclusion is simply due to the fact that we are on a compact domain.     

Now, if $\{ v_n \} \subset X_1$ is a bounded sequence, it follows that, modulo a subsequence, $v_n \to v$ in $L_{\textrm{per}}^r(\mathcal{S})$ for some $v$.  Likewise, 
$$ \| I(v_n) - I(v) \|_{Y} \leq \| v_n - v\|_{L_{\textrm{per}}^r(\mathcal{S})} \to 0.$$
Hence $\{ I(v_n) \}$ has {a} convergent subsequence, and the proof is complete.  
\end{proof}

Fix a rescaled streamline density function $\mathring{\rho}_* \in C^{1,\alpha}([-1,0])$.  By Theorem \ref{Turner main theorem}, we know that for any $R \in (0, R_{\textrm{max}})$ and $L$ sufficiently large, there exists $(w_*, \mathring{\rho}_*, \lambda_*)$ such that $w_*$ solves Problem \ref{ter-krikorov prob} with $\lambda = \lambda_*$, and streamline density function $\mathring{\rho}_*$. Moreover, $\lambda_*$ is supercritical in the sense of \eqref{def lambda crit}: $\lambda_* > \mu_* := \lambda_{\textrm{crit}}(\mathring{\rho}_*)$.  Thus, in particular,
$$ \mathcal{G}(w_*,\mathring{\rho}_*,\lambda_*) = 0.$$
We will apply the implicit function theorem in order to infer the existence of nearby solutions where the density is merely  $L^\infty$.   With that in mind, we compute that the Fr\'echet derivative $\mathcal{G}_w(w, \mathring{\rho}, \lambda) : X_1 \to Y$ applied to $u \in X_1$ and acting on a test function $\Psi$ is given by 
\be \begin{split} \label{computation of DG} 
\langle \mathcal{G}_{w}(w_*,\mathring{\rho}_*, \lambda_*) u, \Psi \rangle & = -\int_{\mathcal{S}} \mathring{\rho}_* \nabla \Psi \cdot [(D^2 f)(\nabla w_*) \nabla u ] \, d\xi \, d\zeta \\
& \qquad + \int_{\mathcal{S}} [ \lambda_* \mathring{\rho}_* u \partial_\zeta \Psi + \lambda_* \rho_* (\partial_\zeta u) \Psi ] \, d\xi \, d\zeta.
\end{split} \ee
%
%\be \label{computation of DG} \begin{split}
% \mathcal{G}_{1w}(w_*,\mathring{\rho}_*, \lambda_*) u & = \nabla \cdot \left( \mathring{\rho}_* ( D^2 f)(\nabla w_*) \nabla u \right) - \lambda_* \partial_\zeta (\mathring{\rho}_* u) + \lambda_* \mathring{\rho}_* \partial_\zeta u \\
% \mathcal{G}_{2w}(w_*, \mathring{\rho}_*, \lambda_*) u & = \left[\mathring{\rho}_* (\nabla \partial_2 f)(\nabla w_*) \cdot \nabla u - \lambda_* u \right]|_{ \{ \zeta = 0 \} }. \end{split}
%\ee
Here $D^2 f$ denotes the Hessian matrix of $f$.  Note that because $\| \nabla w_* \|_{L^\infty}$ lies below the penalization cutoff, the penalized and physical problems coincide --- hence we may use $f$ in place of $a$ above.  Observe also that, owing to the regularity of $w_*$, $D^2 f(\nabla w_*)$ is of class $C_{\textrm{per}}^{0,\alpha}(\overline{\mathcal{S}})$.  

The main lemma is the following.  

\begin{lemma}[Null space] \label{null space lemma} Let $\mathring{\rho}_* \in {C^{1,\alpha}([-1,0])}$ be a stable streamline density function.  There exists $R_0 > 0$ such that, for any non-laminar solution $(w_*, \mathring{\rho}_*, \lambda_*)$ to Problem \ref{ter-krikorov prob} with 
\bse \label{assumptions on w_*} \be \int_{\mathcal{S}} \mathring{\rho}_* \nabla f(\nabla w_*) \, dx \, d\zeta < 2R_0,\label{small amplitude assumption} \ee
and 
\be \| \nabla w_* \|_{L^\infty} < s, \label{w_* unpenalized} \ee \ese
 we have that $\ker{\mathcal{G}_w(w_*, \mathring{\rho}_*, \lambda_*)}$  is trivial.
\end{lemma}
\begin{proof}
Notice that the matrix $D^2 f(\nabla w_*)$ is a perturbation of the $2\times 2$ identity matrix.    With that in mind,  consider the following model problem:  
\begin{equation} \label{null space: model problem} \left\{ \begin{array}{ll} 
\nabla \cdot \left( \mathring{\rho}_* \nabla u \right) = \sigma  \mathring{\rho}^\prime_* u  & \textrm{in } \mathcal{S} \\
u = 0 & \textrm{on } \{ \zeta = -1\} \\
\mathring{\rho}_* \partial_\zeta u = \sigma u  & \textrm{on } \{ \zeta = 0\}. \end{array} \right. \end{equation}
In fact, \eqref{null space: model problem} is nothing but the linearization of \eqref{ter-krikorov eq2} about the trivial solution $w \equiv 0$.   Let $\mu_* := \lambda_{\textrm{crit}}(\mathring{\rho}_*)$ be given as in \eqref{def lambda crit}.  By Theorem \ref{Turner main theorem}, for $L$ sufficiently large, $R$ sufficiently small,  
\begin{equation} \label{null space: lambda < mu}  \lambda_* \leq \mu_* (1- CR^{4/3}) < \mu_*, \end{equation}
for some constant $C$ depending on $\mathring{\rho}_*$.  In particular,  $\lambda_*$ is \emph{not} a generalized eigenvalue of the model problem \eqref{null space: model problem}, and the gap between $\lambda_*$ and $\mu_*$ can be widened by taking $R \to 0$.

Now, for each $t \in [0,1]$, put $\mathbf{A}_t := D^2 f (t \nabla w)$ and define $L_t :X_1 \to Y$ by 
$$ \langle L_t u, \, \Psi \rangle :=  -\int_{\mathcal{S}} \mathring{\rho} \nabla \Psi \cdot  (\mathbf{A}_t \nabla u ) \, d\xi \, d\zeta + \int_{\mathcal{S}} [ \lambda_* \mathring{\rho}_* u \partial_\zeta \Psi + \lambda_* \rho_* (\partial_\zeta u) \Psi ] \, d\xi \, d\zeta, $$
for each test function $\Psi \in W_{\textrm{per}}^{1,r'}(\mathcal{S})$.  
%$$ L_1^t u := \nabla (\mathring{\rho} A^t \nabla u) - \lambda \partial_\zeta (\mathring{\rho} u) + \lambda \mathring{\rho} \partial_\zeta u, \qquad L_2^t u := \left[  N \cdot(\mathring{\rho}_* A^t \nabla u) - \lambda_* u \right]|_{\{ \zeta = 0 \}}, $$
%where $N = \mathbf{e}_2$ is the outward unit normal.   
It follows that $L_1 = \mathcal{G}_w(w_*, \mathring{\rho}_*, \lambda_*)$.  On the other hand, a simple computation confirms that $L_0$ corresponds to the operator associated to the model problem \eqref{null space: model problem}.
% $$ L_1^0 u = \nabla \cdot (\mathring{\rho}_* \nabla u) - \lambda_* \partial_\zeta (\mathring{\rho}_* u) + \lambda_* \mathring{\rho}_* \partial_\zeta u, \qquad L_2^0 = \left[\mathring{\rho}_* \partial_\zeta u - \lambda u\right]|_{\{ \zeta = 0\}}.$$
 
 Seeking a contradiction, suppose that $u_0 \not\equiv 0$ is an element of $\ker{\mathcal{G}_w(w_*,\mathring{\rho}_*,\lambda_*)}$.  Since $r > 2$, the H\"older conjugate $r' < 2$, and hence $W_{\textrm{per}}^{1,r}(\mathcal{S}) \subset W_{\textrm{per}}^{1,r'}(\mathcal{S})$.  We may therefore use $u_0$ as a test function to deduce that
$$ 0 = \langle L_1 u_0, u_0 \rangle = -\int_{\mathcal{S}} \mathring{\rho}_* \nabla u_0 \cdot \mathbf{A}_1 \nabla u_0 \, d\xi \, d\zeta + {\lambda_*} \int_{\{ \zeta = 0 \}} \mathring{\rho}_* u_0^2 \, d\xi - \lambda_* \int_{\mathcal{S}} \mathring{\rho}_*^\prime u_0^2 \, d\xi \, d\zeta.$$
Without loss of generality we may take 
$$ -\int_{\mathcal{S}} \mathring{\rho}_*^\prime u_0^2 \, d\xi \, d\zeta + \int_{\{ \zeta = 0\}} \mathring{\rho}_* u_0^2 \, dx  = 1,$$
and thus 
\be \lambda_* =  \int_{\mathcal{S}} \mathring{\rho}_* \nabla u_0 \cdot \mathbf{A}_1 \nabla u_0 \, d\xi \, d\zeta.\label{null space: lambda identity} \ee
Observe that by Lemma \ref{ellipticity lemma}, $\mathbf{A}_1$ is positive definite, and hence 
$ \| \nabla u_0 \|_{L_{\textrm{per}}^2(\mathcal{S})} \lesssim \lambda_*.$

Consider now the Rayleigh quotient  
$$ \mathscr{R}(t) := \int_{\mathcal{S}} \mathring{\rho}_* \nabla u_0 \cdot \mathbf{A}_t \nabla u_0 \, d\xi \, d\zeta, \qquad t \in [0,1].$$
By the definition of $u_0$ and identity \eqref{null space: lambda identity}, we have that 
\begin{equation} \label{null space: Q1 is lambda} \mathscr{R}(1) = \lambda_*. \end{equation}
Clearly $\mathscr{R}$ is $C^1([0,1])$ and 
$ \| \mathscr{R}^\prime \|_{L^\infty} \lesssim \| \nabla w_* \|_{L^\infty(\overline{\mathcal{S}})}. $
Thus, in light of \eqref{null space: Q1 is lambda} and \eqref{null space: lambda < mu}, for $\| \nabla w_* \|_{L^\infty}$ sufficiently small,  
$$  \mu_* > \mathscr{R}(0) =  \int_{\mathcal{S}} \mathring{\rho}_* |\nabla u_0|^2 \, d\xi \, d\zeta,$$
which violates the criticality of $\mu_*$.  We conclude, therefore, that for $R$ sufficiently small, the null space of $\mathcal{G}_w(w_*, \mathring{\rho}_*, \lambda_*)$ is trivial for any strong solution $(w_*, \mathring{\rho}_*, \lambda_*)$ satisfying \eqref{small amplitude assumption}.  
\end{proof}
\begin{remark}  While Turner uses a variational method, a good way to understand the above statement is through bifurcation theory.  If we imagine a curve of non-laminar solutions bifurcating from the trivial solution $(0, \lambda_{\textrm{crit}})$, then Theorem \ref{Turner main theorem} implies that this curve arcs in the direction $\lambda < \lambda_{\textrm{crit}}$.  Effectively, we are arguing that, in a sufficiently small neighborhood of the point of bifurcation, there is no secondary bifurcation, and hence the linearized operator has a trivial kernel.
\end{remark}

We are now prepared to prove a continuity result for the penalized problem.  

\begin{thm}[Continuity for penalized problem]\label{Calpha continuity theorem}  Let $\mathring{\rho}_* \in {C^{1,\alpha}([-1,0])}$ be a stable streamline density function with $\mathring{\rho}_*(0) = 1$, and let a non-laminar solution $(w_*, \mathring{\rho}_*, \lambda_*)$ of Problem \ref{ter-krikorov prob} satisfying  \eqref{assumptions on w_*}  be given.  There is a neighborhood $\mathring{\mathcal{V}} \times \Lambda$ of $(\mathring{\rho}_*, \lambda_*)$ in $X_2 \times \mathbb{R}$, and a  map $\mathcal{W} \in C^1(\mathring{\mathcal{V}} \times \Lambda; X_1)$ with  
$$ \mathcal{G}(\mathcal{W}(\mathring{\rho}, \lambda), \mathring{\rho}, \lambda) = 0, \qquad \textrm{for all $(\mathring{\rho}, \lambda) \in   \mathring{\mathcal{V}} \times \Lambda$}.$$
Moreover, the set $\{ ( \mathcal{W}(\mathring{\rho}, \lambda), \mathring{\rho}, \lambda) : (\mathring{\rho}, \lambda) \in \mathring{\mathcal{V}} \times \Lambda \}$ gives the complete zero-sets of $\mathcal{G}$ in a neighborhood of $(w_*, \mathring{\rho}_*, \lambda_*)$ in ${X} \times \mathbb{R}$.  
\end{thm}
\begin{proof}  We claim that for any $(w_*, \mathring{\rho}_*, \lambda_*)$ as above, $\mathcal{L} := \mathcal{G}_w(w_*, \mathring{\rho}_*, \lambda_*)$ is an isomorphism from $X_1$ to $Y$.  We show first that $\mathcal{L}$ is a Fredholm operator of index $0$.  Let $\sigma \geq 0$ and $u \in {X_1}$ be given, and suppose that 
$$ (\mathcal{L} - \sigma) u = \mathcal{A} + \partial_\xi \mathcal{B}_1 + \partial_\zeta \mathcal{B}_2 \in Y.$$
From \eqref{computation of DG}, we see that $u$ is weak solution of the divergence form elliptic problem 
$$ \nabla \cdot \left( \mathring{\rho}_* ( D^2 f)(\nabla w_*) \nabla u \right) - \lambda_* \partial_\zeta (\mathring{\rho}_* u) + \lambda_* \mathring{\rho}_* \partial_\zeta u - \sigma u = \mathcal{A} + \nabla \cdot \mathcal{B} \qquad \textrm{in } \mathcal{S},$$
with a co-normal boundary condition on the upper boundary 
$$ \mathring{\rho}_* (\nabla \partial_2 f)(\nabla w_*) \cdot \nabla u - \lambda_* \rho_* u = \mathcal{B}_2 \qquad \textrm{on } \{ \zeta = 0\},$$
and a homogeneous Dirichlet condition on the lower boundary $\{ \zeta = -1\}$.    Due to the smoothness of $w_*$, $D^2 f(\nabla w_*) \in C_{\textrm{per}}^{0,\alpha}(\overline{\mathcal{S}})$.  On the other hand, to emphasize the generality of this result, let us treat  $\mathring{\rho}_*$ simply as an element of $L^\infty$.  Then, the problem above represents a divergence form elliptic equation with coefficients that are bounded and measurable in the $\zeta$-direction, and H\"older continuous in the $\xi$-direction.   

There exists an elliptic regularity theory for such equations due to H. Dong and D. Kim \cite{dong2010partialBMO}.  In part, they prove that there exists a $\sigma_0 \geq 0$ such that, for any $\sigma \geq \sigma_0$, one has \emph{a priori} estimate 
\begin{align*} \sqrt{\sigma} \| \nabla u \|_{L_{\textrm{per}}^r(\mathcal{S})} + {\sigma} \| u \|_{L_{\textrm{per}}^r(\mathcal{S})} & \leq C \left(   \| \mathcal{A} \|_{L^r_{\textrm{per}}(\mathcal{S})}  +  \sqrt{\sigma} \| \mathcal{B}_1 \|_{L^r_{\textrm{per}}(\mathcal{S})} +  \sqrt{\sigma} \| \mathcal{B}_2 \|_{L^r_{\textrm{per}}(\mathcal{S})}  \right),  \end{align*}
where $C > 0$ is independent of $\sigma$ and $u$ (cf. \cite[Theorem 4 and Theorem 5]{dong2010partialBMO}).  From this it follows that 
$$ \| u \|_{X_1} = \| u \|_{W_{\textrm{per}}^{1,r}(\mathcal{S})} \leq \max\{ 1, \frac{1}{{\sigma}}\} C \| (\mathcal{L}-\sigma) u \|_{Y}.  $$
 Because $X_1$ is compactly embedded in $Y$ according to Lemma \ref{compact embedding lemma}, the inequality above implies that $\mathcal{L}-\sigma$ is semi-Fredholm for any $\sigma \geq 0$.    In fact, it is an isomorphism for $\sigma \geq \sigma_0$.  As the Fredholm index is continuous, we infer that $\mathcal{L}$ has Fredholm index $0$.
 
Now, by Lemma \ref{null space lemma}, it is already known that ${\mathcal{L}}$ is injective.  The argument above then shows that it must be an isomorphism.  The conclusion of the theorem follows from an application of the implicit function theorem.\end{proof}

Finally, we we observe that the solutions furnished by Theorem \ref{Calpha continuity theorem} inherit several key qualitative features of $w_*$.  Firstly, they are necessarily even about the $\zeta$-axis.  

\begin{lemma}[Symmetry] \label{symmetry lemma} Let $(\mathring{\rho}_*, w_*, \lambda_*)$ be given satisfying the hypotheses of Theorem \ref{Calpha continuity theorem}, and let $w = \mathcal{W}(\mathring{\rho}, \lambda)$, for some $(\mathring{\rho}, \lambda) \in \mathring{\mathcal{V}} \times \Lambda$.  Then $\xi \mapsto w(\xi, \cdot)$ is even.  
\end{lemma}
\begin{proof}  Let $T : X_1 \to X_1$ be the transformation defined by 
$$ Tw(\xi, \zeta) := w(-\xi, \zeta).$$
It is easy to see that the system \eqref{Turner penalized problem} is invariant under $T$,  thus  $w$ is in the zero-set of $\mathcal{G}(\cdot, \mathring{\rho}, \lambda)$ if and only if $Tw$ is in the zero-set.  However, Theorem \ref{Calpha continuity theorem} ensures the local uniqueness of solutions.  We may conclude, therefore, that $Tw = w$.   
\end{proof}

Even more importantly, each of these solutions is a wave of elevation.  This will be critical to some of the arguments in the next subsection.  

\begin{lemma}[Wave of elevation] \label{wave of elevation lemma} Let $(w_*, \mathring{\rho}_*, \lambda_*)$ be given as in Theorem \ref{Calpha continuity theorem}, and suppose additionally that $w_*$ is a strict wave of elevation.  There exists a neighborhood $\mathring{\mathcal{Z}} \subset \mathring{\mathcal{V}}$ of $\mathring{\rho}_*$ in $L^\infty$ such that, for any $\mathring{\rho} \in \mathring{\mathcal{Z}}$ with $\mathring{\rho}(0) = 1$ and ${\|\mathring{\rho} - \mathring{\rho}_*\|_{L^\infty}}$ sufficiently small, the corresponding solution $w = \mathcal{W}(\mathring{\rho},\lambda_*) \in W_{\mathrm{per}}^{1,r}(\mathcal{S})$ of \eqref{Turner penalized problem} is a wave of elevation:
\[ w \geq 0 \qquad \textrm{in } \mathcal{S}.\]
 \end{lemma}
\begin{proof} We will argue using a maximum principle for quasilinear elliptic equations on thin sets.  Observe that $w$ solves 
\[ \nabla \cdot (\mathring{\rho} \nabla a)(\nabla w) - \lambda \mathring{\rho}^\prime w = 0,\]
if and only if $v := -w$ solves
\[ \nabla \cdot (\mathring{\rho} \nabla b)(\nabla v) - \lambda \mathring{\rho}^\prime v = 0,\]
where
\begin{align*} (\partial_1 b)(p_1, p_2) &:= p_1 + \frac{p_1 p_2}{1-p_2} \Phi_s + \frac{(p_1^2 + p_2^2) p_1 p_2}{1-p_2} \frac{\Phi_s^\prime}{s^2}\\
(\partial_2 b)(p_1, p_2) & := p_2 + \frac{p_2^2}{1-p_2} \Phi_s + \frac{1}{2} \frac{p_1^2 +p_2^2}{(1-p_2)^2} \Phi_s + \frac{(p_1^2+p_2^2) p_2^2}{1-p_2} \frac{\Phi_s^\prime}{s^2}. \end{align*}
It is easy to see, in light of Lemma \ref{ellipticity lemma}, that for $s$ sufficiently small, this defines an elliptic problem; denote its lower ellipticity coefficient $2\theta = 1+ \mathcal{O}(s)$. We can also write this in the form 
\[ \nabla \cdot [ \mathring{\rho} \mathbf{B}(v, \nabla v)] + \lambda \mathring{\rho} w_\zeta = 0,\]
with 
\[ \mathbf{B}(\xi, \zeta, z, p_1, p_2)  := (\nabla b)(p_1, p_2) -(0, \lambda \mathring{\rho}(\zeta) z).\]  
Notice that for all $(\xi,\zeta, z, p_1, p_2) \in \mathcal{S} \times \mathbb{R}_+ \times \mathbb{R}^2$, 
\begin{align} (p_1, p_2) \cdot \mathring{\rho}(\zeta) \mathbf{B}(x, \zeta, z, p_1, p_2) &= (p_1, p_2) \cdot [\mathring{\rho}(\zeta) (\nabla b)(p_1, p_2)] - \lambda \mathring{\rho}(\zeta) z p_2 \nonumber \\
& \geq 2\theta (p_1^2 +p_2^2) -\theta p_2^2 - \frac{1}{4\theta}  \lambda^2  \| \mathring{\rho} \|_{L^\infty}^2 z^2 \nonumber \\
& \geq \theta(p_1^2 + p_2^2) - \frac{1}{4\theta}\lambda^2  \| \mathring{\rho} \|_{L^\infty}^2 z^2. \label{structure 1} \end{align}
Moreover, 
\be \lambda \mathring{\rho} w_\zeta \leq \lambda  \| \mathring{\rho} \|_{L^\infty} |w_\zeta| \leq \lambda  \| \mathring{\rho} \|_{L^\infty} |\nabla w|.\label{structure 2} \ee

Now, choose $\zeta_0 \in (-1,0)$ so that 
\[ \zeta_0 + 1< \frac{1}{3} \pi \frac{\theta^2 d}{L \lambda^2 \| \mathring{\rho} \|_{L^\infty}^2},\]
and define 
\[ \mathcal{S}_+ := \{ (\xi,\zeta) \in \mathcal{S} : \zeta \in [\zeta_0,0] \}, \qquad \mathcal{S}_- := \{ (\xi, \zeta) \in \mathcal{S} : \zeta \in [-1,\zeta_0] \}.\]
Observe that this implies the measure of $\mathcal{S}_-$ satisfies the bound
\be \left| \mathcal{S}_- \right| = 2\frac{L}{d}(\zeta_0+1)< \frac{2}{3} \pi \frac{\theta^2}{\lambda^2 \| \mathring{\rho} \|_{L^\infty}^2}. \label{measure upper bound} \ee

Consider first the situation in $\mathcal{S}_+$.  Denote
\[ m := \inf_{\mathcal{S}_+} w_* > 0.\]
By our continuity result, we know that for $\| \mathring{\rho} - \mathring{\rho}_* \|_{L^\infty}$ sufficiently small, 
\[ \| w - w_*\|_{C^{0,\alpha}(\overline{\mathcal{S}})} < \frac{m}{2},\]
and hence 
\be w \geq \frac{m}{2} > 0, \qquad \textrm{in } \mathcal{S}_+.\label{w positive in Omega+} \ee

On the other hand, this means that $v = -w$ satisfies
\[ v \leq 0 \qquad \textrm{on } \partial \mathcal{S}_-.\]
Finally, we observe that \eqref{structure 1} and \eqref{structure 2} imply that $v$ solves a divergence form quasilinear elliptic problem satisfying the structural hypotheses of \cite[Theorem 3.3.1]{pucci2007book};  the choice of $\mathcal{S}_-$ made in \eqref{measure upper bound} ensures that the thinness hypothesis of that same theorem holds.  We may therefore conclude that  $v \leq 0$ in $\mathcal{S}_-$, or, equivalently, 
\[ w \geq 0 \qquad \textrm{in } \mathcal{S}_-.\]
Taken together with \eqref{w positive in Omega+}, this fact completes the proof.  \end{proof}
\begin{remark}  In the argument above, we have ignored the right and left boundaries of $\Omega_-$ when applying the maximum principle.  This can be justified in several ways.  First,  we may view the domain $\mathcal{S}$ as $\mathbb{T}_{L/d} \times [-1,0]$, where $\mathbb{T}_{L/d} = \mathbb{R} / (L/d)\mathbb{Z}$, so that there will indeed be no horizontal boundaries.  The proof of \cite[Theorem 3.3.1]{pucci2007book} relies only on the H\"older and Poincar\'e inequalities, both of which are valid on periodic domains.  Alternatively, we can periodically extend $w$ to $\mathcal{S}_{2L} := (-4L/d, 4L/d) \times (-1,0)$, and then reconsider the problem using $u := \Psi w$ in place of $w$, where $\Psi = \Psi(\xi)$ is a cutoff function with 
\[ \Psi = 1 \textrm{ on } [-L/d, L/d], \qquad 0 \leq \Psi \leq 1, \qquad \supp{\Psi} \subset [-3L/(2d), 3L/(2d)].\]
It is easy to then see that $u$ will solve a quasilinear elliptic problem with the same structure, but will vanish on the horizontal boundaries of $\mathcal{S}_{2L-}$.
\end{remark}

\subsection{A priori estimates for the penalized problem} \label{a priori section}

In this section, we lay the groundwork needed to show that the solutions constructed in Theorem \ref{Calpha continuity theorem} are physical solutions, provided that $\mathring{\rho}$ is layer-wise smooth and sufficiently close to $\mathring{\rho}_*$ in $L^\infty$.  

The next several lemmas seek to control $w$ ands its derivatives in various norms via the energy $\| \nabla w \|_{L^2}$.   We will do this for smooth $(w, \mathring{\rho})$, but obtain bounds that are independent of $\mathring{\rho}^\prime$; a limiting argument will then allow us to conclude that the same estimates hold for the densities of interest.  The arguments we employ are all quite similar to those in Turner \cite{turner1981internal,turner1984variational}, but we have simplified them in certain places, and, crucially, we have shown that they are uniform in the number of layers.  

\begin{lemma}[$L^2$ control of $\nabla \partial_\xi w$] \label{control of nabla w_x lemma} There exists $s_1 > 0$ such that, if 
$$(w, \mathring{\rho}, \lambda) \in H^1(\mathcal{S}) \cap C^2(\mathcal{S}) \times C^1([-1,0]) \times \mathbb{R}$$
is a  solution of the penalized problem \eqref{Turner penalized problem} for $s \in (0,s_1)$, then 
\be \int_{\mathcal{S}}  | \nabla \partial_\xi w |^2 \, d\xi \, d\zeta \leq C_1 \int_{\mathcal{S}}  |\partial_\xi w|^2 \, d\xi \, d\zeta,\label{nabla w_x control} \ee
for some constant $C_1 = C_1(\lambda, {\mathring{\rho}(0)}, {\mathring{\rho}(-1)}) > 0$.
\end{lemma} 
\begin{proof}  Taking $\partial_\xi$ of \eqref{Turner penalized problem}, we see that $v := w_\xi$ satisfies 
\be 
\left\{ \begin{array}{ll} 
\nabla \cdot (\mathring{\rho} A(\nabla w) \nabla v) - \lambda \partial_\zeta ( \mathring{\rho} v) + \lambda \mathring{\rho} \partial_\zeta v = 0 & \textrm{in } \mathcal{S} \\
v = 0 & \textrm{on } \{ \zeta = -1\} \\
\mathring{\rho} A_2(\nabla w) \cdot \nabla v - \lambda v = 0 & \textrm{on } \{ \zeta = 0 \},  \end{array} \right.\label{v problem} \ee
where we are denoting $A(\nabla w) := (D^2 a)(\nabla w)$.  Taking the inner product with $v$ and integrating by parts further reveals that
\[ \int_{\mathcal{S}} \mathring{\rho} \nabla v \cdot [A(\nabla w) \nabla v] \, d\xi \, d\zeta = - \lambda \int_{\mathcal{S}} \mathring{\rho}^\prime v^2 \, d\xi \, d\zeta + \lambda \int_{\{ \zeta = 0\}} \mathring{\rho} v^2 \, d\xi.\]
From Lemma \ref{ellipticity lemma}, we know that there exists $\nu = \nu(s) > 0$ such that 
\[ \nabla v \cdot [A(\nabla w) \nabla v] \geq \nu |\nabla v|^2 \qquad \textrm{in } \mathcal{S}.  \]
We note that in fact, $\nu = \nu(s) = 1 + \mathcal{O}(s)$, and so for $s$ sufficiently small,  $\nu \in (1/2,3/2)$.
 
Continuing from above, we see that 
\begin{align*} \nu \int_{\mathcal{S}} \mathring{\rho} |\nabla v|^2 \, d\xi \, d\zeta & \leq - \lambda \int_{\mathcal{S}} \mathring{\rho}^\prime v^2 \, d\xi \, d\zeta + \lambda \int_{\{ \zeta = 0\}} \mathring{\rho} v^2 \, d\xi \\
& = 2\lambda \int_{\mathcal{S}} \mathring{\rho} v \partial_\zeta v \, d\xi \, d\zeta \\
& \leq \frac{\nu}{2} \int_{\mathcal{S}} \mathring{\rho} |\partial_\zeta v|^2 \, d\xi \, d\zeta + \frac{2\lambda^2}{\nu} \int_{\mathcal{S}} \mathring{\rho} v^2\, d\xi \, d\zeta.  \end{align*}
Recalling that $v = \partial_\xi w$, and that $\mathring{\rho}$ is nondecreasing, this immediately gives the estimate in \eqref{nabla w_x control} with
\be C_1 := \frac{4\lambda^2 \mathring{\rho}(-1)}{ \nu^2 \mathring{\rho}(0)} \leq \frac{16 \lambda^2 \mathring{\rho}(-1)}{\mathring{\rho}(0)}, \qquad \textrm{for all $s$ sufficiently small.} \label{def C1}  \ee
This completes the proof. 
\end{proof}

\begin{lemma}[H\"older continuity of $\partial_\xi w$] \label{holder control w_x lemma}   Let $(w, \mathring{\rho}, \lambda)$ be given as in the previous lemma.  Then there exists $\beta \in (0,1)$ such that  
\be \| \partial_\xi w \|_{C^{0, \beta}(\mathcal{S})} \leq C_2 \max\{ \| \nabla w \|_{L^2(\mathcal{S})}^2, \, \| \nabla w \|_{L^2(\mathcal{S})}\},\label{Holder control w_x} \ee
for a constant $C_2 = C_2(\lambda, \mathring{\rho}(-1)) > 0$. 
\end{lemma}
\begin{proof}  Denote $v := \partial_\xi w$. By the argument of the previous lemma, treating $w$ as known, we see that $v$ solves a divergence form linear elliptic problem in $\mathcal{S}$, with a co-normal boundary condition on the upper boundary and a homogeneous Dirichlet condition on the lower boundary.   We may therefore appeal to elliptic regularity theory (e.g., \cite[Theorem 8.29]{gilbarg2001elliptic}) to conclude that for some $\beta \in (0,1)$, 
\[ \| v \|_{C^{0,\beta}(\overline{\mathcal{S}})} \leq C\left( \| v \|_{L^2(\mathcal{S})} + \| \lambda \mathring{\rho} v \|_{L^4(\mathcal{S})} + \| \lambda \mathring{\rho} \partial_\zeta v \|_{L^2(\mathcal{S})}  \right),  \]
where $C = C(s, \| \mathring{\rho} \|_{L^2})$, but is independent of $s$ for $s < s_1$.  Now clearly, 
\[ \| \lambda \mathring{\rho} v\|_{L^4(\mathcal{S})} \leq \lambda \| \mathring{\rho} \|_{L^\infty(\mathcal{S})} \| v \|_{L^4(\mathcal{S})} \leq \lambda \| \mathring{\rho} \|_{L^\infty} \| v \|_{L^2(\mathcal{S})}^2. \]
Since $v = 0$ on the bottom boundary, we may apply Poincar\'e's inequality in the form 
\[ \| v \|_{L^2} \leq \frac{1}{\pi} \| \nabla v \|_{L^2},\]
and \eqref{nabla w_x control} to obtain
\begin{align*}  \| v \|_{C^{0,\beta}(\overline{\mathcal{S}})} &\leq C\left(   \| \nabla v \|_{L^2(\mathcal{S})} + \lambda \| \mathring{\rho} \|_{L^\infty}  \| \nabla v \|_{L^2(\mathcal{S})}^2 + \lambda \| \mathring{\rho} \|_{L^\infty} \| \nabla v \|_{L^2(\mathcal{S})} \right) \\
& \leq C_2 \max\left\{ \int_{\mathcal{S}} | \partial_\xi w|^2 \, d\xi \, d\zeta, \, \left[\int_{\mathcal{S}} | \partial_\xi w|^2 \, d\xi \, d\zeta\right]^{1/2}\right\}.  \end{align*}
Here $C_2 = C_2(\lambda, \| \mathring{\rho} \|_{L^\infty}).$   This implies inequality \eqref{Holder control w_x}, and hence the lemma is complete.  
\end{proof}

\begin{lemma}[$L^\infty$ control of $\partial_\zeta w$] \label{preliminary control of w_eta lemma} Let $(w, \mathring{\rho}, \lambda)$ be given as in the previous lemma.  Then 
\be \| w_\zeta \|_{L^\infty(\mathcal{S})} \leq \sqrt{2} s + C_3 \max\{  \| \nabla w \|_{L^2(\mathcal{S})},  \| \nabla w \|_{L^2(\mathcal{S})}^2 \} \label{w_eta prelim estimate} \ee
where $C_3 = C_3(\lambda, \mathring{\rho}(0), \mathring{\rho}(-1)) > 0.$ 
\end{lemma}
\begin{proof}  Let $\sigma \in (0,1)$ be given, and consider $(x_0, \zeta_0) \in \mathcal{S}$ with $\zeta_0 < -\sigma$.  Our argument will closely follow that given in \cite[Lemma 3.3]{turner1981internal} and \cite[Lemma 3.4]{turner1984variational}.   We work in the periodic strip 
\[ \mathcal{S}_0 := \{ (\xi,\zeta) \in \mathcal{S} : \zeta \in [\zeta_0, 0] \}. \]
For notational simplicity, let us change coordinates to 
\[ (\xi, \zeta) \mapsto ( \hat \xi, \hat \zeta ) := (\xi - \xi_0, \zeta - \zeta_0), \]
which has the effect of translating $(\xi_0, \zeta_0)$ to the origin in the $(\hat \xi, \hat \zeta)$-variables.  Likewise $\mathcal{S}_0$ becomes 
\[ \widehat{\mathcal{S}}_0 := \{ (\hat \xi, \hat \zeta) \in \mathbb{T}_{L/d} \times [0, |\zeta_0|] \} \supset \mathbb{T}_k \times \mathbb{T}_{L/d} \times [0, \sigma].\]
Putting
\[ \hat {\rho}(\hat \zeta) = \mathring{\rho}(\zeta), \qquad \hat w(\hat \xi, \hat \zeta) := w(\xi, \zeta),\]
we find that $\hat w$ solves the quasilinear equation 
\[ \mathcal{Q}(\nabla \hat w) := \nabla \cdot [ \hat{{\rho}} (\nabla a)(\nabla \hat{w}) ] = \lambda \hat{{\rho}}^\prime \hat{w}, \qquad \textrm{in } \widehat{\mathcal{S}}_0.\]
In light of Lemma \ref{wave of elevation lemma}, moreover, we see that
\[ \mathcal{Q}(\nabla \hat w) \leq 0, \qquad \textrm{in } \widehat{\mathcal{S}}_0.\]
This is the most critical usage of the fact that $w$ is a wave of elevation.  The strategy is now to use a comparison argument:  we will construct a function $u$ with $\mathcal{Q}(u) \geq \mathcal{Q}(w)$ in $\widehat{\mathcal{S}}_0$, and $u \leq w$ on $\partial \widehat{\mathcal{S}}_0$.  With that in mind, consider a function $u = u(\hat \xi, \hat \zeta)$ of the form 
\[ u(\hat \xi, \hat \zeta) = \hat w(0,0) + U_1 \hat \xi + U_2 \hat \zeta + \frac{1}{2} U_3 \frac{\hat \zeta^2}{|\zeta_0|} +  U_\beta H(\hat \xi, \hat \zeta),\]
where $U_1 := (\partial_\xi w)(\xi_0, \zeta_0)$,  $\beta$ is given as in the previous lemma, the constants $U_2, U_3,  U_\beta \in \mathbb{R}$ are to be determined, and $H$ is a harmonic function of the form
\[ H(\hat \xi, \hat \zeta) :=  \realpart{\hat{z}^{1+\beta}}, \qquad \hat z := \hat \zeta + i\hat \xi. \]  

We compute that 
\begin{align*} \partial_{\hat \xi} H &= (1+\beta) |z|^{\beta-1}\left[ \hat \xi \cos{((1+\beta)\theta)} - \hat \zeta \sin{((1+\beta)\theta)} \right] \\
\partial_{\hat \zeta} H &= (1+\beta)|z|^{\beta-1}\left[ \hat \zeta \cos{((1+\beta)\theta)} - \hat \xi \sin{((1+\beta)\theta)} \right], \end{align*}
where $\theta := \arg{\hat z}$.  A simple estimate is thus 
\[ | \partial_{\hat \zeta} H | \leq (1+\beta) \sup_{\widehat{\mathcal{S}}_0} \left[ \hat \xi^2 + \hat \zeta^2\right]^{\beta/2}  \leq (1+\beta) (\zeta_0^2+(L/d)^2)^{\beta/2}. \]
We see then that, if
\be |U_2| - |U_3| - (1+\beta) |U_\beta| (\zeta_0^2+(L/d)^2)^{\beta/2} > \sqrt{2} s, \label{comparison requirement 1} \ee
then 
\[ \frac{1}{2} |\nabla u|^2 \geq \frac{1}{2} |\partial_{\hat \zeta} u|^2 > s^2, \qquad \textrm{in } \widehat{\mathcal{S}}_0, \]
and thus it falls into the penalized region.  Consequently, 
\[ \mathcal{Q}(\nabla u) = \nabla \cdot (\hat{\rho} \nabla u) =  \hat{\rho} \Delta u + \hat{\rho}^{\prime} \partial_{\hat \zeta} u = \hat{\rho}^\prime \left( U_2+U_3 \frac{\hat \zeta}{|\zeta_0|} + U_\beta \partial_{\hat \zeta} H \right).  \]
Here we have used the fact that $\Delta u = 0$.  From the line above, we see that $u-\hat w$ satisfies 
\begin{align*} \mathcal{Q}(u) - \mathcal{Q}(\hat w) &= \hat{\rho}^\prime \left( U_2 + U_3 \frac{\hat \zeta}{|\zeta_0|} + U_\beta \partial_{\hat \zeta} H   - \lambda w \right) \\
& \geq \hat{\rho}^\prime \left( U_2 + \max\{ 0, U_3\} + |U_\beta| (1+\beta)(\zeta_0^2+(L/d)^2)^{\beta/2} +   \lambda \| w \|_{L^\infty(\mathcal{S})}  \right). \end{align*}
Since $\hat{\rho}^\prime \leq 0$, in order for this to give the desired inequality, we must have that the parenthetical quantity is nonpositive:
\be  U_2 + \max\{ U_3, 0\} + |U_\beta| (1+\beta)(\zeta_0^2+(L/d)^2)^{\beta/2} +    \lambda \| w \|_{L^\infty(\mathcal{S})} \leq 0. \label{comparison requirement 2} \ee

The above arguments show that for any selection of $U_2, U_3, U_\beta$ satisfying \eqref{comparison requirement 1}--\eqref{comparison requirement 2}, $\mathcal{Q}(\hat w) \leq \mathcal{Q}(u)$ in $\widehat{\mathcal{S}}_0$.  To complete the comparison principle argument, we must ensure that $u \leq \hat w$ on the boundary of this region.  First consider the lower boundary portion  $\{\hat \zeta = 0\}$.    From \eqref{Holder control w_x}, we know that $ {\partial_{\hat \xi} \hat w} \in C^{0,\beta}(\widehat{\mathcal{S}}_0)$, and thus 
\[ | \partial_{\hat \xi} \hat w(\hat \xi, 0) - \partial_{\hat \xi}  \hat w(0,0)| \leq [ \partial_{\hat \xi}  \hat w ]_{0,\beta} |\hat \xi|.\]
On the other hand, by construction
\begin{align*}
 \partial_{\hat \xi} u(\hat \xi, 0) &= U_1 + U_\beta \partial_{\hat \xi} H(\hat \xi, 0) \\
 &= \partial_{\hat \xi} \hat w(0,0) +  U_\beta (1+\beta)  \hat \xi |\hat{\xi}|^{\beta-1} \cos{((1+\beta) \frac{\pi}{2})}.  \end{align*}
 Here we have used the fact that $\theta = \arg{[i \hat \xi]} = (\pi/2) \sgn{\hat \xi}.$  Combining these observations, we see that for $\hat \xi \geq 0$, 
 \begin{align*}
 \partial_{\hat \xi} u(\hat \xi, 0) - \partial_{\hat \xi} w(\hat \xi, 0)  & \leq \partial_{\hat \xi} u(\hat \xi, 0) - \partial_{\hat \xi} w(0, 0) + [ \partial_{\hat \xi} \hat w]_{0,\beta} |\hat{\xi}|^\beta \\
 & = \left[ U_\beta (1+\beta) \cos{((1+\beta) \frac{\pi}{2})} + [ \partial_{\hat \xi} \hat w]_{0,\beta} \right ] |\hat \xi|^\beta,  \end{align*}
 whereas, for $\hat \xi \leq 0$,
 \[ \partial_{\hat \xi} u(\hat \xi, 0) - \partial_{\hat \xi} w(\hat \xi, 0)  \geq \left[ -U_\beta (1+\beta) \cos{((1+\beta) \frac{\pi}{2})} - [ \partial_{\hat \xi} \hat w]_{0,\beta} \right ] |\hat \xi|^\beta.\]
 Thus, if $U_\beta$ is selected with 
 \be U_\beta \geq \frac{[\partial_{\hat \xi} \hat w]_{0,\beta}}{(1+\beta)| \cos{((1+\beta) \frac{\pi}{2})} |},\label{comparison requirement 3} \ee
 then
 \[ \partial_{\hat \xi} u \geq \partial_{\hat \xi} \hat w \textrm{ for } \hat \xi \leq 0, \qquad \textrm{and} \qquad \partial_{\hat \xi} u \leq \partial_{\hat \xi} \hat w \textrm{ for } \hat \xi \geq 0.\]
 As $u(0,0) = w(0,0)$, these imply that $u \leq \hat w$ on $\{ \hat \zeta = 0 \}$.  
 
Next consider the sides of $\widehat{\mathcal{S}}_0$ where $\hat \xi = \pm L/d$.  There we note,
 \[ |H(\pm L/d, \hat \zeta)| =  ( (L/d)^2 + \hat \zeta ^2)^{(1+\beta)/2} | \cos{((1+\beta) \arctan{(\frac{L^2}{\hat \zeta d^2})})} |.  \]
 Therefore, $\hat\zeta \mapsto H(\pm L/d, \hat \zeta)$ is an increasing function, for $\zeta$ sufficiently small, and 
 \[ H(\pm L/d, 0) =  (L/d)^{1+\beta} \cos{((1+\beta)\frac{\pi}{2})} < 0.\]
 We may choose $\epsilon= \epsilon( \beta, L/d) \in (0,\sigma)$ such that 
 \[ H(\pm L, \hat \zeta) \leq \frac{1}{2}(L/d)^{1+\beta} \cos{((1+\beta)\frac{\pi}{2})}   < 0, \qquad \textrm{for all } \hat \zeta \in [0,\epsilon].\]
 Let $\widehat{\mathcal{S}}_{0\epsilon} := \{ (\hat\xi, \hat\zeta)  \in \widehat{\mathcal{S}}_0 : \zeta \in (0, \epsilon)\}$ denote the corresponding subdomain.  Note that this is precisely the reason we consider separately the cases where $\zeta_0 < -\sigma$ and $\zeta_0 \geq - \sigma$:  in general, $\epsilon$ will vanish as we approach the top of the fluid domain $\{ \zeta = 0\}$.  
  
 Now, on the sides of $\widehat{\mathcal{S}}_{0\epsilon}$ we have
 \be \hat w - u \geq \hat w - \hat w(0,0) \mp U_1 \frac{L}{d} - U_2 \hat \zeta - \frac{1}{2} U_3 \frac{\hat \zeta^2}{|\zeta_0|} - U_\beta H, \qquad \textrm{on } \{ \hat \xi = \pm L/d \}.\label{comparison sides} \ee 
 Since $\hat w \geq 0$ by Lemma \ref{wave of elevation lemma}, if we require that 
 \be U_2, U_3 \leq 0 \label{comparison requirement 4} \ee
 then this becomes
 \[ \hat w - u \geq -|U_1| (L/d) - U_\beta h - \| w \|_{L^\infty}, \qquad \textrm{on } \{ \hat \xi = \pm L/d \}.\]
 We infer that, for 
 \be U_\beta \geq \frac{2k |U_1| + 2\| w \|_{L^\infty}}{(L/d)^{1+\beta} \cos{((1+\beta)\frac{\pi}{2})}},\label{comparison requirement 5} \ee
 one has $w \geq u$ on the horizontal boundary portion of $\widehat{\mathcal{S}}_{0\epsilon}$.
 
 Finally, consider the top of $\widehat{\mathcal{S}}_{0\epsilon}$.  
 \begin{align*}
 \hat w - u & \geq \hat w - \hat w(0,0) - U_1 \hat \xi - U_2 \epsilon - \frac{1}{2} U_3 \frac{\epsilon^2}{|\zeta_0|}  - U_\beta ((L/d)^2 + \epsilon^2)^{(\beta+1)/2} \\
 & \geq -2 \| w \|_{L^\infty} - U_2 \epsilon - U_\beta ((L/d)^2 + \epsilon^2)^{(\beta+1)/2}, \qquad \textrm{on } \{ \hat \zeta = \epsilon \}. \end{align*}
 Taking
 \be |U_2| = -U_2 \geq \frac{2 \| w \|_{L^\infty} + U_\beta ((L/d)^2+\epsilon^2)^{(\beta+1)/2}}{\epsilon} \label{comparison requirement 6} \ee
 ensures that $w \geq u$ on the upper boundary portion of $\widehat{\mathcal{S}}_{0\epsilon}$.  
 
Collecting these statements together, we have proved the following:  let $U_3 := 0$, and define $U_\beta = U_\beta(\beta, \| \partial_\xi w \|_{C^{0,\beta}}, \| w \|_{L^\infty}, L/d)$ by
\[ U_\beta := \max\left\{\frac{[\partial_{\hat \xi} \hat w]_{0,\beta}}{(1+\beta)| \cos{((1+\beta) \frac{\pi}{2})} |} , \, \frac{2(1+L/d)\| w \|_{L^\infty}}{(L/d)^{1+\beta} \cos{((1+\beta)\frac{\pi}{2})}}  \right\}.  \]
This guarantees that \eqref{comparison requirement 3} and \eqref{comparison requirement 5} are satisfied.  Set $U_2 = U_2(\beta, \| w \|_{L^\infty} \| \partial_\xi w \|_{C^{0,\beta}}, \lambda, L/d, s)$ to be
\begin{align*} |U_2| = -U_2 & := \max\Big\{ \sqrt{2}s + (1+\beta)|U_\beta|(1+(L/d)^2)^{\beta/2}, \, (1+\beta)|U_\beta|(1+(L/d)^2)^{\beta/2} + \lambda \| w \|_{L^\infty},  \\
& \qquad\qquad\qquad \frac{2 \| w \|_{L^\infty} + U_\beta ((L/d)^2+\epsilon^2)^{(\beta+1)/2}}{\epsilon} \Big\}, \end{align*}
ensuring that \eqref{comparison requirement 1}--\eqref{comparison requirement 2}, \eqref{comparison requirement 4}, and \eqref{comparison requirement 6} hold.  Then $u \leq \hat w$ on $\partial \widehat{\mathcal{S}}_{0\epsilon}$, while $\mathcal{Q}(\hat w) \leq 0 \leq \mathcal{Q}(u)$ in $\widehat{\mathcal{S}}_{0\epsilon}$.  Applying the quasilinear comparison principle \cite[Theorem 10.7]{gilbarg2001elliptic}, we conclude that $u \leq \hat w$ in $\widehat{\Omega}_{0\epsilon}$.  But then, since $u(0,0) = \hat w(0,0)$, we must have that 
\[ \partial_{\hat \zeta} \hat w(0,0) \geq \partial_{\hat \zeta} u(0,0) = U_2.\]
Repeating the above argument with $-\hat w$ in place of $\hat w$, we find likewise that
\[ \partial_{\hat \zeta} \hat w(0,0) \leq |U_2|.\]
From Lemma \ref{holder control w_x lemma}, we see that $\| w \|_{C^{0,\beta}}$ is controlled by $\| \nabla w \|_{L^2(\mathcal{S})}$ and $\lambda$, $\mathring{\rho}(0), \mathring{\rho}(-1)$.  Therefore, 
\[ | w_\zeta | \leq \sqrt{2} s + C_3 \max\{  \| \nabla w \|_{L^2(\mathcal{S})},  \| \nabla w \|_{L^2(\mathcal{S})}^2 \}, \qquad \textrm{for } \zeta \in [-1, -\sigma] \]
where $C_3 = C_3(\sigma, \lambda, L/d, \mathring{\rho}(0), \mathring{\rho}(-1)).$  

Next consider the remainder of the domain where $\zeta \in [-\sigma, 0]$.  Let $(\xi_0, \zeta_0)$ be a point in this subdomain, and change coordinates as before.  We use the same form of comparison function $u$, and in fact take $U_2$ and $U_\beta$ exactly as above.    Then $\| \nabla u \| \geq \sqrt{2} s$, and $\mathcal{Q}(u) \geq \mathcal{Q}(\hat w)$ in $\widehat{\mathcal{S}}_0$.     We must show that $u \leq w$ on the entire boundary, $\partial \widehat{\mathcal{S}}_0$.   For the bottom boundary portion, this works exactly as before.  Moreover, using the crude estimate 
\[ |H(\pm L/d, \hat \zeta)| \leq ((L/d)^2+\hat \zeta^2)^{(\beta+1)/2}, \qquad \textrm{on } \{ \hat \xi = \pm L/d \}, \]
we see from \eqref{comparison sides} that, if
\[ |U_3| = -U_3 \geq 2 \| w \|_{L^\infty} + (L/d) |U_1| + |U_\beta| ((L/d)^2+ 1)^{(\beta+1)/2} ,\]
then 
\[ \hat w - u \geq 0, \qquad \textrm{on } \{ (\pm L/d, \hat \zeta) : \hat \zeta \in [ 0, |\zeta_0| ] \}.\]

Lastly, on the free surface $\{ \hat \zeta = |\zeta_0|\}$, we have
\begin{align*} \hat w - u &= \hat w - \hat w(0,0) - U_1 \hat x - U_2 |\zeta_0| - \frac{1}{2} U_3 |\zeta_0| - U_\beta H \\
& \geq -2 \| w \|_{L^\infty} - (L/d) |U_1| + |U_2| \sigma + \frac{1}{2} |U_3| \sigma - U_\beta ((L/d)^2 + \sigma^2)^{(\beta+1)/2} \qquad \textrm{on } \{ \hat \zeta = |\zeta_0| \}.\end{align*}
We have proved, therefore, that for $U_2, U_\beta$ defined as above, and with 
\[ |U_3| = -U_3 := \frac{4\| w \|_{L^\infty} + 2  (L/d) \| \partial_\xi w\|_{L^\infty} + 2|U_\beta| ((L/d)^2+\sigma^2)^{(\beta+1)/2}}{\sigma},\]
it holds that  $u \leq w$ on $\partial \widehat{\Omega}_0$.  This implies, along the same lines as above, that we have 
\[ \| \partial_\zeta w \|_{L^\infty} < \sqrt{2} s + C_3 \max\{  \| \nabla w \|_{L^2(\mathcal{S})},  \| \nabla w \|_{L^2(\mathcal{S})}^2 \},   \]
where $C_3 = C_3(\lambda, L/d, \mathring{\rho}(0), \mathring{\rho}(-1))$.
 \end{proof}
 
 \begin{lemma}[Control of $\nabla \partial_\xi^2 w$] \label{control of nabla w_xx lemma} There {exist} $s_2 > 0$ and $\gamma \in (0,1)$ such that, if 
 $$(w, \mathring{\rho}, \lambda) \in C_{\textrm{per}}^3(\overline{S}) \times C^1([-1,0]) \times \mathbb{R} $$ 
 is a solution to the penalized problem \eqref{Turner penalized problem} with $s \in (0,s_2)$, then 
\be  \|\nabla \partial_\xi^2 w \|^2_{L^2(\mathcal{S})} \leq C_4 \|\nabla w\|^2_{L^2(\mathcal{S})}, \label{w_xx H1 esimate} \ee
and
\be \| \partial_\xi^2 w \|_{C^{0, \gamma}(\mathcal{S})} \leq C_4 \max\{ \| \nabla w \|_{L^2(\mathcal{S})}^2, \, \| \nabla w \|_{L^2(\mathcal{S})}\},  \label{w_xx holder estimate} \ee
where $C_4 = C_4(\lambda, \mathring{\rho}(0), \mathring{\rho}(-1))$.
 \end{lemma}
 \begin{proof}  The first statement \eqref{w_xx H1 esimate} follows exactly as in \cite[Lemma 3.4]{turner1981internal}, as the argument there does not involve any reference to the width of the layers of the limiting rescaled streamline density.  The second statement \eqref{w_xx holder estimate} is then proved by repeating the arguments leading to Lemma \ref{holder control w_x lemma}:  notice that $w_{\xi\xi}$ also solves a linear divergence form elliptic PDE.  By appealing to the same \emph{a priori} estimates, we see that it can be controlled in terms of $\nabla w_{\xi\xi}$, which can then be estimated by $ \nabla w$ according to the first statement. 
 \end{proof}

In the next lemma, we improve the $L^\infty$ bound of $w_\zeta$ found in Lemma \ref{preliminary control of w_eta lemma} by removing its dependence on the penalization parameter $s$.

\begin{lemma}[Improved $L^\infty$ control of $w_\zeta$]\label{control of w_eta in terms of R lemma}  There exists $s_3 > 0$ such that, if $(w, \mathring{\rho}, \lambda)$ has the regularity dictated in Lemma \ref{control of nabla w_xx lemma} and solves the penalized problem \eqref{Turner penalized problem} for $s \leq s_3$, then,
\be\label{Linfty bound of w_eta}
\|w_\zeta\|_{L^\infty(\mathcal{S})} \leq C_5 \left( \|\nabla w\|_{L^2(\mathcal{S})} + \|\nabla w\|^2_{L^2(\mathcal{S})} + \|\nabla w\|^4_{L^2(\mathcal{S})} \right) \ee
where $C_5 = C_5(\lambda, \mathring{\rho}(0), \mathring{\rho}(-d)) > 0.$ 
\end{lemma}
\begin{proof}
First, we prove an anisotropic Sobolev-type inequality.  Let $\tilde{w}$ be an extension of $w$ to $\mathbb{R}^2$.  By a proper choice of cut-off function, and using the Sobolev extension theorem, we can arrange it so that 
$$
\|\tilde{w}\|_{L^\infty(\mathbb{R}^2)} = \|w\|_{L^\infty(\mathcal{S})}, \ \|\nabla \tilde{w}\|_{L^2(\mathbb{R}^2)} \leq C \|\nabla w\|_{L^2(\mathcal{S})}, \ \|\nabla \tilde{w}_x\|_{L^2(\mathbb{R}^2)} \leq C \|\nabla w_x\|_{L^2(\mathcal{S})}.
$$
On the other hand we have that for any $f = f(x,y) \in H^1(\mathbb{R}^2)$ with $\nabla f_x \in L^2(\mathbb{R}^2)$, 
\begin{align*}
\|\hat{f}\|_{L^1(\mathbb{R}^2)} & \leq \left( \int_{\mathbb{R}^2} \left[ 1 + (\xi_1^2 + \xi_2^2) (1+\xi_1^2) \right] |\hat{f}|^2\ d\xi_1d\xi_2 \right)^{1/2} \left( \int_{\mathbb{R}^2} {1\over 1 + (\xi_1^2 + \xi_2^2) (1+\xi_1^2)} \ d\xi_1d\xi_2\right)^{1/2}\\
& \leq C \left( \|f\|_{L^2(\mathbb{R}^2)} + \|\nabla f\|_{L^2(\mathbb{R}^2)} + \|\nabla f_x\|_{L^2(\mathbb{R}^2)} \right).
\end{align*}
Applying this reasoning to $\tilde w$ and restricting the domain to $\mathcal{S}$ we obtain
\begin{equation*}
\|w\|_{L^\infty(\mathcal{S})} \leq C \left( \|w\|_{L^2(\mathcal{S})} + \|\nabla w\|_{L^2(\mathcal{S})} + \|\nabla w_\xi\|_{L^2(\mathcal{S})} \right)
\end{equation*}
which, by a simple use of Poincar\'e inequality, is reduced to 
\begin{equation*}
\|w\|_{L^\infty(\mathcal{S})} \leq C \left( \|\nabla w\|_{L^2(\mathcal{S})} + \|\nabla w_\xi\|_{L^2(\mathcal{S})} \right).
\end{equation*}
Then using \eqref{nabla w_x control} we conclude that
\begin{equation}\label{w control}
\|w\|_{L^\infty(\mathcal{S})} \leq C_0  \|\nabla w\|_{L^2(\mathcal{S})} 
\end{equation}
where $C_0 = C_0(\lambda, \mathring{\rho}(0),  \mathring{\rho}(-d))$.

%{\bf Alternatively}, one can also apply elliptic estimates to $w$ from the equation, like in [Lemma 3.1 of Turner's variational] paper. But that way the constant in the estimate may depend on $k$. However in the above argument the constant does not depend on $k$ since the anisotropic Sobolev estimate does not involve $k$. 

Now consider $w$ on some vertical line $\{\xi = t\}$ in $\mathcal{S}$. Choose two points $(t, \tilde\zeta), (t, \tilde\zeta + d/2) \in \mathcal{S}$. Then from the mean value theorem, there exists some point $\zeta_0 \in (\tilde\zeta, \tilde\zeta + d/2)$ such that
\[
w(t, \tilde\zeta + {1\over2}) - w(t, \tilde\zeta) = {1\over 2} w_\zeta(t, \zeta_0).
\]
Therefore from \eqref{w control} we know that
\begin{equation}\label{w_eta one point}
|w_\zeta(t, \zeta_0)| \leq 4 C_0 \|\nabla w\|_{L^2(\mathcal{S})}
\end{equation}

Fix $\delta > 0$ such that $(t, \zeta_0 + \delta)$ in $\mathcal{S}$.  We integrate equation \eqref{Turner penalized problem} from $(t, \zeta_0)$ to $(t, \zeta_0 + \delta)$ to obtain
\begin{equation}\label{integral eqn}
\begin{split}
\int^{\zeta_0+\delta}_{\zeta_0} \mathring{\rho}(\zeta) \partial_\xi (a_1(\nabla w(t, \zeta))) \ d\zeta + \mathring{\rho}(\zeta) a_2(\nabla w(t, \zeta)) |^{\zeta_0 + \delta}_{\zeta_0} & = \lambda \mathring{\rho}(\zeta) w(t, \zeta)|^{\zeta_0 + \delta}_{\zeta_0} \\
&  \qquad - \lambda \int^{\zeta_0 + \delta}_{\zeta_0} \mathring{\rho}(\zeta)w_\zeta(t, \zeta) \ d\zeta.
\end{split}
\end{equation}
Note the we are working in a periodic setting, $w$, and hence $a(\nabla w)$, is periodic in $\xi$. Therefore, for each $\zeta\in [-1, 0]$, there exists an $\xi_0 \in [-L/d, L/d]$ such that $\partial_\xi (a_1(\nabla w(\xi_0, \zeta))) = 0$. So the first term on the left-hand side can be estimated as follows.
\begin{align*}
 \int^{\zeta_0+\delta}_{\zeta_0} \mathring{\rho}(\zeta) \partial_\xi (a_1(\nabla w(t, \zeta))) \ d\zeta  & =  \int^{\zeta_0+\delta}_{\zeta_0} \int^t_{\xi_0} \mathring{\rho}(\zeta) \partial^2_\xi (a_1(\nabla w(\xi, \zeta))) \, d\xi \,  d\zeta    \\
& = \int^{\zeta_0+\delta}_{\zeta_0} \int^t_{\xi_0} \mathring{\rho}(\zeta) \left[ \nabla a_1 \cdot \nabla w_{\xi\xi} + (\nabla w_\xi)^T \nabla^2 a_1 \nabla w_{\xi} \right] \ d\xi \,  d\zeta.   
\end{align*}
Recall that Lemma \ref{ellipticity lemma} states that 
$$ |\nabla a_1|, |\nabla^2 a_1| \leq C + \mathcal{O}(s). $$
With that in mind, we denote $v = w_\xi$, and continue the estimate to find 
\begin{equation}\label{intest0}
\begin{split}
\left|  \int^{\zeta_0+\delta}_{\zeta_0} \mathring{\rho}(\zeta) \partial_\xi (a_1(\nabla w(t, \zeta))) \ d\zeta  \right| & \leq (C + \mathcal{O}(s)) \|\mathring{\rho}\|_{L^\infty}  \int^{\zeta_0+\delta}_{\zeta_0} \int^t_{-\frac{L}{d}} |\nabla v_\xi| + |\nabla v|^2 \ d\xi d\zeta \\
& \leq (C + \mathcal{O}(s)) \|\mathring{\rho}\|_{L^\infty} \left( \delta^{1/2} \|\nabla {v_\xi}\|_{L^2(\mathcal{S})} + \|\nabla v\|^2_{L^2(\mathcal{S})} \right).
\end{split}
\end{equation}
Then Lemma \ref{control of nabla w_xx lemma} and \eqref{nabla w_x control} furnish the bound
\begin{equation}\label{intest1}
\left|  \int^{\zeta_0+\delta}_{\zeta_0} \mathring{\rho}(\zeta) \partial_\xi (a_1(\nabla w(t, \zeta))) \ d\zeta  \right| \leq C(1 + \mathcal{O}(s)) \|\mathring{\rho}\|_{L^\infty} \left( \delta^{1/2} \|\nabla w\|_{L^2(\mathcal{S})} + \|\nabla w\|^2_{L^2(\mathcal{S})} \right),
\end{equation}
where $C = C(\lambda,  \mathring{\rho}(0), \mathring{\rho}(-1))$.

The second term on the right-hand side of \eqref{integral eqn} can be estimated a similar way:
\begin{equation}\label{intest2}
\left| \lambda \int^{\zeta_0 + \delta}_{\zeta_0} \mathring{\rho}(\zeta)w_\zeta(t, \zeta) \ d\zeta\right| \leq C_1 \lambda\|\mathring{\rho}\|_{L^\infty} \delta^{1/2} \|\nabla w\|_{L^2(\mathcal{S})}.
\end{equation}
The first term on the right-hand side of \eqref{integral eqn} is controlled via \eqref{w control}:
\begin{equation}\label{intest3}
\left|\lambda \mathring{\rho}(\zeta) w(t, \zeta)|^{\zeta_0 + \delta}_{\zeta_0} \right| \leq 2C_0\lambda \|\mathring{\rho}\|_{L^\infty}  \|\nabla w\|_{L^2(\mathcal{S})}.
\end{equation}

Combining \eqref{intest1}--\eqref{intest3}, we infer that
\begin{equation}\label{intest4}
\left| a_2(\nabla w(s, \zeta)) |^{\zeta_0 + \delta}_{\zeta_0} \right| \leq C (1 + \mathcal{O}(s)) \left[ (1+ \delta^{1/2}) \|\nabla w\|_{L^2(\mathcal{S})} + \|\nabla w\|^2_{L^2(\mathcal{S})} \right],
\end{equation}
where $C = C(\lambda,  \mathring{\rho}(0), \mathring{\rho}(-1))$.

From Lemma \ref{ellipticity lemma} we see that 
$$ a_2(p_1, p_2; s) = p_2(1+\mathcal{O}(s)) + \mathcal{O}(p_1^2),$$
and hence we may choose $s_0$ small enough so that for $s\leq s_0$, one has $|\mathcal{O}(s)| < 1/2$. In \eqref{intest4} we are evaluating $a_2$ at  $p_1 = w_\xi(t, \zeta_0 +\delta)$, and $p_1 = w_\xi(t, \zeta_0)$.  But, from \eqref{Holder control w_x} we know that
\[
\|w_\xi\|_{L^\infty(\mathcal{S})} \leq C_2 \max\{ \| \nabla w \|_{L^2(\mathcal{S})}^2, \, \| \nabla w \|_{L^2(\mathcal{S})}\}.
\]
Therefore,
\begin{equation}\label{intest5}
\left| w_\zeta(t, \zeta_0 +\delta) - w_\zeta(t, \zeta_0) \right| \leq C \left( \|\nabla w\|_{L^2(\mathcal{S})} + \|\nabla w\|^2_{L^2(\mathcal{S})} + \|\nabla w\|^4_{L^2(\mathcal{S})} \right),
\end{equation}
where $C = C(\lambda,  \mathring{\rho}(0), \mathring{\rho}(-d))$. This, together with \eqref{w_eta one point}, implies that
\begin{equation*}
\|w_\zeta\|_{L^\infty(\mathcal{S})} \leq C \left( \|\nabla w\|_{L^2(\mathcal{S})} + \|\nabla w\|^2_{L^2(\mathcal{S})} + \|\nabla w\|^4_{L^2(\mathcal{S})} \right).
\end{equation*}
The proof of the lemma is complete.
\end{proof}

\subsection{Proof of continuous dependence for the Ter-Krikorov problem} \label{continuity proof section}
With the \emph{a priori} estimates established in \S\ref{a priori section}, we are now in a position to prove our main result of this section.  

\begin{proof}[Proof of Theorem \ref{main continuity theorem}]
Put $S_{\textrm{max}} := \min\{ s_0, s_1, s_2, s_3\}$, where $s_i$ is given as in Lemma \ref{ellipticity lemma}, Lemma \ref{control of nabla w_x lemma}, Lemma \ref{control of nabla w_xx lemma}, and Lemma \ref{control of w_eta in terms of R lemma}.  Let $(\mathring{\rho}_* , w_*, \lambda_*)$ given as in the statement of the theorem.  Then, $w_*$ satisfies \eqref{assumptions on w_*}, and hence Theorem \ref{Calpha continuity theorem} may be applied. 

By Lemma \ref{wave of elevation lemma}, there exists a neighborhood $\mathring{\mathcal{Z}} \times \Lambda$ of $(\mathring{\rho}_*, \lambda_*)$ in $L^\infty([-1,0]) \times \mathbb{R}$ such that, for any $(\mathring{\rho},\lambda) \in \mathring{\mathcal{Z}} \times \Lambda$ with $\mathring{\rho}(-1) \leq \mathring{\rho}_{\max}$, $\mathcal{W}(\mathring{\rho}, \lambda)$ is a wave of elevation.  Let  $\mathring{\rho}$ be any such density, and assume that it has the additional regularity 
$$ \mathring{\rho} \in W^{1,\infty}([-1, \zeta_1]) \cap \cdots \cap W^{1,\infty}([\zeta_{N-1}, 0]).$$
Denote $w := \mathcal{W}(\mathring{\rho}, \lambda)$.

 We may let $\{\mathring{\rho}_n\}$ be a sequence of $C^{2,\alpha}([-1,0])$ rescaled streamline density functions with 
 $$ \mathring{\rho}_n \to \mathring{\rho} \textrm{ in } L^\infty([-1,0]) \cap W^{1,\infty}([-1, \zeta_1]) \cap \cdots \cap W^{1,\infty}([\zeta_{N-1}, 0]).$$
 Without loss of generality, assume that each $\mathring{\rho}_n \in \mathring{\mathcal{Z}}$ and $\mathring{\rho}_n(-d) \leq \mathring{\rho}_{\max}$.   By elliptic regularity, 
 $$ w_n := \mathcal{W}(\mathring{\rho}_n, \lambda) \in C_{\textrm{per}}^{3,\alpha}(\overline{\mathcal{S}}), \qquad n \geq 1.$$
 Furthermore, according to Lemma \ref{wave of elevation lemma}, $w_n$ is a wave of elevation.   In other words, each $(w_n, \mathring{\rho}_n)$ satisfies the hypotheses of the lemmas in section \ref{a priori section}.  In light of Lemma \ref{holder control w_x lemma}, Lemma \ref{control of w_eta in terms of R lemma}, and the continuity of $\mathcal{W}$, we can find a smaller neighborhood $\mathring{\mathcal{U}} \subset \mathring{\mathcal{Z}}$ so that 
 \be \sup_{n} \| \nabla w_n \|_{L^\infty(\mathcal{S})} \leq S_{\textrm{max}} < 1. \label{uniform W1infty bound on wn} \ee
 
 Now, observe that by Lemma \ref{control of nabla w_x lemma} and Lemma \ref{control of nabla w_xx lemma}, $\{ \partial_\xi w_n\}$ and $\{ \partial_\xi^2 w_n \}$ are uniformly bounded sequences in $W_{\textrm{per}}^{1, 2}(\mathcal{S})$.  Also, since $w_n|_{\{ \zeta = - d \}} = 0$, we have that
 $$ \partial_\xi w_n, \,  \partial_\xi^2 w_n = 0 \qquad \textrm{on } \{ \zeta = -1 \}.$$
 As we have seen, $v_n := \partial_\xi w_n$ is a generalized solution of the divergence form linear elliptic problem \eqref{v problem} with coefficients that are uniformly bounded in $L^\infty(\mathcal{S})$.  Moreover, $\partial_\xi w_n$ is bounded uniformly in $L^\infty(\mathcal{S})$ by Lemma \ref{holder control w_x lemma}.  We can therefore apply \cite[Theorem 13.1]{ladyzhenskaya1968linear} to conclude that, in fact, $\{\partial_\xi w_n\}$ is uniformly bounded in $W_{\textrm{per}}^{1,\infty}(\mathcal{S})$.  
 
 A similar argument can be made for $\{ \partial_\xi^2 w_n\}$.  Notice that $\partial_\xi^2 w_n$ is also a $W_{\textrm{per}}^{1,2}(\mathcal{S})$ solution of a  divergence form linear elliptic problem with coefficients bounded in $L^\infty(\mathcal{S})$.  Likewise, $\{\partial_\xi^2 w_n\}$ is uniformly bounded in $L^\infty(\mathcal{S})$ by  Lemma \ref{control of nabla w_xx lemma}.  Again citing \cite[Theorem 13.1]{ladyzhenskaya1968linear}, we infer that $\{ \partial_\xi^2 w_n \}$ is uniformly bounded in $W_{\textrm{per}}^{1,\infty}(\mathcal{S})$.  
 
The most sensitive estimate is for $\partial_\zeta^2 w_n$; in the limit, it will not be smooth over the interfaces.  Anticipating this, we  restrict our attention to a single layer $\mathcal{S}_i$.   Let $\tilde\alpha \in (\alpha, 1)$ be given, and put $\tilde r := 2/(1-\tilde\alpha)$.    We can express $\partial_\zeta^2 w_n$ in terms of the other derivatives using the equation:
\be \begin{split} a_{22}(\nabla w_n) \partial_\zeta^2 w_n &= -a_{11}(\nabla w_n) \partial_\xi^2 w_n - (a_{12}(\nabla w_n) + a_{21}(\nabla w_n)) \partial_{\xi}\partial_{\zeta} w_n \\
& \qquad  - {\mathring{\rho}^\prime_n \over \mathring{\rho}_n} a_2(\nabla w_n) + \lambda { \mathring{\rho}_n^\prime \over \mathring{\rho}_n} w_n.\end{split} \label{solve for w_zetazeta} \ee
In light of Lemma \ref{ellipticity lemma} and \eqref{uniform W1infty bound on wn},
\be\label{bound a_ij}
 \frac{1}{a_{22}(\nabla w_n)}, \, a_{12}(\nabla w_n), \, a_{21}(\nabla w_n), \, a_2(\nabla w_n) \in L^\infty(\mathcal{S}_i),\ee
and hence
\begin{align*} \int_{\mathcal{S}_i} | \partial_\zeta^2 w_n|^{\tilde r} \, d\xi \, d\zeta &=  \int_{\mathcal{S}_i} \frac{1}{|a_{22}(\nabla w_n)|^{\tilde r}} \bigg| -a_{11}(\nabla w_n) \partial_\xi^2 w_n - (a_{12}(\nabla w_n) + a_{21}(\nabla w_n)) \partial_{\xi}\partial_{\zeta} w_n \\
& \qquad - {\mathring{\rho}^\prime_n \over \mathring{\rho}_n} a_2(\nabla w_n) + \lambda { \mathring{\rho}_n^\prime \over \mathring{\rho}_n} w_n \bigg|^{\tilde r} \, d\xi \, d\zeta \\
& \leq C \left( \| \partial_\xi w_n \|_{W_{\textrm{per}}^{1,\tilde r}(\mathcal{S}_i)} + \| \mathring{\rho}_n^\prime\|_{L^\infty(\mathcal{S}_i)} \| w_n \|_{W_{\textrm{per}}^{1, \tilde r}(\mathcal{S}_i)} \right). \end{align*}
Thus, $\{ \partial_\zeta^2 w_n \}$ is bounded uniformly in $L_{\textrm{per}}^{\tilde r}(\mathcal{S}_i)$.   

Taking another $\xi$-derivative of \eqref{solve for w_zetazeta} we obtain an equation for $\partial_\xi\partial^2_\zeta w_n$:
\be \begin{split} a_{22}(\nabla w_n) \partial_\zeta^2\partial_\xi w_n = & -a_{221}(\nabla w_n)\partial^2_\xi w_n \partial^2_\zeta w_n - a_{222}(\nabla w_n)\partial_\xi\partial_\zeta w_n \partial^2_\zeta w_n \\
&  -a_{111}(\nabla w_n) (\partial_\xi^2 w_n)^2 - a_{112}(\nabla w_n) \partial_\xi\partial_\zeta w_n\partial_\xi^2 w_n - a_{11}(\nabla w_n) \partial_\xi^3 w_n\\
&  - \LB a_{12}(\nabla w_n) + a_{21}(\nabla w_n)\RB \partial^2_{\xi}\partial_{\zeta} w_n \\
&  - \LB a_{112}(\nabla w_n)\partial^2_\xi w_n + a_{122}(\nabla w_n)\partial_\xi\partial_\zeta w_n \right. \\
& \qquad \left. + a_{211}(\nabla w_n)\partial^2_\xi w_n + a_{212}(\nabla w_n)\partial_\xi\partial_\zeta w_n \RB \partial^2_{\xi}\partial_{\zeta} w_n \\
&  - {\mathring{\rho}^\prime_n \over \mathring{\rho}_n} \LB a_{21}(\nabla w_n)\partial^2_\xi w_n + a_{22}(\nabla w_n)\partial_\xi\partial_\zeta w_n \RB + \lambda { \mathring{\rho}_n^\prime \over \mathring{\rho}_n} \partial_\xi w_n.\end{split} \label{solve for w_zetazetaxi} \ee
From Lemma \ref{ellipticity lemma},  \eqref{uniform W1infty bound on wn}, and \eqref{bound a_ij},  we see that
\be\label{bound a_ijk}
a_{ijk}(\nabla w_n)\in L^\infty(\mathcal{S}).
\ee
Moreover, our analysis up to now confirms that
$$
\{\partial_{\zeta} w_n\}, \, \{\partial^2_\xi w_n\}, \, \{\partial_\xi\partial_\zeta w_n\}, \, \{\partial^2_\zeta w_n\}, \, \{\partial_\xi^3 w_n\}, \,\{\partial^2_{\xi}\partial_{\zeta} w_n\} \textrm{ uniformly bounded in } L^\infty(\mathcal{S}_i).
$$
This, along with \eqref{bound a_ijk}, allows us to conclude from \eqref{solve for w_zetazetaxi} that $\{ \partial_\xi \partial_\zeta^2 w_n \}$ is likewise bounded uniformly in $L^\infty(\mathcal{S}_i)$, for each strip $\mathcal{S}_i$.  

Now, from Morrey's inequality, we have the following chain of inclusions
 $$ W_{\textrm{per}}^{k,\infty}(\mathcal{S}_i) \subset W_{\textrm{per}}^{k,\tilde r}(\mathcal{S}_i) \subset C_{\textrm{per}}^{k-1,\tilde \alpha}(\overline{\mathcal{S}}_i) \subset\subset C_{\textrm{per}}^{k-1,\alpha}(\overline{\mathcal{S}_i}), \qquad k \geq 1.$$ 
Together with the arguments in the previous several paragraphs, this implies that  $\{w_n \}$ and $\{ \partial_\xi w_n\}$ are uniformly bounded in $C_{\textrm{per}}^{1,\tilde \alpha}(\overline{\mathcal{S}_1}) \cap \cdots \cap C_{\textrm{per}}^{1,\tilde \alpha}(\overline{\mathcal{S}_{N}})$.
Immediately, then, we have 
\be \label{reg of w}  w, \, w_\xi \in C_{\textrm{per}}^{1,\alpha}(\overline{\mathcal{S}_1}) \cap \cdots \cap C_{\textrm{per}}^{1,\alpha}(\overline{\mathcal{S}_{N}}). \ee 
Furthermore, \eqref{uniform W1infty bound on wn} allows us to conclude that $\| \nabla w \|_{L^\infty(\mathcal{S})} \leq S_{\textrm{max}}$.  Thus $w$ is a solution of the physical problem, Problem \ref{ter-krikorov prob}.

 Only one minor task remains:  confirming that $w \in W_{\textrm{per}}^{2,r}(\mathcal{S}_i)$, for $i = 1, \ldots, N$.  Because $w_\xi \in C_{\textrm{per}}^{1,\alpha}(\overline{\mathcal{S}})$, the only potential problem lies in $w_{\zeta\zeta}$.  However, as $\nabla w \in C_{\textrm{per}}^{0,\alpha}(\overline{\mathcal{S}_i})$,  \eqref{solve for w_zetazeta} shows that the distributional derivative $w_{\zeta\zeta} \in \mathcal{D}_{\textrm{per}}^\prime(\mathcal{S}_i)$ can be identified with an $L^\infty(\mathcal{S}_i)$ function.  Thus $w \in W_{\textrm{per}}^{2,r}(\mathcal{S}_i)$, and the proof is complete. \end{proof} 
\begin{remark} \label{extra regularity remark}  From the proof above, we see that, in fact $w \in W^{2,\infty}(\mathcal{S}_i)$, for each layer $\mathcal{S}_i$.  
\end{remark}

\section{Proof of the main result} \label{proof of main theorems section}
In this section, we state and prove our main theorem --- the rigorous version of the statements in \S\ref{intro statement of results section}.  We have already accomplished the lion's share of the work in the previous section; Theorem \ref{main continuity theorem} essentially proves statement (A).  What remains is to transition back to the original formulation, and also to prove statements (B) and (C).  

To state things concisely, we define the set of stable streamline density functions 
\begin{align*} \mathscr{D} &:= \Big\{ \rho \in L^\infty([p_0,0]) : \rho > 0, \,  \rho(0) = 1, \, \textrm{$\rho$ is non-increasing, and }  \\
& \qquad\qquad \rho \in W^{1,\infty}([p_0, p_1]) \cap \cdots \cap W^{1,\infty}([p_{N-1}, 0]),  \textrm{ for some $p_1, \ldots, p_{N-1}$} \Big\}.\end{align*}
It is easy to see that this is a convex subset of $L^\infty([p_0,0])$.  

At last, the result is the following.

\begin{thm}[Main theorem] \label{main theorem} Fix a H\"older exponent $\alpha \in (0,1)$, and put $r := 2/(1-\alpha)$.  Choose a pseudo volumetric mass flux $p_0 < 0$ and ocean depth $d > 0$.  Let $\rho_* \in C^{1,\alpha}([-p_0,0])$ be a stably stratified streamline density function.  Choose a wave speed $c_* > c_{\mathrm{crit}}(\rho_*)$.  There exists a minimal period $L_{\mathrm{min}}$ and amplitude bound $S_{\mathrm{max}}$ such that, for any $(u_*,v_*, \varrho_*, P_*, \eta_*)$ solving Problem \ref{weak Euler prob} with streamline density $\rho_*$ that is (i) periodic localized near the crest, with period $L > L_{\mathrm{min}}$; (ii) a wave of strict elevation; and (iii) sufficiently small-amplitude,
\be  \left|\frac{v_*}{c_*-u_*}\right|, \left| \frac{u_*}{c-u_*}\right| < S_{\mathrm{max}} \qquad \textrm{in } \Omega_*, \label{eulerian small amplitude} \ee
there exists a neighborhood $\mathcal{U}$ of $\rho_*$ in $L^\infty([p_0,0])$ such that the following statements hold. 
\begin{itemize}
\item[(a)] There exists a Lipschitz continuous mapping
$$ \mathcal{H} : \mathscr{D} \cap \mathcal{U} \to W_{\mathrm{per}}^{1,r}(\mathcal{R}) \subset C_{\mathrm{per}}^{0,\alpha}(\overline{\mathcal{R}})$$ 
such that $\mathcal{H}(\mathring{\rho}_*) = h_*$, the height function corresponding to the fixed flow $(u_*, v_*, \varrho_*, P_*, \eta_*)$, and, for each $\rho \in \mathscr{D} \cap \mathcal{U}$, $\mathcal{H}(\rho)$ solves the height equation Problem \ref{localized height equation prob} with streamline density $\rho$.   Also, $\mathcal{H}(\rho)$ is even in $q$ and a wave of elevation.  Lastly, the corresponding wave speed $c$ obeys the identity 
\be \label{wave speed estimate} c - c_* = \frac{1}{d} \int_{p_0}^0 \frac{ \rho_* - \rho }{\rho \sqrt{\rho_*} + \rho_*\sqrt{\rho} } \, ds = \mathcal{O}(\| \rho - \rho_*\|_{L^\infty}). \ee
\item[(b)] Fix $\rho \in \mathscr{D} \cap  \mathcal{U}$, and let $I \subset\subset [p_0, 0] \setminus \{ p_1, \ldots, p_{N-1}\}$ be a connected with $\rho \in C^{1,\alpha}(\overline{I})$.  Then 
$$ \| \mathcal{H}(\rho) - h_* \|_{C_{\mathrm{per}}^{1,\alpha}( \mathbb{R} \times \overline{I})} \leq C_1\left( \| \rho - \rho_* \|_{L^\infty([p_0,0])} +  \| \rho - \rho_* \|_{C^{1,\alpha}(\overline{I})} \right), $$
where $C_1 > 0$ depends on  $|I|$, $\rho_*$, and $h_*$.   
\item[(c)] Let $I$ and $\rho \in \mathscr{D} \cap \mathcal{U}$ be given as in \emph{(b)}, and let $P$ denote the pressure for the wave with height function $\mathcal{H}(\rho)$.  Then, 
$$ \| P - P_* \|_{C^{0,\alpha}_{\textrm{per}}(\mathbb{R} \times \overline{I})} \leq C_2\left(  \| \rho - \rho_* \|_{L^\infty([p_0, 0])} + \| \rho - \rho_* \|_{C^{1,\alpha}(I)} \right),$$
where $C_2 > 0$ depends on $|I|$, $\rho_*$, and $h_*$.   
\end{itemize}
\end{thm}
\begin{proof}  Let $(u_*, v_*, \varrho_*, P_*, \eta_*)$, $c_*$, and $\rho_*$  satisfying the above hypotheses be given; let $(w_*, \mathring{\rho}_*, \lambda_*)$ be the corresponding objects in the Ter-Krikorov formulation.  Note that hypothesis {(iii)}  says that $\| \nabla w_* \|_{L^\infty} < S_{\textrm{max}}$, taking into account the change of variables identities.  We may then apply Theorem \ref{main continuity theorem}, and also let a neighborhood $\mathring{\mathcal{U}} \times \Lambda$ of $(\mathring{\rho}_*,\lambda_*)$ in $L^\infty([-1,0]) \times \mathbb{R}$, and a mapping $\mathcal{W} \in C^1(\mathring{\mathcal{U}} \times \Lambda; X_1)$ be given as in Theorem \ref{Calpha continuity theorem}.  

This gives solutions to the Ter-Krikorov problem.  To translate them back to the height equation formulation, we consider the mappings $\mathpzc{c}: \mathscr{D} \to \mathbb{R}$, $\mathpzc{l} : \mathscr{D} \to \mathbb{R}$, and  $\mathpzc{z} : \mathscr{D} \to  W^{1,\infty}([p_0,0])$ defined by
$$ \mathpzc{c}(\rho) := \frac{1}{d} \int_{p_0}^0 \frac{ds}{\sqrt{\rho(s)}}, \qquad \mathpzc{l}(\rho) := \frac{gd}{\mathpzc{c}(\rho)^2},$$
and
$$ \mathpzc{z}(\rho)(p) := \frac{1}{ \mathpzc{c}(\rho)d  } \int_{p_0}^p \frac{ds}{ \sqrt{\rho(s)}} -1.$$
Notice that $\mathpzc{z}(\rho)(\cdot)$ is monotonic and surjective, and is thus a homeomorphism from $[p_0, 0]$ to $[-1,0]$.  It is likewise easy to confirm that $\mathpzc{z}$, $\mathpzc{l}$,  and $\mathpzc{c}$ are Lipschitz continuous in $\rho$.     

With these facts in mind, we observe that for any $\rho \in \mathscr{D}$, the corresponding rescaled streamline density function $\mathring{\rho}$ satisfies the identity $\rho = \mathring{\rho}\circ \mathpzc{z}(\rho)$.  Similarly, the Richardson number is found by setting $\lambda = \mathpzc{l}(\rho)$.   We may thus recover the neighborhood $\mathcal{U}$ of $\rho_*$ by taking it to be the pre-image of $\mathring{\mathcal{U}}$ under $\mathpzc{z}$.  Likewise, $\mathcal{H}$ is defined by pulling back $\mathcal{W}$: 
$$ [\mathcal{H}(\rho)](q,p) := \left[\mathcal{W}\left(\mathring{\rho}, \lambda \right) \right] \left(\frac{q}{d}, \mathpzc{z}(p) \right), \qquad \textrm{for all } (q,p) \in \overline{\mathcal{R}}.$$
The regularity statements follow from Theorem \ref{main continuity theorem} and the equivalence of the formulations.  Notice that, because we are restricting the domain to $\mathscr{D}$ which is merely convex, $\mathcal{H}$ is merely Lipschitz, not differentiable.  The fact that $\mathcal{H}(\rho)$ is even and a wave of elevation {is} a consequence of Lemmas \ref{symmetry lemma} and \ref{wave of elevation lemma}.  Lastly, \eqref{wave speed estimate} simply comes from evaluating $\mathpzc{c}(\rho) - \mathpzc{c}(\rho_*)$.  The proof of (a) is complete.

For (b), it is easier to work first in the Ter-Krikorov formulation.  Let $I$ be given as above and put $J := \mathpzc{z}(\rho)(I)$, and ${\mathpzc{T}} := [-L/d, L/d] \times J$.  By definition, ${\mathpzc{T}} \subset \overline{\mathcal{S}_i}$, for some layer $\mathcal{S}_i$.  

Denote $w := \mathcal{W}(\rho, \lambda)$, and  $u := w_*-w$. Then $u$ satisfies the following PDE 
\be  \nabla \cdot (\mathbf{M} \nabla u) - \lambda \mathring{\rho}^\prime u =  \nabla \cdot \left( [\mathring{\rho}-\mathring{\rho}_*] \mathbf{F} (\nabla w_*)\right)  + [\lambda_* - \lambda] \mathring{\rho}_* w_*- \lambda [\mathring{\rho}^\prime - \mathring{\rho}_*^\prime] w_* \qquad \textrm{in } \mathpzc{T}, \label{difference equation} \ee
where $\mathbf{F}$ is defined as in \eqref{ter-krikorov eq}, and $\mathbf{M} = (M_{ij})$ is the self-adjoint matrix:
\begin{align*}
M_{11} &= \frac{\mathring{\rho}}{1+\partial_\zeta w_*} \\
M_{12} & = M_{21} = -\frac{1}{2} \frac{\mathring{\rho}}{(1+\partial_\zeta w_*)^2 (1+\partial_\zeta w)^2} \Big[ (1+ \partial_\zeta(w+w_*) + (\partial_\zeta w)(\partial_\zeta w_*)) \partial_\xi w \\
& \qquad \qquad +  \frac{1}{2}(1 + \partial_\zeta w)^2 \partial_\xi (w+w_*) \Big]  \\
M_{22} & = \frac{\mathring{\rho}}{(1+\partial_\zeta w)^2 (1+\partial_\zeta w_*)^2} \left[ 1+(\partial_\xi w)^2 + \frac{1}{2}(1 - 2\partial_\zeta w + (\partial_\xi w)^2) \partial_\zeta (w + w_*) \right].
\end{align*} 
An elementary but tedious calculation confirms that $\mathbf{M}$ is positive definite.  By the regularity of $w$ and $w_*$, the entries of $\mathbf{M}$ are of class $C^{0,\alpha}$ in $\overline{\mathpzc{T}}$.   Next consider the terms occurring on the right-hand side of \eqref{difference equation}.  By the equation satisfied by $w_*$, we have that
$$ \nabla \cdot \mathbf{F}(\nabla w_*) = \frac{1}{\mathring{\rho}}\left[ -\mathring{\rho}_*^\prime F_2(\nabla w_*) + \lambda_* \mathring{\rho}_*^\prime w \right] \in C_{\textrm{per}}^{0,\alpha}(\overline{\mathpzc{T}}). $$
The other terms are likewise of class $C^{0,\alpha}$, taking into account the regularity of $\rho$, $\rho_*$, $w$, and $w_*$ in $\mathcal{S}_i$.  A similar computation shows that on $\{ \zeta = 0 \}$, a co-normal boundary condition is satisfied.  Moreover, since $u = 0$ on the lower boundary $\{ \zeta = -1\}$, we may pose a homogeneous Dirichlet condition there.  Applying a standard Schauder-type estimate (cf., e.g., \cite[Theorem 3]{constantin2011discontinuous}) yields 
\begin{align*} \| w - w_* \|_{C^{1,\alpha}(\overline{\mathpzc{T}})} &\leq C \Big( \| w - w_* \|_{C^0(\overline{\mathpzc{T}})} + \| (\mathring{\rho} - \mathring{\rho}_*) \nabla \cdot \mathbf{F}(\nabla w_*)\|_{C^{0,\alpha}_{\textrm{per}}(\overline{\mathpzc{T}})} \\
& \qquad + |\lambda- \lambda_*| \| \rho_* w_* \|_{C^{1,\alpha}(\overline{\mathpzc{T}})} + \| (\mathring{\rho}^\prime - \mathring{\rho}_*^\prime)(\lambda w_* - F_2(\nabla w_*) \|_{C^{0,\alpha}_{\textrm{per}}(\overline{\mathpzc{T}})}  \Big) \\
& \leq C\left( \| \mathring{\rho} - \mathring{\rho}_* \|_{L^\infty([p_0,0])} +  \| \mathring{\rho} - \mathring{\rho}_* \|_{C^{1,\alpha}(J)} \right).  \end{align*} 
Part (b) follows now by simply re-expressing this in terms of  $h_*$ and $\mathcal{H}(\rho, \lambda)$.

Finally, to prove the pressure convergence in part (c), we recall that from Bernoulli's theorem and the change of variable identities, 
\be  P(q,p) =  \sum_i \frac{Q_i}{2} \mathds{1}_{\mathcal{R}_i} + P_{\textrm{atm}} -  \frac{1+h_q^2}{2h_p^2} - g \rho h + B(p). \label{P in semi-Lagrangian} \ee
For $p \in [p_j, 0)$, 
\begin{align*} B(p) - B_*(p) &=  (\frac{1}{2} c^2-gd)(\rho(p) - \rho_*(p)) + \frac{1}{2}(c^2 - c_*^2) \rho_*(p)+ g\int_0^p (\rho^\prime - \rho_*^\prime) \mathring{h} + \rho^\prime_* ( \mathring{h} - \mathring{h}_*) \, ds \\
& = (\frac{1}{2} c^2-gd)(\rho(p) - \rho_*(p)) + \frac{1}{2}(c^2 - c_*^2) \rho_*(p) \\
& \qquad  +g \sum_{i \geq j} \left[  ( \jump{\rho}_i - \jump{\rho_*}_i) \mathring{h}(p_i)  + \jump{\rho_*}_i (\mathring{h}(p_i) - \mathring{h}_*(p_i)) \right] \\
& \qquad - g\int_0^p \left[ (\rho - \rho_*) \frac{1}{c \sqrt{\rho}} + \rho_* \left( \frac{1}{c \sqrt{\rho}} - \frac{1}{c_* \sqrt{\rho_*}} \right)  \right] \, ds. 
\end{align*}
Hence 
$$\| B - B_*\|_{C^{0,\alpha}(\overline{I})} \leq C\left( \| \rho - \rho_* \|_{L^\infty([p_0,0])}  + \| \rho - \rho_* \|_{C^{0,\alpha}(\overline{I})} \right).$$
Now, applying part (b) to \eqref{P in semi-Lagrangian}, we get the desired estimate for $P - P_*$.   \end{proof}

\appendix
\section{Proof of formulation equivalence} \label{formulation equivalence appendix}
We provide, in this appendix, the proof that the various formulations of the steady wave problem are indeed equivalent.  The arguments here  closely follow those of Constantin and Strauss in \cite{constantin2011discontinuous}.  The main task it to generalize their work to allow for multiple layers, as well as heterogeneous density.

We mention that the regularity statements in these results are sub-optimal.  We are ultimately interested in dealing with  H\"older continuous functions --- we work in Sobolev spaces in order to draw on certain key results in the literature of elliptic equations.  But this compels us to assume more regularity than should be necessary.  For example, to ensure that $\psi$ is of class $C^{0,\alpha}(\Omega)$, we are assuming the stronger statement that it is in $W^{1,r}$, and then appealing to Sobolev embedding.  

In the constant density regime, this problem was resolved by Varvaruca and Zarnescu \cite{varvaruca2012equivalence}.  They show that, for $\alpha \in (1/3,1]$, an equivalence between the Eulerian formulation, stream function formulation, and height equation holds working exclusively in H\"older continuous spaces with exponent $\alpha$.  This is a rather deep result: the appearance of $1/3$ is connected to the famous Onsager conjecture.   It is our suspicion that a generalization of Varvaruca and Zarnescu's argument would apply to the stratified regime, but that is beyond the scope of our present ambitions. 

We begin with a technical lemma which essentially states that, starting from the weak Euler formulation, the pseudo relative stream function and Bernoulli function are each well-defined and have the stated regularity. 

\begin{lemma}[Chain rule and composition in Sobolev spaces]  \label{chainrulelemma} Let $\alpha \in (0,1)$ be given and put $r := 2/(1-\alpha)$.  \begin{itemize}
\item[{(i)}] Suppose that there exists a solution $(u,v,P, \varrho, \eta, \eta_1, \ldots, \eta_{N-1})$ to the weak Euler problem as detailed in the statement of Problem \ref{weak Euler prob}.  Then there exists {a}
\be  \psi \in  W_{\mathrm{per}}^{1,r}(\Omega) \cap W_{\mathrm{per}}^{2,r}(\Omega_1) \cap \cdots \cap W_{\mathrm{per}}^{2,r}(\Omega_N) \label{chain rule reg psi} \ee
 satisfying \eqref{defpsi}, \eqref{nostagnationpsi}, \eqref{psisurfacecond}, and \eqref{psibedcond}.  Moreover, if $F: [p_0,0] \to \mathbb{R}$ has the regularity 
 \be  F \in L^{r}([p_0,0]) \cap W^{1,r}([p_0,p_1]) \cap \cdots \cap W^{1,r}([p_{N-1}, p_0]), \label{chain rule reg F} \ee
  then $F \circ (-\psi) \in L_{\mathrm{per}}^r(\Omega)$, $F \circ (-\psi) \in W_{\textrm{per}}^{1,r}(\Omega_i)$, and the chain rule holds in the interior of each $\Omega_i$:
\be \partial_x F(-\psi) = -F^\prime(-\psi) \psi_x, \qquad \partial_y F(-\psi) = - F^\prime(-\psi) \psi_y. \label{sobolevchainrule} \ee

\item[{(ii)}]  Suppose that $\psi$ has the regularity given in \eqref{chain rule reg psi}, and solves  the Stream-function problem \eqref{streamfunctionprob}, for some $(Q, \beta, \rho, \eta, \eta_1, \ldots, \eta_{N-1})$ with the regularity specified in Problem \ref{localized stream function prob}.  Then for every $F$ as in \eqref{chain rule reg F}, the same conclusion holds as in part {\em(i)}:  $F \circ (-\psi) \in {L^r_{\mathrm{per}}(\Omega)}$, $F \circ (-\psi) \in W_{\textrm{per}}^{1,r}(\Omega_i)$, and the chain rule \eqref{sobolevchainrule} applies in the interior of each $\Omega_i$.\end{itemize}   \end{lemma}
\begin{proof}  (i)  Suppose that we have a solution to the weak Euler problem. The choice of $r$ was made so that we may exploit the embedding $W_{\textrm{per}}^{1,r}(\mathbb{R}^2) \subset C_{\textrm{per}}^\alpha(\mathbb{R}^2)$.  Consider the restriction of the velocity field and density to a single layer: $u_i := u|_{\Omega_i}, v_i := v|_{\Omega_i}$, $\varrho|_i := \varrho|_{\Omega_i}$.  As $\eta, \eta_i \in {C^{1,\alpha}}$, the boundary of $\Omega_i$ is better than Lipschitz, allowing us to extend $u_i, v_i,$ and $\varrho_i$ to functions in $W_{\textrm{per}}^{1,r}(\mathbb{R}^2)$.  Indeed, we may do this in such a way that the extensions are compactly supported in the $y$-variable.  It follows by Morrey's inequality that the extensions are all of class $C^\alpha$, hence their restrictions to $\Omega_i$ are in $C_{\textrm{per}}^{\alpha}(\overline{\Omega_i})$.  In particular this implies that they are all $L^\infty(\Omega_i)$.

By H\"older's inequality, $\sqrt{\varrho}(u-c, v) \in W^{1,r}({\Omega_i})$, thus in each layer \eqref{defpsi} defines a function $\psi_i \in W_{\textrm{per}}^{2,r}(\Omega_i)$ up to a constant.  Since the interfaces $\eta, \eta_i$ are Lipschitz, the traces of the $\psi_i$ are well-defined.  Moreover, \eqref{weakkinematicsurface}, \eqref{weakkinematicbed}, and \eqref{weakkinematicinterface} ensure that the traces are constants.  It follows that we may take $\psi_N = 0$ on $\{ y = \eta(x) \}$, and that there is a unique choice of the remaining constants such that $\psi_i = \psi|_{\Omega_i}$, for a function $\psi \in C_{\textrm{per}}^0(\Omega)$; in fact, $\psi$ is globally {defined} by the formula:
$$ \psi(x,y) = -\int_y^{\eta(x)} \sqrt{\varrho(x,z)} [ u(x,z) - c ] \, dz, \qquad \textrm{in } \Omega.$$
Consequently,  the trace of $\psi$ on $\{ y = -d\}$ is $-p_0$, where $p_0$ is defined as in \eqref{defp0}.  It is simple to show that the formula above implies that $\psi \in C_{\textrm{per}}^{0,1}(\overline{\Omega})$, and so in particular $ \psi \in W_{\textrm{per}}^{1,r}(\Omega)$.  

For the second statement in part (i), it is most convenient to transition to the {height} equation formulation.  Recalling the change of variables in \cite{walsh2009stratified}, we have that $h$ defined by \eqref{defheight} satisfies 
$$ h_q = \frac{v}{u-c}, \qquad h_p = \frac{1}{\sqrt{\varrho}(c-u)}, \qquad v = -\frac{h_q}{h_p}, \qquad u = c- \frac{1}{\sqrt{\varrho} h_p},$$
and
$$ \partial_x = \partial_q - \frac{h_q}{h_p} \partial_p, \qquad \partial_y = \frac{1}{h_p} \partial_p.$$
From these statements it is obvious that $h \in W_{\mathrm{per}}^{1,r}(\mathcal{R}) \cap W_{\mathrm{per}}^{2,r}(\mathcal{R}_1) \cap \cdots \cap W_{\mathrm{per}}^{2,r}(\mathcal{R}_N)$.  

Let $F$ be given as in \eqref{chain rule reg F}.  First note that $F \circ (-\psi) \in L_{\textrm{per}}^r(\Omega)$, since $\psi \in C^{0,\alpha}(\overline{\Omega})$.  Letting $\varphi \in \mathcal{D}_{\textrm{per}}(\Omega_i)$ be a (periodic) test function, we calculate that
\begin{align*}
\int\!\!\!\int_{\Omega_i} \varphi \partial_x F(-\psi) \, dy\,dx & = \int_{-L}^L  \int_{\eta_{i-1}(x)}^{\eta_i(x)} \varphi \partial_x F(-\psi) \, dy\, dx  = -\int_{-L}^L  \int_{\eta_{i-1}(x)}^{\eta_i(x)} F(-\psi) \partial_x \varphi \, dy\, dx \\
& = -\int_{-L}^L \int_{p_{i-1}}^{p_i} h_p F(p) \partial_x \varphi \, dp \, dq  = -\int\!\!\!\int_{\mathcal{R}_i} h_p F \left( \partial_q - h_q h_p^{-1} \partial_p \right) \varphi \, dp \, dq \\
& = - \int\!\!\!\int_{\mathcal{R}_i} h_p F \varphi_q \, dp \, dq + \int\!\!\!\int_{R} F h_q \varphi_p \, dp \, dq  =  -\int\!\!\!\int_{\mathcal{R}_i} F_p h_q \varphi \, dp \, dq \\
& =  \int\!\!\!\int_{\Omega_i} F^\prime(-\psi) h_q h_p^{-1}  \varphi \, dy\,dx = \int\!\!\!\int_{\Omega_i} F^\prime(-\psi) \psi_x  \varphi \, dy\,dx.
\end{align*}
We have therefore shown that 
$$ \partial_x F(-\psi) = -F^\prime(-\psi) \psi_x \in L_{\textrm{per}}^r(\Omega_i).$$
Again, letting $\varphi \in \mathcal{D}_{\textrm{per}}(\Omega_i)$ be given, we compute
\begin{align*}
\int\!\!\!\int_{\Omega_i} \varphi \partial_y F(-\psi) \, dy\,dx & = \int_{-L}^L  \int_{\eta_{i-1}(x)}^{\eta_i(x)} \varphi \partial_y F(-\psi) \, dy\, dx  = - \int_{-L}^L  \int_{\eta_{i-1}(x)}^{\eta_i(x)} \partial_y \varphi  F(-\psi) \, dy\, dx \\
& = -\int\!\!\!\int_{\mathcal{R}_i} h_p F(p) \partial_y \varphi\, dp\, dq  =  -\int\!\!\!\int_{\mathcal{R}_i} F \varphi_p \, dp\, dq \\
& =  \int\!\!\!\int_{\mathcal{R}_i}  F_p \varphi \, dp\, dq  =   \int\!\!\!\int_{\Omega_i} F^\prime(-\psi) \psi_y \varphi \, dy\, dx.
\end{align*}
This identity means,
$$ \partial_y F(-\psi) = -F^\prime(-\psi) \psi_y \in L_{\textrm{per}}^r(\Omega_i).$$
Together with our last computation, this shows that $F \circ (-\psi) \in W_{\textrm{per}}^{1,r}(\Omega_i)$, for each $i$, and the chain rule \eqref{sobolevchainrule} indeed holds for $F \circ (-\psi)$.   

(ii) Now assume that we are given the stream function $\psi$ directly.  We may again consider the Dubriel-Jacotin variables, and define $h(q,p) = y +d$.  The corresponding change of variable rules are now
$$ h_q = \frac{\psi_x}{\psi_y}, \qquad h_p =- \frac{1}{\psi_y}, \qquad \partial_x = \partial_q -\frac{h_q}{h_p} \partial_p, \qquad \partial_y = \frac{1}{h_p} \partial_p.$$
It follows directly that $h \in W_{\textrm{per}}^{2,r}(\mathcal{R}_i)$.   In fact, all of the relevant computations done in part (i) hold verbatim.  So an identical proof shows that, for all $F$ as in \eqref{chain rule reg F}, $F \circ (-\psi) \in L_{\textrm{per}}^r(\Omega)$, moreover, in each layer $F \circ (-\psi) \in W_{\textrm{per}}^{1,r}(\Omega_i)$, and the chain rule \eqref{sobolevchainrule} holds.  
\end{proof}

\begin{lemma}[Equivalence] \label{equivalence lemma 1} Let $\alpha \in (0,1)$ be given and put $r := 2/(1-\alpha)$.  The following statements are equivalent.
\begin{itemize}
\item[(i)] There exists a solution $(u,v,P, \varrho, \eta, \eta_1, \ldots, \eta_{N-1})$ to the weak Euler problem, as stated in Problem \ref{weak Euler prob}.  

\item[(ii)] There exists a solution $(\psi,\eta, \eta_1, \ldots, \eta_{N-1}, Q)$ to the weak stream function problem, as stated in Problem \ref{general weak stream function prob}, for some $\beta$, and $\rho$ satisfying \eqref{rhopositive}--\eqref{stablestratification}.  

\item[(iii)] There exists a solution $(h, Q)$ to the height equation problem, as stated in Problem \ref{general weak height eq prob}, for some $B$, and $\rho$ satisfying \eqref{rhopositive}--\eqref{stablestratification}.  
\end{itemize} \label{equivalencelemma} 
\end{lemma}
\begin{proof}  Suppose that (i) holds, meaning that we have a solution with the stated regularity to the Weak Euler problem.  In Lemma \ref{chainrulelemma} we argued that this allows us to take $u|_{\Omega_i} ,v|_{\Omega_i} ,\varrho|_{\Omega_i} \in  C_{\textrm{per}}^{\alpha}(\overline{\Omega_i})$.  We also know that the pseudo stream function is well-defined, as is the Dubreil-Jacotin transformation.  

First we confirm the existence of the streamline density function.  Fix a layer $\Omega_i$.  Observe that
\begin{align*} (u-c)\partial_q \varrho &= (u-c) \left(\partial_x + \frac{h_q}{h_p} \partial_p\right) \varrho = (u-c) \left(\partial_x + h_q \partial_y\right) \varrho \\
& =  \left((u-c) \partial_x + v \partial_y \right) \varrho = 0,\end{align*}
by \eqref{weakmass}.  From \eqref{nostagnation}, and the fact that $(c-u)^{-1} \in W_{\textrm{per}}^{1,r}(\Omega_i)$, it follows that $\partial_q \varrho = 0$, that is, $\varrho|_{\Omega_i}$ depends only on $p$.  Since the boundary of $\partial\Omega_i$ are streamlines, we may therefore define $\rho(p) := \varrho(x,y)$ for $p \in [p_0, 0] \setminus \{ p_1, \ldots, p_{N-1}\}$.  As $\varrho \in L^\infty(\Omega)$, we have $\rho \in L^{\infty}([p_0, 0])$.  Moreover, the change of variables gives
$$ \partial_q \rho = 0, \qquad \partial_p \rho = \frac{1}{\sqrt{\varrho}(c-u)} \partial_x \varrho. $$
The no stagnation condition \eqref{nostagnation} and positivity of $\rho$ \eqref{rhopositive}, together with the fact that $u$ and $\rho$ are in $C^\alpha(\overline{\Omega}_i)$, imply that ${[\sqrt{\varrho}(c-u)]^{-1}} \in L^\infty(\Omega_i)$.  This proves $\rho \in W^{1,r}([p_{i-1},p_i])$.

Next, letting $E$ be the quantity in \eqref{defE}, we deduce that $E$ does not depend on $q$.  Let $\mathbf{g} := (0,g)$.  Given the regularity of $u,v,\varrho, P$, it follows that $E \in W_{\textrm{per}}^{1,r}(\Omega_i)$, for each layer $\Omega_i$.  Denoting $\nabla := \nabla_{(x,y)}$, and working in $\Omega_i$, we compute
\begin{align}
 \nabla E &= \nabla P + \frac{1}{2}((u-c)^2 + v^2)  \nabla \varrho + \varrho ( -cu_x + uu_x + vv_x, -cu_y + uu_y + vv_y) \nonumber \\ 
 & \qquad + \varrho \mathbf{g} + gy \nabla \varrho \nonumber \\
 & = \left( (\sqrt{\varrho} (u-c))_y - (\sqrt{\varrho} v)_x\right) \left(-\sqrt{\varrho}v, \sqrt{\varrho}(u-c)\right) + gy \nabla \varrho, \label{nablaE}
\end{align}
in light of \eqref{weakmomentumx}--\eqref{weakmomentumy}.   Combining this with the change of variable formulas, we see that
\begin{align*} (u-c) \partial_q E &= \left(u-c, v \right) \cdot \nabla E \\
& =  \left( (\sqrt{\varrho}(u-c))_y - (\sqrt{\varrho}v)_x \right) \left\{ -\sqrt{\varrho}(u-c)v + \sqrt{\varrho} v(u-c)\right\} \\
& \qquad + {g}y(u-c,v) \cdot \nabla \varrho\\ 
& = {g}y(u-c) \partial_q \varrho = 0.  \end{align*}
Hence $E$ is independent of $q$ in each layer.  As the layers themselves have streamlines for boundaries, we may therefore introduce a function 
$$ B \in L^r([p_0, 0]) \cap W^{1,r}([p_0,p_1]) \cap \cdots \cap W^{1,r}([p_{N-1},0])$$ 
such that $E(x,y) = B(p) = B(-\psi(x,y))$,  and define $\beta \in L^r([0,|p_0|]) \cap W^{1,r}((p_{i-1}, p_i))$ by 
$$\beta(-p) + \sum_{i=1}^{N-1} \jump{B}_i \delta_{p_i}  = B^\prime(p).$$ 
Here $\delta_{p_i}$ is the Dirac $\delta$ measure centered on $p_i = -\psi|_{\{y = \eta_i(x)\}}$.  
The last preliminary step is to show that the functions $\rho\circ (-\psi)$ and $B \circ (-\psi)$ are in $W_{\textrm{per}}^{1,r}(\Omega_i)$ and obey the chain rule.  This follows directly from Lemma \ref{chainrulelemma}, taking $F = B$ and $F = \rho$.

With these facts established, we can begin proving the equivalences of the three formulations, beginning with (i) implies (ii).   From Lemma \ref{chainrulelemma}, we may introduce $\psi \in W_{\textrm{per}}^{2,r}(\Omega)$ such that \eqref{defpsi}, \eqref{nostagnationpsi}, \eqref{psisurfacecond}, and \eqref{psibedcond} hold.     We must now show that $\psi$ satisfies Yih's equation \eqref{weakyih} and the Bernoulli condition \eqref{weakbernoulli}, where $\rho$ and $\beta$ are the streamline density function and Bernoulli function whose existence we proved earlier, and $Q, Q_i$ are defined as in \eqref{defQ} and \eqref{defQi}.  The Bernoulli condition follows directly from the definitions.  To obtain Yih's equation, let us return to the computation of $\nabla E$ in \eqref{nablaE}.  Written in terms of $\psi$, this identity becomes
 $$ \nabla E = (\Delta \psi) \nabla \psi + gy \nabla \varrho, \qquad \textrm{in } \bigcup_i \Omega_i.$$
 Taking the inner product with $\nabla \psi$, this simplifies to the scalar equation 
 \be \nabla E \cdot \nabla \psi = |\nabla \psi|^2 \Delta \psi  + gy \nabla \varrho \cdot \nabla \psi. \label{yiheqcalc1} \ee

 Fix a layer $\Omega_i$. Using the fact that $E(x,y) = B(-\psi(x,y))$ in $\Omega_i$, we compute 
 $$ \nabla E = -B^\prime(-\psi) \nabla \psi, \qquad \nabla \varrho = -\rho^\prime(-\psi) \nabla\psi.$$
 Here we have made use of the chain rule, whose validity in this setting we confirmed earlier.  Inserting this into \eqref{yiheqcalc1} reveals
$$ -|\nabla \psi|^2 B^\prime(-\psi) = |\nabla \psi|^2 \left( \Delta \psi - gy \rho^\prime(-\psi) \right).$$
By the no stagnation condition, $|\nabla \psi|^2 > 0$, and hence the line above reduces to Yih's equation \eqref{weakyih} upon dividing by $|\nabla \psi|^2$ and recalling the definition of $\beta$ in \eqref{defbeta}.   This completes the proof of (i) implies (ii).

W next show (ii) implies (i).   Let $\beta \in L^r([0,|p_0|])$, $\rho \in L^{1,r}([p_0,0])$ with $\rho \in W^{1,r}([p_{i-1}, p_i])$ for each $i = 1, \ldots, N$, and  satisfying \eqref{rhopositive}--\eqref{stablestratification}.   Suppose that $(\psi,\eta, \eta_1, \ldots, \eta_{N-1}, Q, Q_1, \ldots, Q_{N-1})$ solves Problem \ref{general weak stream function prob} for this choice of Bernoulli function and streamline density function.   We can recover the Eulerian density $\varrho$ by taking $\varrho = \rho \circ (-\psi)$, which has the correct regularity by Lemma \ref{chainrulelemma}.  The velocity field in each layer is found by taking 
$$ (u-c, v) = \frac{\nabla^\perp \psi}{\sqrt{\rho(-\psi)}} \in W_{\textrm{per}}^{1,r}(\Omega_i).  $$    
Finally, to obtain the pressure we use Bernoulli's theorem:  let 
$$ P := \sum_i \frac{Q_i}{2} \mathds{1}_{\Omega_i} + P_{\textrm{atm}} - \frac{1}{2} |\nabla \psi|^2 - g \rho(-\psi) (y + d) + B(-\psi),$$
It is easy to verify that $(u,v, \varrho, P)$ thus defined satisfy the weak Euler equation in each layer $\Omega_i$; it remains only to check that the pressure is continuous in the entire fluid domain.  This is simply a consequence of the transmission boundary conditions \eqref{psicontpress}.

Next we prove that (ii) implies (iii).  But this is easy, since (i) and (ii) are equivalent, we may introduce $h$ as before, and the height equation follows from rewriting Yih's equation in the new variables $(q,p)$ (cf. \cite{walsh2009stratified}).

The last step is to show (iii) implies (ii).  Let $(h,Q, Q_1, \ldots, Q_{N-1})$ solve Problem \ref{general weak height eq prob} for some choice of $\rho$ and $B$.  We seek to reconstruct the stream function and the fluid domain.  The latter is simple.  For $x \in \mathbb{R}$, define
\be \eta(x) := h(x,0) - d.  \label{defetaequiv}  \ee
Because $h \in W^{2,r}(\mathcal{R}_N) \subset C^{1,\alpha}(\overline{\mathcal{R}_N})$, the trace is well-defined, and $\eta$ is of class ${C_{\textrm{per}}^{1,\alpha}(\mathbb{R})}$.  The fluid domain is thus $\Omega := \{(x,y) \in \mathbb{R}^2 : {-d} < y < \eta(x) \}$.

Now consider the stream function.  We will work from the top layer down.  Define $F \in C^\alpha(\overline{\mathcal{R}_N})$ by $F = h_p^{-1}$.  (Note that the regularity of $F$ here is once again a consequence of the embedding $C^\alpha(\mathbb{R}) \subset W^{1,r}(\mathbb{R})$, and the positivity of $h_p$ condition \eqref{hppositive}).  By Peano's theorem, for each $x_0 \in \mathbb{R}$, we may let $\psi(x_0, \cdot)$ be the solution of the ODE
\be \psi_y(x_0, y) = -F(x_0, -\psi(x_0,y)), \qquad \psi_y(x_0,\eta(x_0)) = 0,\label{psiode} \ee
which is guaranteed to exist for $y$ in some interval  $I(x_0) := [y_{\textrm{min}}(x_0), \eta(x_0)]$.   In fact, if $y(x_0)$ is finite for some $x_0$, then $|\psi(x_0,y)| \to \infty$ as $y \searrow y_{\textrm{min}}(x_0)$.  By \eqref{hppositive}, we know that $F$ is bounded strictly away from zero, and hence $\psi_y(x_0, \cdot)$ is strictly positive on $I(x_0)$.  From the boundary conditions, this implies that there exists some $y_{N-1}(x_0) \in I(x_0)$ such that 
$$ \psi(x_0, y_{N-1}(x_0)) = -p_{N-1}.$$  
Fix $y \in I_0(x_0)$, and differentiate the quantity $y + d - h(x,-\psi(x_0,y))$ in $y$ to find
\begin{align*} \partial_y \left[ y + d - h(x_0,-\psi(x_0,y)) \right] & =1+ h_p(x_0,-\psi(x_0,y)) \psi_y \\
& =1 - \frac{\psi_y}{F(x_0, -\psi(x_0,y))} = 0. \end{align*}

Evaluated on the free surface
$$ \left( y+d -h(x_0,-\psi(x_0,y)) \right)\Big|_{y = \eta(x_0)} = \eta(x_0) + d - h(x_0,0) = 0,$$
by the definition of $\eta$ and $\psi$.  Hence 
\be y = h(x,-\psi(x,y)) - d, \qquad \textrm{in  $\{ (x,y) \in \mathbb{R}^2 : y \in I(x)\}$,} \label{yhidentity} \ee
and, in particular,
\be y_{N-1}(x) = h(x,-\psi(x, y_{N-1}(x))) - d = h(x,p_1) - d =: \eta_{N-1}(x). \label{equivalence reconstruction eta_N-1} \ee
Thus we have reconstructed $\psi$ and the entire upper layer $\Omega_{N}$.   Notice that the periodicity of $h$ ensures that $\psi$ and $\eta_{N-1}$ are $2L$-periodic in $x$.  

Peano's theorem does not imply uniqueness of the solution $\psi(x_0,\cdot)$, so an additional argument is  required to guarantee that $\psi$ depends smoothly on $x$.  Fix $y_0$ so that $y_0 \in I(x)$ for $x$ in a sufficiently small neighborhood $\mathcal{U}$ of $x_0$ (which is permissible since $\eta$ is continuous).  Since $h_p > 0$, we may apply the Implicit Function Theorem to the equation 
$$ y_0 = h(x,p) - d, \qquad x \in \mathcal{U},$$
to obtain a unique ${C^{1,\alpha}}$-parameterization $(x,p(x))$ of all solutions near the point $(x_0, y_0)$.  In light of \eqref{yhidentity}, we see that $p(x) = - \psi(x,y_0)$.   Patching these solutions together using uniqueness gives a solution $\psi \in {C^{1,\alpha}(\Omega_N)}$ to 
\be \psi_y(x,y) = -F(x,-\psi(x,y)) = \frac{1}{h_p(x,-\psi(x,y))}. \label{psiyidentity} \ee 
Moreover, since $h_p$ is $2L$-periodic in $x$, the uniqueness of solutions implies that $\psi$ is $2L$-periodic in $x$, i.e. it is in ${C_{\textrm{per}}^{1,\alpha}(\Omega_N)}$.  Also, we see from \eqref{equivalence reconstruction eta_N-1} that $\eta_{N-1} \in {C_{\textrm{per}}^{1,\alpha}(\mathbb{R})}$. 

It remains to prove that $\psi$ solves Yih's equation \eqref{weakyih}, \eqref{weakbernoulli} in $\Omega_N$.  Note that Lemma \ref{chainrulelemma} guarantees that \eqref{psisurfacecond} and \eqref{psibedcond} are satisfied.    In \eqref{yhidentity}, the basic relationship between $\psi$ and $h$ was reestablished.  We will use this to find the change of variables formulas that allow us to transform from \eqref{weakheighteq} back to \eqref{weakyih}.  For instance, differentiating \eqref{yhidentity} in $x$ yields
$$ 0 = h_q(x,-\psi(x,y)) - h_p(x, -\psi(x,y)) \psi_x(x,y).$$
Since $h_p > 0$, we can rearrange terms to see 
\be \psi_x(x,y) = \frac{h_q(x,-\psi(x,y))}{h_p(x,-\psi(x,y))}.\label{psixidentity} \ee
By construction, $\psi \in {C_{\textrm{per}}^{1,\alpha}(\Omega_N)}$, and thus the right-hand side above is in $W_{\textrm{per}}^{1,r}(\Omega_N)$.  
On the other hand, in \eqref{psiyidentity},  the composition on the right-hand side is clearly in $W_{\textrm{per}}^{1,r}(\Omega_N)$, and hence $\psi \in W_{\textrm{per}}^{2,r}(\Omega_N)$.

Differentiating identities \eqref{psiyidentity}, \eqref{psixidentity}, we find
\begin{align*} \psi_{yy}(x,y) &= -\frac{h_{pp}(x,-\psi(x,y)) \psi_y(x,y)}{h_{p}(x, -\psi(x,y))^2} = \frac{h_{pp}(x,-\psi(x,y))}{h_p(x,-\psi(x,y))^3} = \left[\partial_p \left( -\frac{1}{2h_p^2} \right)\right] \bigg|_{(x,-\psi(x,y))} \\
\psi_{xx}(x,y) &= \frac{h_{qq}(x,-\psi(x,y)) - h_{qp}(x,-\psi(x,y))\psi_x(x,y)}{h_p(x,-\psi(x,y))} \\
& \qquad - \frac{h_q(x,-\psi(x,y)) \left[ h_{qp}(x,-\psi(x,y)) - h_{pp}(x,-\psi(x,y)) \psi_x(x,y) \right]}{h_p(x,-\psi(x,y))^2} \\
& = -\frac{h_{qq}(x,-\psi(x,y)) h_q(x,-\psi(x,y))}{h_p(x,-\psi(x,y))^2} + \frac{h_q(x,-\psi(x,y))^2h_pp(x,-\psi(x,y))}{h_p(x,-\psi(x,y))^3} \\
& \qquad \frac{h_{qq}(x,-\psi(x,y)) h_p(x,-\psi(x,y))^2 - h_q(x,-\psi(x,y)) h_{qp}(x,-\psi(x,y))}{h_p(x,-\psi(x,y))^2} \\
& = \left[ \partial_p \left(-\frac{h_q^2}{2h_p^2}\right) +\partial_q \left(\frac{h_q}{h_p} \right) \right] \bigg|_{(x,-\psi(x,y))}.  \end{align*}
Since $h$ solves \eqref{weakheighteq}, the computations above can be combined to obtain 
$$ \Delta \psi(x,y) = \left[ g(h-{d})\rho_p - B_p\right] \Big|_{(x,-\psi(x,y))} = gy\rho^\prime(-\psi(x,y)) - \beta(\psi(x,y)). $$
Here we are defining $\beta \in L^r[0,|p_0|]$ by \eqref{defB} and using \eqref{yhidentity} to equate $y$ and $h-{d}$ evaluated at $(x,-\psi(x,y))$.   Thus $\psi$ satisfies \eqref{weakyih}.

Finally, to show \eqref{weakbernoulli}, we note that by\eqref{psiyidentity} and \eqref{psixidentity}, 
$$ \frac{1+h_q(x,-\psi(x,y))^2}{h_p(x,-\psi(x,y))^2} = |\nabla \psi(x,y)|^2.$$
Thus \eqref{weakbernoulli} is an immediate consequence of \eqref{weakventtseleq}.

Up to this point we have proved that $\psi$ has the correct regularity and solves Yih's equation in the upper layer $\Omega_N$.  Repeating this procedure, we can reconstruct the lower layers $\Omega_1, \ldots, \Omega_{N-1}$, and extend $\psi$ so that $\psi \in C_{\textrm{per}}^0(\overline{\Omega})$ and $\psi|_{\Omega_i} \in W_{\textrm{per}}^{2,r}(\Omega_i)$, for each $i = 1, \ldots, N$.  The pressure condition \eqref{weakbernoulli} will, as above, be a consequence of \eqref{weaktransmissioncond}.  Lastly, observe that since $\psi \in C_{\textrm{per}}(\overline{\Omega})$, $\nabla \psi \in L_{\textrm{per}}^r(\Omega)$. Thus $\psi$ exhibits the required smoothness over the interior interfaces.   \end{proof}

\begin{lemma}[Equivalence of height equation and Ter-Krikorov] \label{equivalence ter-krikorov} Let $\rho$ be given with the regularity
\begin{align*}
 \rho &\in L^\infty([p_0, 0]) \cap W^{1,r}([p_0, p_1]) \cap \cdots \cap W^{1,r}([p_{N-1}, 0]) \end{align*}
 and such that \eqref{rhopositive}--\eqref{stablestratification} hold. Define $\mathring{h}$ according to \eqref{identity for limiting h}.  Then 
 $$  \mathring{h} \in W_{\mathrm{per}}^{1,r}([p_0, 0]) \cap W_{\mathrm{per}}^{2,r}([p_0, p_1]) \cap \cdots \cap W_{\mathrm{per}}^{2,r}([p_{N-1}, 0]). $$
 Moreover, the following statements are equivalent.  
\begin{itemize}
\item[(i)] There exists a solution $(\psi,\eta, \eta_1, \ldots, \eta_{N-1})$ to the localized stream function problem, as stated in Problem \ref{localized stream function prob}, for this choice of $\rho$. 

\item[(ii)] There exists a solution $h$ to the localized height equation problem, as stated in Problem \ref{localized height equation prob}, for this choice of  $\rho$.  

\item[(iii)] There exists a solution $w$ to the Ter-Krikorov problem, as stated in Problem \ref{ter-krikorov prob}, for the rescaled streamline density $\mathring{\rho}$ given by
$$ \mathring{\rho}(\zeta) := \rho(p(\zeta)),$$
where $\zeta \mapsto p(\zeta)$ is the inverse of $p \mapsto \mathring{h}(p) - d$. 
\end{itemize}
\end{lemma}
\begin{proof}  Given the regularity of $\rho$, the stated regularity of $\mathring{h}$ is obvious.  The equivalence of (i) and (ii) is implied by Lemma \ref{equivalence lemma 1}.  To show that (ii) and (iii) are equivalent is relatively straightforward, particularly compared to the previous lemma.  Moreover, this fact is implicit in the works of Turner \cite{turner1981internal,turner1984variational}, e.g.  We therefore omit the details.
\end{proof}

\noindent 
{\bf Acknowledgments.}  The authors would like to thank Miles Wheeler, Hongjie Dong, Stephen Hoffman, and Ariel Barton for helpful comments and enlightening conversations during the research leading to this work.  

%%%%%%%%%%%%%%%%%%%%%%%%%%%%%%%%%%%%%%%%%%%%%%%%%%%%%%
% References
%%%%%%%%%%%%%%%%%%%%%%%%%%%%%%%%%%%%%%%%%%%%%%%%%%%%%%
\bibliographystyle{siam}
\bibliography{projectdescription}
\end{document}